\documentclass[a4paper,10pt]{amsart}
\usepackage[arrow,curve,matrix]{xy}
\usepackage{graphicx}
\usepackage[height=23cm,width=17cm,twoside,centering]{geometry}
\usepackage{longtable}
\usepackage{arydshln}

\DeclareMathOperator{\End}{End}
\DeclareMathOperator{\GL}{GL}
\DeclareMathOperator{\Hom}{Hom}
\DeclareMathOperator{\per}{per}

\newcommand{\cC}{\mathcal{C}}
\newcommand{\cD}{\mathcal{D}}
\newcommand{\cI}{\mathcal{I}}
\newcommand{\cQ}{\mathcal{Q}}
\newcommand{\cS}{\mathcal{S}}
\newcommand{\bF}{\mathbb{F}}
\newcommand{\bQ}{\mathbb{Q}}
\newcommand{\bZ}{\mathbb{Z}}
\newcommand{\eps}{\varepsilon}
\newcommand{\gL}{\Lambda}
\newcommand{\wt}{\widetilde}

\theoremstyle{plain}
\newtheorem{theorem}{Theorem}[section]
\newtheorem{prop}[theorem]{Proposition}
\newtheorem{lemma}[theorem]{Lemma}
\newtheorem{cor}[theorem]{Corollary}
\newtheorem{conj}[theorem]{Conjecture}

\theoremstyle{definition}
\newtheorem{dfn}[theorem]{Definition}
\newtheorem{remark}[theorem]{Remark}
\newtheorem{example}[theorem]{Example}
\newtheorem{notat}[theorem]{Notation}
\newtheorem{quest}[theorem]{Question}
\newtheorem{alg}[theorem]{Algorithm}

\numberwithin{equation}{section}

\begin{document}

\title{Towards derived equivalence classification of the
cluster-tilted algebras of Dynkin type D}

\author{Janine Bastian}
\address{Janine Bastian \newline
Leibniz Universit\"at Hannover, Institut f\"ur Algebra, Zahlentheorie
und Diskrete Mathematik, \newline
Welfengarten 1, 30167 Hannover, Germany}
\email{bastian@math.uni-hannover.de}

\author{Thorsten Holm}
\address{Thorsten Holm \newline
Leibniz Universit\"at Hannover, Institut f\"ur Algebra, Zahlentheorie
und Diskrete Mathematik, \newline
Welfengarten 1, 30167 Hannover, Germany}
\email{holm@math.uni-hannover.de}
\urladdr{http://www.iazd.uni-hannover.de/\~{}tholm}

\author{Sefi Ladkani}
\address{Sefi Ladkani \newline
Institut des Hautes \'{E}tudes Scientifiques, Le Bois Marie,
35, route de Chartres, 91440 Bures-sur-Yvette, France}
\curraddr{Mathematisches Institut der Universit\"at Bonn, Endenicher
Allee 60, 53115 Bonn, Germany}
\email{sefil@math.uni-bonn.de}
\urladdr{http://www.math.uni-bonn.de/people/sefil/}



\thanks{{\em Acknowledgement. }This work has been carried out in
the framework of the research priority program SPP 1388
\emph{Representation Theory} of the Deutsche Forschungsgemeinschaft
(DFG). We gratefully acknowledge financial support through the grants
HO 1880/4-1 and LA 2732/1-1. S.\,Ladkani was also supported by a
European Postdoctoral Institute (EPDI) fellowship.}

\begin{abstract}
We provide a far reaching derived equivalence classification of the
cluster-tilted algebras of Dynkin type~D and suggest standard forms for
the derived equivalence classes. We believe that the classification is
complete, but some subtle questions remain open. We introduce another
notion of equivalence called \emph{good mutation} equivalence which is
slightly stronger than derived equivalence but is algorithmically more
tractable, and give a complete classification together with standard
forms.
\end{abstract}

\maketitle


\section*{Introduction}

Cluster categories have been introduced in~\cite{BMRRT} (see
also~\cite{CCS06a} for Dynkin type $A$) as a
representation-theoretic approach to Fomin and Zelevinsky's cluster
algebras without coefficients having skew-symmetric exchange
matrices (so that matrix mutation becomes the combinatorial recipe
of mutation of quivers). This approach allows to use deep algebraic
and representation-theoretic methods in the context of cluster
algebras. A crucial role is played by the cluster tilting objects in
the cluster category which model the clusters in the cluster
algebra. The endomorphism algebras of these cluster tilting objects
are called cluster-tilted algebras.

Cluster-tilted algebras are particularly well-understood if the
quiver underlying the cluster algebra, and hence the cluster
category, is of Dynkin type. Cluster-tilted algebras of Dynkin type
can be described as quivers with relations where the possible
quivers are precisely the quivers in the mutation class of the
Dynkin quiver, and the relations are uniquely determined by the
quiver in an explicit way~\cite{BMR_finite}. Moreover, the quivers
in the mutation classes of Dynkin quivers are explicitly known; for
type $A_n$ they can be found in~\cite{Buan-Vatne}, for type $D_n$
in~\cite{Vatne} and for type $E_{6,7,8}$ they can be enumerated
using a computer, for example by the Java
applet~\cite{Keller-software}.

However, despite knowing the cluster-tilted algebras of Dynkin type
as quivers with relations, many structural properties are not
understood yet. In particular, one would want to know when two
cluster-tilted algebras have equivalent derived categories. A
derived equivalence classification has been achieved so far for the
cluster-tilted algebras of Dynkin type $A_n$ by Buan and
Vatne~\cite{Buan-Vatne}, and for Dynkin type $E_{6,7,8}$ by the
authors~\cite{BHL09}. Moreover, a derived equivalence classification
has also been given by the first author for the cluster-tilted
algebras of extended Dynkin type $\tilde{A}_n$~\cite{Bastian}.

Cluster-tilted algebras of types $A_n$, $\tilde{A}_n$ and $D_n$ are
naturally associated to triangulations of marked surfaces. To an
(ideal) triangulation of a compact, connected, oriented Riemann
surface (possibly with boundary) with a set of marked points on it,
a quiver with potential has been associated in~\cite{Labardini},
linking the theory of cluster algebras associated to marked
surfaces~\cite{FST} with the theory of quivers with potentials and
their mutations initiated in~\cite{DWZ}. Any such triangulation
gives rise to an algebra by considering the Jacobian algebra of the
associated quiver with potential. In this way, the cluster-tilted
algebras of type $A_n$ arise from (triangulations of) a disc with
$n+3$ marked points on its boundary~\cite{CCS06a}, those of type
$\tilde{A}_n$ arise from an annulus with $n$ marked points on its
boundary, whereas those of type $D_n$ arise from a once-punctured
disc with $n$ marked points on its boundary~\cite{Schiffler}.

When the marked points lie entirely on the boundary (i.e.\ there are
no punctures), the associated Jacobian algebras are gentle and were
studied in~\cite{ABCP}. A derived equivalence classification of
these algebras has been presented by the third
author~\cite{Ladkani_gentle}, generalizing the aforementioned
classifications in types $A_n$ and $\tilde{A}_n$. On the other hand,
when the boundary of the surface is empty (so that all the marked
points are punctures), the corresponding Jacobian algebras are
symmetric~\cite{Ladkani_closed} and in a forthcoming paper by the
third author it will be shown that they are derived equivalent.

In the present paper we address the problem of derived equivalence
classification of the cluster-tilted algebras of Dynkin type $D_n$.
These algebras form the simplest instance of Jacobian algebras
arising from triangulations of a marked surface having punctures as
well as non-empty boundary. As one of our main results we obtain a
far reaching derived equivalence classification, see
Theorem~\ref{t:stdder}. This classification is complete for $D_n$
when $n \leq 14$, see Theorem~\ref{thm:D14}, but it will turn out to
be surprisingly subtle to distinguish certain of the cluster-tilted
algebras up to derived equivalence. Nevertheless we believe that our
classification is complete also for $n \geq 15$, see
Conjecture~\ref{conj:dereq} and the remark following it.

There are two natural approaches to address derived equivalence
classification problems of a given collection of algebras arising
from some combinatorial data. The \emph{top-down} approach is to
divide these algebras into equivalence classes according to some
invariants of derived equivalence, so that algebras belonging to
different classes are not derived equivalent. The \emph{bottom-up}
approach is to systematically construct, based on the combinatorial
data, tilting complexes yielding derived equivalences between pairs
of these algebras and then to arrange these algebras into groups
where any two algebras are related by a sequence of such derived
equivalences. To obtain a complete derived equivalence
classification one has to combine these approaches and hope that the
two resulting partitions of the entire collection of algebras
coincide.

For the bottom-up approach we use constructions that are based on
good mutations of quivers with potentials, which correspond to
particular kind of derived equivalences between their corresponding
Jacobian algebras. The notion of good mutation has been introduced
by the third author~\cite{Ladkani10} in relation with assessing the
derived equivalence of endomorphism algebras of neighboring
cluster-tilting objects in $2$-Calabi-Yau categories. We now explain
this notion in the more restrictive setup of cluster-tilted algebras
which is sufficient for the purposes of the current paper.

Since any two quivers in a mutation class are connected by a
sequence of mutations, it is natural to ask when a single mutation
of quivers is accompanied by derived equivalence of their
corresponding cluster-tilted algebras. The paper~\cite{Ladkani10}
presents a procedure to determine when two cluster-tilted algebras
whose quivers are related by a single mutation are also related by
Brenner-Butler (co-)tilting, which is a particular kind of derived
equivalence. We call such quiver mutation \emph{good mutation}. In
other words, a mutation at some vertex is good if the corresponding
Brenner-Butler tilting module is defined and moreover its
endomorphism algebra is isomorphic to the cluster-tilted algebra of
the mutated quiver. Obviously, the cluster-tilted algebras of
quivers connected by a sequence of good mutations (i.e.\ \emph{good
mutation equivalent}) are derived equivalent. The explicit knowledge
of the relations for cluster-tilted algebras of Dynkin type together
with the procedure in~\cite{Ladkani10} imply that for these algebras
there is an algorithm to decide if a mutation is good or not. By
utilizing this algorithmic approach we achieve a complete good
mutation equivalence classification of the cluster-tilted algebras
of Dynkin type $D_n$, see Theorem~\ref{t:stdgood}, which is another
main result of the present paper.

It turns out that in many interesting cases, including the
cluster-tilted algebras of Dynkin types, a single mutation at one
vertex is good if and only if the corresponding algebras are derived
equivalent, see Corollary~\ref{c:goodmut}. Hence the initial focus
on a particular kind of derived equivalence, motivated by
K-theoretic considerations, is actually not restrictive.

The top-down and the bottom-up approaches have been successfully
combined to give complete derived equivalence classifications for
the Jacobian algebras arising from a marked surface without
punctures or a marked surface without boundary, as well as for the
cluster-tilted algebras of Dynkin type $E_{6,7,8}$. Moreover, for
these algebras the notions of derived equivalence and good mutation
equivalence coincide (see for example Theorem~\ref{thm-good-typeA}
below for type $A_n$ and~\cite{BHL09} for type $E_{6,7,8}$), so that
any two derived equivalent algebras can be connected by a sequence
of good mutations.

However, for the cluster tilted algebras of Dynkin type $D_n$,
derived equivalence is strictly weaker than good mutation
equivalence which underlines and explains why the situation in this
case is much more complicated. The fact that there are
cluster-tilted algebras of Dynkin type $D_n$ which are derived
equivalent without being connected by a sequence of good mutations
occurs already for types $D_6$ and $D_8$, see
Examples~\ref{ex:selfinj} and~\ref{ex:D8}. Although we have been
able to find further systematic derived equivalences including what
we call `good double mutations', one cannot be sure that these are
all. Indeed, there are arbitrarily large sets of cluster-tilted
algebras of type $D$ such that all computable derived invariants
available to us coincide for all the algebras within such a set, but
nevertheless we could not determine whether any two of these
algebras are derived equivalent or not, see Section~\ref{sec:D15}.

\smallskip

The paper is organized as follows. In Section~\ref{sec:prelim} we
collect some preliminaries about invariants of derived equivalence,
mutations of algebras and fundamental properties of cluster-tilted
algebras, particularly of Dynkin types $A$ and $D$. These are needed
for the statements of our main results, which are given in
Section~\ref{sec:mainresults}. In particular, Theorem~\ref{t:stdder}
gives the far reaching derived equivalence classification of
cluster-tilted algebras of type $D$, and Theorem~\ref{t:stdgood} the
complete classification up to good mutation equivalence. Open
questions and some examples are also given in that section. In
Section~\ref{sec:goodmut} we determine all the good mutations for
cluster-tilted algebras of Dynkin types $A$ and $D$, whereas in
Section~\ref{sec:further-derived} we present further derived
equivalences between cluster-tilted algebras of type $D$ which are
not given by good mutations. Building on these results we provide in
Section~\ref{sec:algstdform}, which is purely combinatorial,
standard forms for derived equivalence as well as ones for good
mutation equivalence of cluster-tilted algebras of type $D$, thus
proving Theorem~\ref{t:stdder} and Theorem~\ref{t:stdgood}. We also
describe an explicit algorithm which decides on good mutation
equivalence. Finally, the appendix contains the proof of the
formulae for the determinants of the Cartan matrices of
cluster-tilted algebras of type $D$, as given in
Theorem~\ref{thm-det-typeD}. This invariant is used in the paper to
distinguish some cluster-tilted algebras up to derived equivalence.

\section{Preliminaries}
\label{sec:prelim}

\subsection{Derived equivalences and tilting complexes}

Throughout this paper let $K$ be an algebraically closed field. All
algebras are assumed to be finite-dimensional $K$-algebras. For a
$K$-algebra $A$, we denote the bounded derived category of right
$A$-modules by $\cD^b(A)$. Two algebras $A$ and $B$ are called
\emph{derived equivalent} if $\cD^b(A)$ and $\cD^b(B)$ are
equivalent as triangulated categories.

A famous theorem of Rickard~\cite{Rickard} characterizes derived
equivalence in terms of the so-called tilting complexes, which we now
recall. Denote by $\per A$ the full triangulated subcategory of
$\cD^b(A)$ consisting of the \emph{perfect} complexes of $A$-modules,
that is, complexes (quasi-isomorphic) to bounded complexes of finitely
generated projective $A$-modules.

\begin{dfn} \label{tilting-complex}
A \emph{tilting complex} $T$ over $A$ is a complex $T \in \per A$ with
the following two properties:
\begin{enumerate}
\renewcommand{\theenumi}{\roman{enumi}}
\item
It is \emph{exceptional}, i.e.\ $\Hom_{\cD^b(A)}(T,T[i])=0$ for all
$i\neq 0$, where $[1]$ denotes the shift functor in $\cD^b(A)$;
\item
It is a \emph{compact generator}, that is, the minimal triangulated
subcategory of $\per A$ containing $T$ and closed under taking direct
summands, equals $\per A$.
\end{enumerate}
\end{dfn}

\begin{theorem}[Rickard \cite{Rickard}] \label{t:Rickard}
Two algebras $A$ and $B$ are derived equivalent if and only if there
exists a tilting complex $T$ over $A$ such that $\End_{\cD^b(A)}(T)
\simeq B$.
\end{theorem}

Although Rickard's theorem gives us a criterion for derived
equivalence, it does not give a decision process nor a constructive
method to produce tilting complexes. Thus, given two algebras $A$ and
$B$ in concrete form, it is sometimes still unknown whether they are
derived equivalent or not, as we do not know how to construct a
suitable tilting complex or to prove the non-existence of such, see
Sections~\ref{sec:opp} and~\ref{sec:D15} for some concrete examples.

\subsection{Invariants of derived equivalence}
\label{sec:invariants}

Let $P_1, \dots, P_n$ be a complete collection of pairwise
non-isomorphic indecomposable projective $A$-modules
(finite-dimensional over $K$). The \emph{Cartan matrix} of $A$ is then
the $n \times n$ matrix $C_A$ defined by $(C_A)_{ij} = \dim_K
\Hom_A(P_j, P_i)$. An important invariant of derived equivalence is
given by the following well known proposition. For a proof see the
\emph{proof} of Proposition~1.5 in~\cite{Bocian}, and
also~\cite[Prop.~2.6]{BHL09}.

\begin{prop}
Let $A$ and $B$ be two finite-dimensional, derived equivalent algebras.
Then the matrices $C_A$ and $C_B$ represent equivalent bilinear forms
over $\bZ$, that is, there exists $P \in \GL_n(\bZ)$ such that $P C_A
P^T = C_B$, where $n$ denotes the number of
indecomposable projective modules of $A$ and $B$ (up to isomorphism).
\end{prop}

In general, to decide whether two integral bilinear forms are
equivalent is a very subtle arithmetical problem. Therefore, it is
useful to introduce somewhat weaker invariants that are computationally
easier to handle. In order to do this, assume further that $C_A$ is
invertible over $\bQ$. In this case one can consider the rational
matrix $S_A = C_A C_A^{-T}$ (here $C_A^{-T}$ denotes the inverse of the
transpose of $C_A$), known in the theory of non-symmetric bilinear
forms as the \emph{asymmetry} of $C_A$.

\begin{prop} \label{prop-asymmetry}
Let $A$ and $B$ be two finite-dimensional, derived equivalent algebras
with invertible (over $\bQ$) Cartan matrices. Then we have the
following assertions, each implied by the preceding one:
\begin{enumerate}
\renewcommand{\theenumi}{\alph{enumi}}
\item
There exists $P \in \GL_n(\bZ)$ such that $P C_A P^T = C_B$.

\item
There exists $P \in \GL_n(\bZ)$ such that $P S_A P^{-1} = S_B$.

\item
There exists $P \in \GL_n(\bQ)$ such that $P S_A P^{-1} = S_B$.

\item
The matrices $S_A$ and $S_B$ have the same characteristic polynomial.
\end{enumerate}
\end{prop}

For proofs and discussion, see for example \cite[Section~3.3]{Ladkani}.
Since the determinant of an integral bilinear form is also invariant
under equivalence, we obtain the following discrete invariant of
derived equivalence.

\begin{dfn}
For an algebra $A$ with invertible Cartan matrix $C_A$ over $\bQ$, we
define its \emph{associated polynomial} as $(\det C_A) \cdot
\chi_{S_A}(x)$, where $\chi_{S_A}(x)$ is the characteristic polynomial
of the asymmetry matrix $S_A = C_A C_A^{-T}$.
\end{dfn}

\begin{remark}
The matrix $S_A$ (or better, minus its transpose $-C_A^{-1} C_A^T$) is
related to the \emph{Coxeter transformation} which has been widely
studied in the case when $A$ has finite global dimension (so that $C_A$
is invertible over $\bZ$), see~\cite{Lenzing99}.
It is the $K$-theoretic shadow of the Serre
functor and the related Auslander-Reiten translation in the derived
category. The characteristic polynomial is then known as the
\emph{Coxeter polynomial} of the algebra.
\end{remark}

\begin{remark}
In general, $S_A$ might have non-integral entries. However, when the
algebra $A$ is \emph{Gorenstein}, the matrix $S_A$ is integral, which
is an incarnation of the fact that the injective modules have finite
projective resolutions. By a result of Keller and Reiten
\cite{Keller-Reiten}, this is the case for cluster-tilted algebras.
\end{remark}

\subsection{Mutations of algebras}
We recall the notion of mutations of algebras from~\cite{Ladkani10}.
These are local operations on an algebra $A$ producing new algebras derived
equivalent to $A$.

Let $A = KQ/I$ be an algebra given as a quiver with relations. For any
vertex $i$ of $Q$, there is a trivial path $e_i$ of length 0; the
corresponding indecomposable projective module $P_i=e_i A$ is spanned by the
images of the paths starting at $i$. Thus an arrow $i
\xrightarrow{\alpha} j$ gives rise to a map $P_j \to P_i$ given by left
multiplication with $\alpha$.

Let $k$ be a vertex of $Q$ without loops.
Consider the following two complexes of projective $A$-modules
\begin{align*}
T^-_k(A) = \bigl( P_k \xrightarrow{f} \bigoplus_{j \to k} P_j \bigr) \oplus
       \bigl( \bigoplus_{i \neq k} P_i \bigr) &, &
T^+_k(A) = \bigl( \bigoplus_{k \to j} P_j \xrightarrow{g} P_k \bigr) \oplus
       \bigl( \bigoplus_{i \neq k} P_i \bigr)
\end{align*}
where the map $f$ is induced by all the maps $P_k \to P_j$ corresponding to
the arrows $j \to k$ ending at $k$,
the map $g$ is induced by the maps $P_j \to P_k$ corresponding to the
arrows $k \to j$ starting at $k$, the term $P_k$ lies in degree $-1$ in
$T^-_k(A)$ and in degree $1$ in $T^+_k(A)$, and all other terms are in
degree $0$.

\begin{dfn}
Let $A$ be an algebra given as a quiver with relations and $k$ a vertex
without loops.
\begin{enumerate}
\renewcommand{\theenumi}{\alph{enumi}}
\item
We say that the negative mutation of $A$ at $k$ is \emph{defined} if
$T^-_k(A)$ is a tilting complex over $A$. In this case, we call the
algebra $\mu^-_k(A) = \End_{\cD^b(A)} T^-_k(A)$ the \emph{negative
mutation} of $A$ at the vertex $k$.

\item
We say that the positive mutation of $A$ at $k$ is \emph{defined} if
$T^+_k(A)$ is a tilting complex over $A$. In this case, we call the
algebra $\mu^+_k(A) = \End_{\cD^b(A)} T^+_k(A)$ the \emph{positive
mutation} of $A$ at the vertex $k$.
\end{enumerate}
\end{dfn}

\begin{remark}
By Rickard's Theorem~\ref{t:Rickard}, the negative and the positive mutations
of an algebra $A$ at a vertex, when defined, are always derived equivalent
to $A$.
\end{remark}

There is a combinatorial criterion to determine whether a mutation at a
vertex is defined, see~\cite[Prop.~2.3]{Ladkani10}. Since the algebras
we will be dealing with in this paper are schurian, we state here the
criterion only for this case, as it takes a particularly simple form.
Recall that an algebra is \emph{schurian} if the entries of its Cartan
matrix are only $0$ or $1$.

\begin{prop} \label{p:critmut}
Let $A$ be a schurian algebra.
\begin{enumerate}
\renewcommand{\theenumi}{\alph{enumi}}
\item
The negative mutation $\mu^-_k(A)$ is defined if and only if for any
non-zero path $k \rightsquigarrow i$ starting at $k$ and ending at some
vertex $i$, there exists an arrow $j \to k$ such that the composition
$j \to k \rightsquigarrow i$ is non-zero.

\item
The positive mutation $\mu^+_k(A)$ is defined if and only if for any
non-zero path $i \rightsquigarrow k$ starting at some vertex $i$ and
ending at $k$, there exists an arrow $k \to j$ such that the
composition $i \rightsquigarrow k \to j$ is non-zero.
\end{enumerate}
\end{prop}

\begin{remark}
It follows from~\cite[Remark~2.10]{Ladkani10}, that in many cases, and
in particular when $A$ is schurian, the negative mutation of $A$ at $k$
is defined if and only if one can associate with $k$ the corresponding
Brenner-Butler tilting module. Moreover, in this case, $T^-_k(A)$ is
isomorphic in $\cD^b(A)$ to that Brenner-Butler tilting module.
\end{remark}

\subsection{Cluster-tilted algebras}

In this section we assume that all quivers are without loops and
$2$-cycles. Given such a quiver $Q$ and a vertex $k$, we denote by
$\mu_k(Q)$ the Fomin-Zelevinsky quiver mutation~\cite{FominZelevinsky02}
of $Q$ at $k$.
Two quivers are called \emph{mutation equivalent} if one can be reached
from the other by a finite sequence of quiver mutations. The
\emph{mutation class} of a quiver $Q$ is the set of all quivers which
are mutation equivalent to $Q$.

For a quiver $Q'$ without oriented cycles (i.e.\ an {\em acyclic} quiver),
the corresponding cluster
category $\cC_{Q'}$ was introduced in~\cite{BMRRT}. A \emph{cluster-tilted
algebra} of \emph{type $Q'$} is an endomorphism algebra of a cluster-tilting
object in $\cC_{Q'}$, see~\cite{BMR}.
It is known by~\cite{BMR} that for any quiver $Q$ mutation equivalent to
$Q'$, there is a cluster-tilted algebra whose quiver is $Q$. Moreover,
by~\cite{BIRSm}, it is unique up to isomorphism.
Hence, there is a bijection between the quivers in the mutation class of an
acyclic quiver $Q'$ and the isomorphism classes of cluster-tilted algebras
of type $Q'$. This justifies the following notation.

\begin{notat}
Throughout the paper, for a quiver $Q$ which is mutation equivalent to an
acyclic quiver,
we denote by $\gL_Q$ the corresponding cluster-tilted algebra and by $C_Q$
its Cartan matrix $C_{\gL_Q}$.
\end{notat}

When $Q'$ is a Dynkin quiver of types $A$, $D$ or $E$, the
corresponding cluster-tilted algebras are said to be of Dynkin type.
These algebras have been investigated in~\cite{BMR_finite}, where it is
shown that they are schurian and moreover they can be defined by using
only zero and commutativity relations that can be extracted from their
quivers in an algorithmic way.

\subsection{Good quiver mutations}

For cluster-tilted algebras of Dynkin type, the statement of
Theorem~5.3 in~\cite{Ladkani10}, linking more generally mutation of
cluster-tilting objects in 2-Calabi-Yau categories with mutations of
their endomorphism algebras, takes the following form.

\begin{prop} \label{p:goodmut}
Let $Q$ be mutation equivalent to a Dynkin quiver and let $k$ be a
vertex of $Q$.
\begin{enumerate}
\renewcommand{\theenumi}{\alph{enumi}}
\item
$\gL_{\mu_k(Q)} \simeq \mu^-_k(\gL_Q)$ if and only if the two algebra
mutations $\mu^-_k(\gL_Q)$ and $\mu^+_k(\gL_{\mu_k(Q)})$ are defined.

\item
$\gL_{\mu_k(Q)} \simeq \mu^+_k(\gL_Q)$ if and only if the two algebra
mutations $\mu^+_k(\gL_Q)$ and $\mu^-_k(\gL_{\mu_k(Q)})$ are defined.
\end{enumerate}
\end{prop}

This motivates the following definition.

\begin{dfn}
When (at least) one of the conditions in the proposition holds, we say
that the quiver mutation of $Q$ at $k$, is \emph{good}, since it
implies the derived equivalence of the corresponding cluster-tilted
algebras $\gL_Q$ and $\gL_{\mu_k(Q)}$. When none of the conditions in
the proposition hold, we say that the quiver mutation is \emph{bad}.
\end{dfn}

\begin{remark}
In view of Propositions~\ref{p:critmut} and~\ref{p:goodmut}, there is an
algorithm which decides, given a quiver which is mutation equivalent to a
Dynkin quiver, whether a mutation at a vertex is good or not.
\end{remark}

Whereas in Dynkin types $A$ and $E$, the quivers of any two derived
equivalent cluster-tilted algebras are connected by a sequence of good
mutations~\cite{BHL09}, this is no longer the case in type $D$.
Therefore, we need also to consider mutations of algebras going beyond
the family of cluster-tilted algebras (which is not closed under
derived equivalence).

\begin{dfn}
Let $Q$ and $Q'$ be quivers with vertices $k$ and $k'$ such that
$\mu_k(Q) = \mu_{k'}(Q')$. We call the sequence of the two mutations
from $Q$ to $Q'$ (first at $k$ and then at $k'$) a \emph{good double
mutation} if both algebra mutations $\mu^-_k(\gL_Q)$ and
$\mu^+_{k'}(\gL_{Q'})$ are defined and moreover, they are isomorphic to
each other.
\end{dfn}

By definition, for quivers $Q$ and $Q'$ related by a good double mutation,
the cluster-tilted algebras $\gL_Q$ and $\gL_{Q'}$ are derived equivalent.
Note, however, that we do not require the intermediate algebra
$\mu^-_k(\gL_Q) \simeq \mu^+_{k'}(\gL_{Q'})$ to be a cluster-tilted algebra.

\subsection{Cluster-tilted algebras of Dynkin types $A$ and $D$}
\label{sec:ctaAD}

In this section we recall the explicit description of cluster-tilted
algebras of Dynkin types $A$ and $D$, which are our main objects of study, as
quivers with relations.

Recall that the quiver $A_n$ is the following directed graph on $n
\geq 1$ vertices
\[
\xymatrix{ {\bullet_1} \ar[r] & {\bullet_2} \ar[r] & {\dots} \ar[r] &
{\bullet_n}}.
\]
The quivers which are mutation equivalent to $A_n$ have been explicitly
determined in~\cite{Buan-Vatne}. They can be characterized as follows.

\begin{dfn}
The \emph{neighborhood} of a vertex $x$ in a quiver $Q$ is the full
subquiver of $Q$ on the subset of vertices consisting of $x$ and the
vertices which are targets of arrows starting at $x$ or sources of
arrows ending at $x$.
\end{dfn}

\begin{prop}
Let $n \geq 2$. A quiver is mutation equivalent to $A_n$ if and only if
it has $n$ vertices, the neighborhood of each vertex is one of the nine
depicted in Figure~\ref{fig:ngbrA}, and there are no cycles in its
underlying graph apart from those induced by oriented cycles contained
in neighborhoods of vertices.
\end{prop}

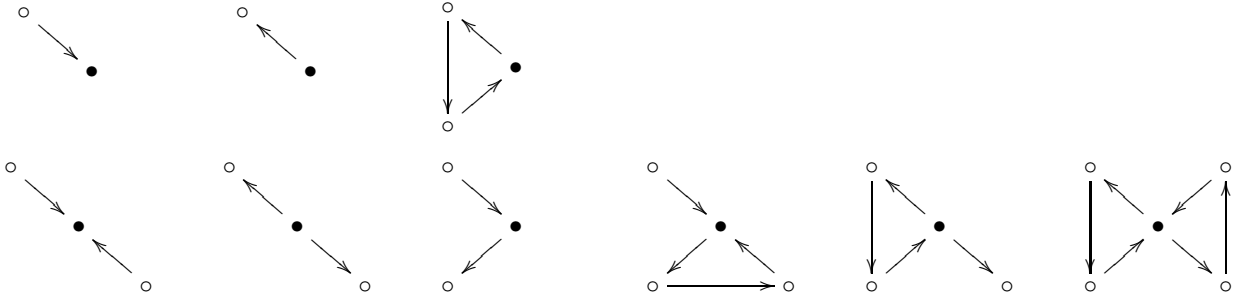
\begin{figure}
\begin{align*}
\begin{array}{c}
\xymatrix@R=1pc@C=1.2pc{
{\circ} \ar[dr] \\
& {\bullet} \\
& &
}
\end{array}
&&
\begin{array}{c}
\xymatrix@R=1pc@C=1.2pc{
{\circ} \\
& {\bullet} \ar[ul] \\
& &
}
\end{array}
&&
\begin{array}{c}
\xymatrix@R=1pc@C=1.2pc{
{\circ} \ar[dd] \\
& {\bullet} \ar[ul] \\
{\circ} \ar[ur] & &
}
\end{array}
\\
\begin{array}{c}
\xymatrix@R=1pc@C=1.2pc{
{\circ} \ar[dr] \\
& {\bullet} \\
& & {\circ} \ar[ul]
}
\end{array}
& &
\begin{array}{c}
\xymatrix@R=1pc@C=1.2pc{
{\circ} \\
& {\bullet} \ar[ul] \ar[dr] \\
& & {\circ}
}
\end{array}
& &
\begin{array}{c}
\xymatrix@R=1pc@C=1.2pc{
{\circ} \ar[dr] \\
& {\bullet} \ar[dl] \\
{\circ} & &
}
\end{array}
&&
\begin{array}{c}
\xymatrix@R=1pc@C=1.2pc{
{\circ} \ar[dr] \\
& {\bullet} \ar[dl] \\
{\circ} \ar[rr] & & {\circ} \ar[ul]
}
\end{array}
&&
\begin{array}{c}
\xymatrix@R=1pc@C=1.2pc{
{\circ} \ar[dd] \\
& {\bullet} \ar[ul] \ar[dr] \\
{\circ} \ar[ur] & & {\circ}
}
\end{array}
&&
\begin{array}{c}
\xymatrix@R=1pc@C=1.2pc{
{\circ} \ar[dd] & & {\circ} \ar[dl] \\
& {\bullet} \ar[ul] \ar[dr] \\
{\circ} \ar[ur] & & {\circ} \ar[uu]
}
\end{array}
\end{align*}
\caption{The 9 possible neighborhoods of a vertex $\bullet$ in a quiver
which is mutation equivalent to $A_n$, $n \geq 2$. The three at the top
row are the possible neighborhoods of a \emph{root} in a rooted quiver
of type $A$.} \label{fig:ngbrA}
\end{figure}

\begin{dfn}
Let $Q$ be a quiver mutation equivalent to $A_n$. A \emph{triangle} is
an oriented 3-cycle in $Q$, and a \emph{line} is an arrow in $Q$ which
is not part of a triangle. We denote by $s(Q)$ and $t(Q)$ the number of
lines and triangles in $Q$, respectively.
\end{dfn}

\begin{remark}
We have $n = 1 + s(Q) + 2t(Q)$.
\end{remark}

\begin{remark} \label{rem:relctaA}
Given a quiver $Q$ mutation equivalent to $A_n$, the \emph{relations}
defining the corresponding cluster-tilted algebra $\gL_Q$ (which has
$Q$ as its quiver) are obtained as follows~\cite{BMR_finite,CCS06a,CCS06b};
any triangle
\[
\xymatrix@=0.5pc{
& {\bullet} \ar[ddr]^{\beta} \\ \\
{\bullet} \ar[uur]^{\alpha} & & {\bullet} \ar[ll]^{\gamma}
}
\]
in $Q$ gives rise to three zero relations $\alpha \beta$, $\beta
\gamma$, $\gamma \alpha$, and there are no other relations.
\end{remark}

Recall that the quiver $D_n$ is the following quiver
\[
\xymatrix@=1pc{
{\bullet_1} \ar[rd] \\
& {\bullet_3} \ar[r] & {\dots} \ar[r] & {\bullet_n} \\
{\bullet_2} \ar[ru]
}
\]
on $n \geq 4$ vertices.  We now recall the description by
Vatne~\cite{Vatne} of the quivers which are mutation equivalent to
$D_n$, and the relations defining the corresponding cluster-tilted
algebras following~\cite{BMR_finite}. It will be most convenient to
use the language of gluing of rooted quivers.

\begin{dfn}
A \emph{rooted quiver of type $A$} is a pair $(Q,v)$ where $Q$ is a
quiver which is mutation equivalent to $A_n$ for some $n \geq 1$, and
$v$ is a vertex of $Q$ (the \emph{root}) whose neighborhood is one of
the three appearing in the first row of Figure~\ref{fig:ngbrA} if $n
\geq 2$.
\end{dfn}

By abuse of notation, we shall sometimes refer to such a rooted quiver
$(Q,v)$ just by $Q$.

\begin{dfn}
Let $Q_0$ be a quiver, called a \emph{skeleton}, and let $c_1, c_2,
\dots, c_k$ be $k \geq 0$ distinct vertices of $Q_0$. The \emph{gluing}
of $k$ rooted quivers of type $A$, say $(Q_1, v_1), (Q_2, v_2), \dots,
(Q_k, v_k)$, to $Q_0$ at the vertices $c_1, \dots, c_k$ is defined as
the quiver obtained from the disjoint union $Q_0 \sqcup Q_1 \sqcup
\dots \sqcup Q_k$ by identifying each vertex $c_i$ with the
corresponding root $v_i$, for $1 \leq i \leq k$.
\end{dfn}

\begin{remark} \label{rem:relglue}
Given relations (i.e.\ linear combinations of parallel paths) on the
skeleton $Q_0$, they induce relations on the gluing, namely by taking
the union of all the relations on $Q_0, Q_1, \dots, Q_k$, where the
relations on the rooted quivers of type $A$ are those stated in
Remark~\ref{rem:relctaA}.
\end{remark}

A cluster-tilted algebra of Dynkin type $D$ belongs to one of the
following four families, which are called \emph{types} and are defined
as gluing of rooted quivers of type $A$ to certain skeleta. Note that
in view of Remark~\ref{rem:relglue}, it is enough to specify the
relations on the skeleton. For each type, we define \emph{parameters}
which will be useful in the sequel when referring to the cluster-tilted
algebras of that type.

\subsubsection*{\underline{Type I}}
The gluing of a rooted quiver $Q'$ of type $A$ at the vertex $c$ of one
of the three skeleta
\begin{align*}
\xymatrix@=1.2pc{
{\bullet_a} \ar[dr] \\
& {\bullet_c} \ar[dl] \\
{\bullet_b}
}
& &
\xymatrix@=1.2pc{
{\bullet_a} \ar[dr] \\
& {\bullet_c}  \\
{\bullet_b} \ar[ur]
}
& &
\xymatrix@=1.2pc{
{\bullet_a} \\
& {\bullet_c} \ar[dl] \ar[ul] \\
{\bullet_b}
}
\end{align*}
as in the following picture:
\begin{center}
\includegraphics[scale=0.75]{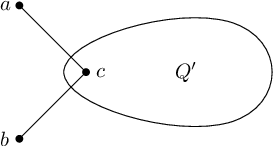}
\end{center}

The \emph{parameters} are $\bigl( s(Q'), t(Q') \bigr)$.

\subsubsection*{\underline{Type II}}
The gluing of two rooted quivers $Q'$ and $Q''$ of type $A$ at the
vertices $c'$ and $c''$, respectively, of the following skeleton
\[
\xymatrix@=1.5pc{
& {\bullet_b} \ar[dl]_\beta \\
{\bullet_{c''}} \ar[rr]^\eps & & {\bullet_{c'}} \ar[ul]_\alpha \ar[dl]^\gamma \\
& {\bullet_a} \ar[ul]^\delta}
\]
with the commutativity relation $\alpha \beta - \gamma \delta$ and the
zero relations $\eps \alpha$, $\eps \gamma$, $\beta \eps$, $\delta \eps$
as in the following picture:
\begin{center}
\includegraphics[scale=0.75]{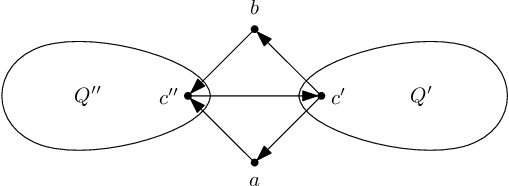}
\end{center}

The \emph{parameters} are $\bigl( s(Q'), t(Q'), s(Q''), t(Q'') \bigr)$.

\subsubsection*{\underline{Type III}}

The gluing of two rooted quivers $Q'$ and $Q''$ of type $A$ at the
vertices $c'$ and $c''$, respectively, of the following skeleton
\[
\xymatrix@=1.5pc{
& {\bullet_b} \ar[dl]_\beta \\
{\bullet_{c''}} \ar[dr]_\gamma & & {\bullet_{c'}} \ar[ul]_\alpha \\
& {\bullet_a} \ar[ur]_\delta}
\]
with the four zero relations $\alpha \beta \gamma$, $\beta \gamma \delta$,
$\gamma \delta \alpha$, $\delta \alpha \beta$, as in the following
picture:
\begin{center}
\includegraphics[scale=0.75]{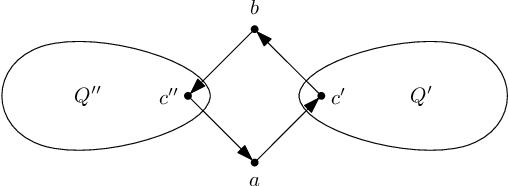}
\end{center}

As in Type II, the \emph{parameters} are $\bigl( s(Q'), t(Q'), s(Q''),
t(Q'') \bigr)$.

\subsubsection*{\underline{Type IV}}

The gluing of $r \geq 0$ rooted quivers $Q^{(1)}, \dots, Q^{(r)}$ of
type $A$ at the vertices $c_1, \dots, c_r$ of a skeleton
$Q(m, \{i_1, \dots, i_r\})$ defined below, see Figure~\ref{fig:ctaDIV}.

\begin{dfn}
Given integers $m \geq 3$, $r \geq 0$ and an increasing sequence $1
\leq i_1 < i_2 < \dots < i_r \leq m$, we define the following quiver
$Q(m, \{i_1, \dots, i_r\})$ with relations.
\begin{enumerate}
\renewcommand{\theenumi}{\alph{enumi}}
\item
$Q(m, \{i_1, \dots, i_r\})$ has $m+r$ vertices, labeled $1, 2, \dots,
m$ together with $c_1, c_2, \dots, c_r$, and its arrows are
\[
\bigl\{ i \to (i+1) \bigr\}_{1 \leq i \leq m} \cup \bigl\{ c_j \to i_j,
(i_j+1) \to c_j \bigr\}_{1 \leq j \leq r},
\]
where $i+1$ is considered modulo $m$, i.e.\ $1$, if $i = m$.

The full subquiver on the vertices $1, 2, \dots, m$ is thus an oriented
cycle of length $m$, called the \emph{central cycle}, and for every $1
\leq j \leq r$, the full subquiver on the vertices $i_j, i_j+1, c_j$ is
an oriented 3-cycle, called a \emph{spike}.

\item
The \emph{relations} on $Q(m, \{i_1, \dots, i_r\})$ are as follows:
\begin{itemize}
\item
The paths $i_j, i_j+1, c_j$ and $c_j, i_j, i_j+1$ are zero for all $1
\leq j \leq r$;

\item
For any $1 \leq j \leq r$, the path $i_j+1, c_j, i_j$ equals the path
$i_j+1, \dots, 1, \dots, i_j$ of length $m-1$ along the central cycle;

\item
For any $i \not \in \{i_1, \dots, i_r\}$, the path $i+1, \dots, 1,
\dots, i$ of length $m-1$ along the central cycle is zero.
\end{itemize}
\end{enumerate}
\end{dfn}

\begin{figure}[h]
\begin{center}
\includegraphics[scale=0.75]{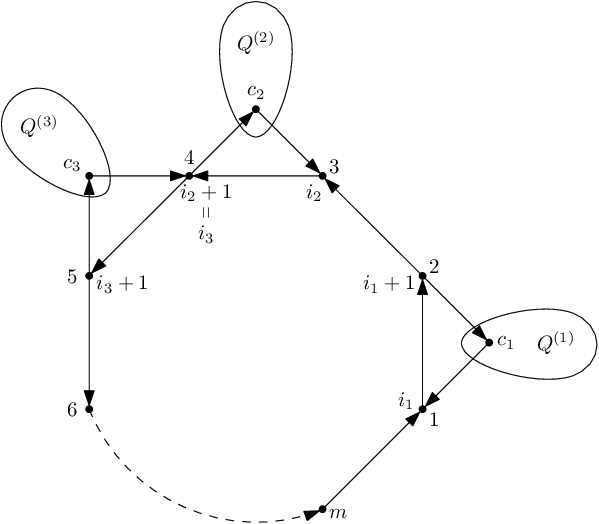}
\end{center}
\caption{A quiver of a cluster-tilted algebra of Type IV.}
\label{fig:ctaDIV}
\end{figure}

The \emph{parameters} are encoded as follows. If $r=0$, that is, there
are no spikes hence no attached rooted quivers of type $A$, the quiver
is just an oriented cycle, thus parameterized by its length $m \geq 3$.
In all other cases, due to rotational symmetry, we define the
\emph{distances} $d_1, d_2, \dots, d_r$ by
\[
d_1 = i_2 - i_1, d_2 = i_3 - i_2, \dots, d_r = i_1+m-i_r
\]
so that $m = d_1 + d_2 + \dots + d_r$, and encode the cluster-tilted
algebra by the sequence of triples
\begin{equation} \label{e:parDIV}
\bigl( (d_1, s_1, t_1), (d_2, s_2, t_2), \dots, (d_r, s_r, t_r) \bigr)
\end{equation}
where $s_j = s(Q^{(j)})$, $t_j = t(Q^{(j)})$ are the numbers of lines
and triangles of the rooted quiver $Q^{(j)}$ of type $A$ glued at the
vertex $c_j$ of the $j$-th spike.

\begin{remark}
Note that the cluster-tilted algebras in Type III can be viewed as a
degenerate version of Type IV, namely corresponding to the skeleton
$Q(2, \{1,1\})$ with central cycle of length $2$ (hence it is ``invisible'')
with all spikes present. It turns out that this point of view is
consistent with the constructions of good mutations and double
mutations as well as with the determinant computations presented later
in this paper. However, for simplicity, the proofs that we give for
Type III will not rely on this observation.
\end{remark}

\section{Main results}
\label{sec:mainresults}

In this section we describe the main results of the paper.

\subsection{Standard forms for derived equivalence}

We start by providing standard forms for derived equivalence. Since
rooted quivers of Dynkin type $A$ are important building blocks of the
quivers of cluster-tilted algebras of type $D$, we recall the results
on derived equivalence classification of cluster-tilted algebras of
type $A$, originally due to Buan and Vatne~\cite{Buan-Vatne}.

\begin{dfn}
Let $Q$ be a quiver of a cluster-tilted algebra of type $A$. The
\emph{standard form} of $Q$ is the following quiver consisting of
$s(Q)$ lines and $t(Q)$ triangles arranged as follows:
\begin{equation} \label{e:stdA}
\xymatrix@=0.3pc{
&& && && & {\bullet} \ar[ddl] && && {\bullet} \ar[ddl]
\\ \\
{\bullet_v} \ar[rr] & & {\bullet} \ar[rr] & & {\ldots} \ar[rr] & &
{\bullet} \ar[rr]  & & {\bullet} \ar[uul] & {\ldots} & {\bullet}
\ar[rr] & & {\bullet} \ar[uul]
}
\end{equation}
The \emph{standard form} of a rooted quiver $(Q,v)$ of type $A$ is a
rooted quiver of type $A$ as in~\eqref{e:stdA} consisting of $s(Q)$
lines and $t(Q)$ triangles with the vertex $v$ as the root.
\end{dfn}

The name ``standard form'' is justified by the next theorem which
follows from the results of~\cite{Buan-Vatne}, see also
Section~\ref{sec:goodmut}.

\begin{theorem}
\label{thm-good-typeA}
Let $Q$ be a quiver of a cluster-tilted algebra of Dynkin type $A$.
Then $Q$ can be transformed via a sequence of good mutations to its
standard form. Moreover, two standard forms are derived equivalent if
and only if they coincide.
\end{theorem}

In Dynkin type $D$, we suggest the following standard forms.

\begin{theorem} \label{t:stdder}
A cluster-tilted algebra of type $D_n$ is derived
equivalent to one of the cluster-tilted algebras with the following
quivers, which we call ``standard forms'' for derived equivalence:
\begin{enumerate}
\renewcommand{\theenumi}{\alph{enumi}}
\renewcommand{\labelenumi}{(\theenumi)}
\item
$D_n$ (i.e.\ Type I with a linearly oriented $A_{n-2}$ quiver
attached);
\[
\xymatrix@=1pc{
{\bullet} \ar[dr] \\
& {\bullet} \ar[r] & {\dots} \ar[r] & {\bullet} \\
{\bullet} \ar[ur]
}
\]

\item
Type II as in the following figure, where $s,t \geq 0$ and $s+2t=n-4$;
\[
\xymatrix@=0.25pc{
&& {\bullet} \ar[ddll] &&
&& && & {\bullet} \ar[ddl] && && {\bullet} \ar[ddl]
\\ \\
{\bullet} \ar[rrrr] && &&
{\bullet} \ar[rr]^{1} \ar[uull] \ar[ddll]
& & {\ldots} \ar[rr]^{s} & &
{\bullet} \ar[rr]^{1}  & & {\bullet} \ar[uul] & {\ldots} & {\bullet}
\ar[rr]^{t} & & {\bullet} \ar[uul]
\\ \\
&& {\bullet} \ar[uull]
}
\]

\item
Type III as in the following figure, where $s,t \geq 0$ and $s+2t =
n-4$;
\[
\xymatrix@=0.25pc{
&& {\bullet} \ar[ddll] &&
&& && & {\bullet} \ar[ddl] && && {\bullet} \ar[ddl]
\\ \\
{\bullet} \ar[ddrr] && && {\bullet} \ar[rr]^{1} \ar[uull]
&& {\ldots} \ar[rr]^{s} &&
{\bullet} \ar[rr]^{1}  && {\bullet} \ar[uul] & {\ldots} & {\bullet}
\ar[rr]^{t} && {\bullet} \ar[uul]
\\ \\
&& {\bullet} \ar[uurr]
}
\]

\item [{($\mathrm{d_1}$)}]
(only when $n$ is odd)
Type IV with a central cycle of length $n$ without spikes,
as in the following picture:
\begin{center}
\includegraphics[scale=0.75]{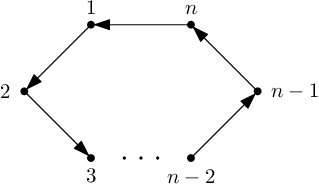}
\end{center}

\item [{($\mathrm{d_2}$)}]
Type IV with parameter sequence
$
\bigl( (1, s, t), (1, 0, 0), \dots, (1, 0, 0) \bigr)
$
of length $b \geq 3$, with $s, t \geq 0$ such that $n = 2b + s + 2t$,
and the attached rooted quiver of type $A$ is in standard form;
\begin{center}
\includegraphics[scale=0.75]{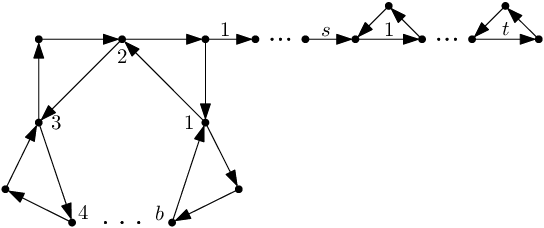}
\end{center}

\item [{($\mathrm{d_3}$)}]
Type IV with parameter sequence
$
\bigl( (1, 0, 0), (1, 0, 0), \dots, (1, 0, 0), (3, s_1, t_1), (3, s_2,
t_2), \dots, (3, s_k, t_k) \bigr)
$
for some $k > 0$, where the number of triples $(1, 0, 0)$ is $b
\geq 0$, the non-negative integers $s_1, t_1, \dots, s_k, t_k$ are
considered up to rotation of the sequence
$
\bigr((s_1, t_1), (s_2, t_2), \dots, (s_k, t_k) \bigr),
$
the attached rooted quivers of type $A$ are in standard form and
$n = 4k + 2b + s_1 + 2t_1 + \dots + s_k + 2t_k>4$.
\begin{center}
\includegraphics[scale=0.75]{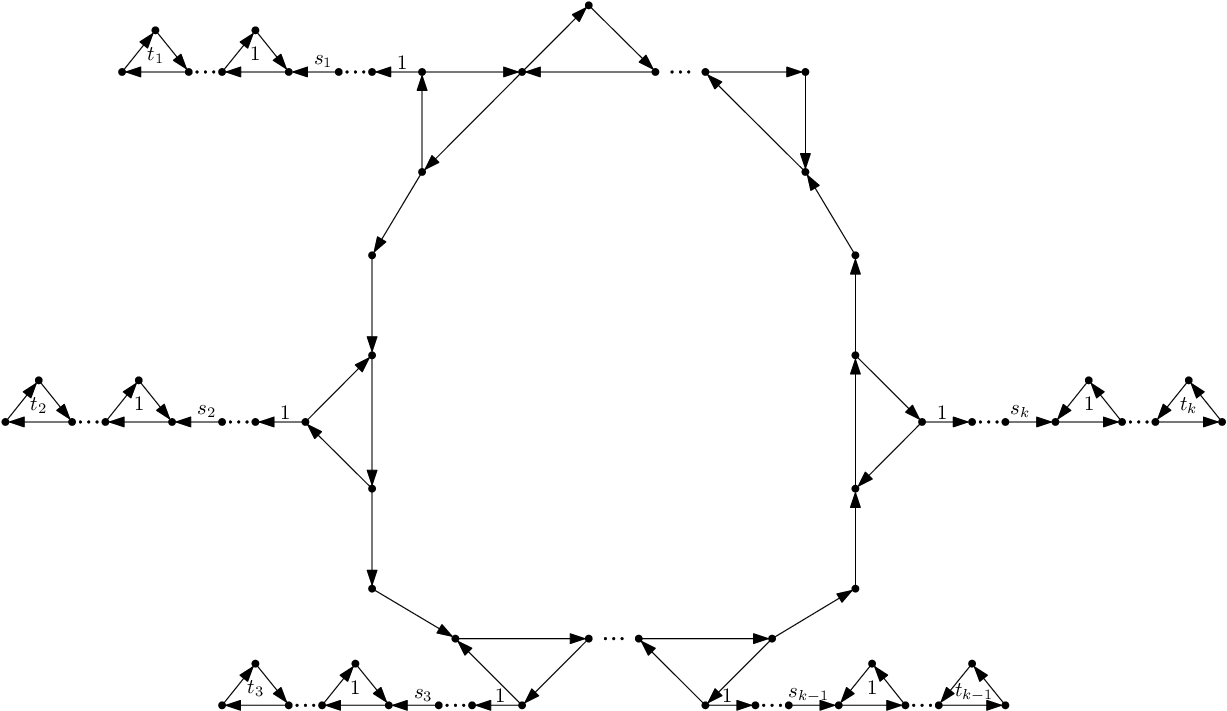}
\end{center}
\end{enumerate}

Moreover, any two distinct standard forms which are not of the class
($\mathrm{d_3}$) are not derived equivalent.
\end{theorem}

\begin{remark}
There is no known example of two distinct standard forms which are
derived equivalent.
\end{remark}

\begin{remark}
Our proof of Theorem \ref{t:stdder} actually shows that any
cluster-tilted algebra of type $D$ which is not self-injective can
be brought to a standard form by a sequence of good mutations and
good double mutations. An algorithm to compute the standard form of
a cluster-tilted algebra of type $D$ given in parametric notation is
presented in Section \ref{sec:algder}.
\end{remark}

\begin{remark} \label{rem:stddblgood}
Moreover, it can be shown that any two distinct standard forms
cannot be connected by a sequence of good mutations or good double
mutations.
\end{remark}

The latter aspect leads us to formulate the following conjecture.

\begin{conj} \label{conj:dereq}
Let $\gL$ and $\gL'$ be two cluster-tilted algebras of Dynkin type
$D$ which are not self-injective. Then $\gL$ and $\gL'$ are derived
equivalent if and only if one can connect $\gL'$ to $\gL$ or to its
opposite algebra 
$\gL^{op}$ by a sequence of good mutations and good double
mutations.
\end{conj}

The smallest cases in which this conjecture could not be settled
occur in types $D_{17}$ and $D_{19}$, see Example~\ref{ex:D19} and
Example~\ref{ex:D17}.

\begin{remark}
If Conjecture~\ref{conj:dereq} holds, then by
Remark~\ref{rem:stddblgood} any two derived equivalent standard
forms $\gL$ and $\gL'$ satisfy $\gL' \simeq \gL$ or $\gL' \simeq
\gL^{op}$. Hence, Conjecture~\ref{conj:dereq} together with an
affirmative answer to Question~\ref{q:opp} would imply a complete
derived equivalence classification of the cluster-tilted algebras of
Dynkin type $D$.
\end{remark}

\subsection{Numerical invariants}

The main tool for distinguishing the various standard forms appearing
in Theorem~\ref{t:stdder} is the computation of their numerical
invariants of derived equivalence described in
Section~\ref{sec:invariants}. We start by giving the formulae for the
determinants of the Cartan matrices of all cluster-tilted algebras of
type $D$.

\begin{theorem} \label{thm-det-typeD}
Let $Q$ be a quiver which is mutation equivalent to $D_n$ for $n \geq 4$.
Using the notation from Section \ref{sec:ctaAD} we have the following formulae
for the determinants of the Cartan matrices.
\begin{enumerate}
\renewcommand{\theenumi}{\Roman{enumi}}
\item
If $Q$ is of Type I, then $\det C_Q = 2^{t(Q')} = \det C_{Q'}$.
\item
If $Q$ is of Type II, then $\det C_Q = 2 \cdot 2^{t(Q')+t(Q'')} = 2
\cdot \det C_{Q'} \cdot \det C_{Q''}$.
\item
If $Q$ is of Type III, then $\det C_Q = 3 \cdot 2^{t(Q')+t(Q'')} = 3
\cdot \det C_{Q'} \cdot \det C_{Q''}$.
\item
For a quiver $Q$ of Type IV with central cycle of length $m \ge 3$, let
$Q^{(1)}, \ldots, Q^{(r)}$ be the rooted quivers of type $A$ glued to
the spikes and let $c(Q)$ be the number of vertices on the central
cycle which are part of two (consecutive) spikes, i.e.\ $c(Q)
= |\{1 \leq j \leq r \,:\, d_j = 1\}|$, cf.~\eqref{e:parDIV}. Then
\[
\det C_Q = (m+c(Q)-1) \cdot \prod_{j=1}^r 2^{t(Q^{(j)})} = (m+c(Q)-1)
\cdot \prod_{j=1}^r \det C_{Q^{(j)}}.
\]
\end{enumerate}
\end{theorem}

After uploading the first version of this paper we have been informed
by D.\ Vatne that he independently also computed these Cartan
determinants.

From Theorem~\ref{thm-det-typeD} we immediately obtain the
following.
\begin{cor} \label{cor:notdereq}
{\ }
\begin{enumerate}
\renewcommand{\theenumi}{\alph{enumi}}
\item
A cluster-tilted algebra in Type II is not derived equivalent to any
cluster-tilted algebra in Type III.
\item
A cluster-tilted algebra in Type II is not derived equivalent to any
cluster-tilted algebra in Type IV whose Cartan determinant is not a power of
$2$.
\end{enumerate}
\end{cor}

Note that the determinant alone is not enough to distinguish Types II
and IV, the smallest example occurs already in type $D_5$.
\begin{example}
The Cartan matrices of the cluster-tilted algebra in Type II with
parameters $(1,0,0,0)$ and the one in Type IV with parameters
$\bigl((3,1,0)\bigr)$ whose quivers are given by
\begin{align*}
\begin{array}{c}
\xymatrix@=1pc{
& {\bullet} \ar[dl] \\
{\bullet} \ar[rr] & & {\bullet} \ar[ul] \ar[dl] \ar[r]
& {\bullet} \\
& {\bullet} \ar[ul]
}
\end{array}
& &
\begin{array}{c}
\xymatrix@=1pc{
& {\bullet} \ar[dl] \ar[dr] \\
{\bullet} \ar[dr] & & {\bullet} \ar[dl] \ar[r] & {\bullet} \\
& {\bullet} \ar[uu]
}
\end{array}
\end{align*}
have both determinant $2$, but the characteristic polynomials of their
asymmetries differ, namely $x^5-x^3+x^2-1$ and $x^5-2x^3+2x^2-1$,
respectively.
\end{example}

Since the determinants of the Cartan matrices of all cluster-tilted
algebras of type $D$ do not vanish, one can consider their asymmetry
matrices and the corresponding characteristic polynomials. We list
below the characteristic polynomials of the asymmetry matrices for
cluster-tilted algebras of Types I, II and III of Dynkin type $D$
and for certain cases in Type IV. Combining this with
Theorem~\ref{thm-det-typeD}, we get the corresponding associated
polynomials. Using these it is possible to distinguish several further
standard forms of Theorem \ref{t:stdder} up to derived equivalence.
However, in the present paper we are not embarking on this aspect, and
therefore only list these polynomials for the sake of completeness. For
the proofs we refer to the thesis of the first
author~\cite{Bastian-thesis} as well as to the note~\cite{Ladkani_poly}
containing a general method for computing the associated polynomial of
an algebra obtained by gluing rooted quivers of type $A$ to a given
quiver with relations.

\begin{notat}
For a quiver $Q$ mutation equivalent to a Dynkin quiver, we denote
by $\chi_Q(x)$ the characteristic polynomial of the asymmetry matrix of
the Cartan matrix $C_Q$ of the cluster-tilted algebra corresponding to
$Q$.
\end{notat}

\begin{remark}
Let $Q$ be the quiver of a cluster-tilted algebra of type $A$. Then
\[
\chi_Q(x)= (x+1)^{t-1} \Bigl( x^{s+t+2} + (-1)^{s+1} \Bigr)
\]
where $s = s(Q)$ and $t = t(Q)$.
\end{remark}
\begin{remark} \label{rem:polyIV}
Consider a cluster-tilted algebra in type $D$ with quiver $Q$ of Type
I, II, III or IV and parameters as defined in Section~\ref{sec:ctaAD}.
\begin{enumerate}
\renewcommand{\theenumi}{\Roman{enumi}}
\item [(I)]
If $Q$ is of Type I, then
$
\chi_Q(x)= (x+1)^{t} (x-1) \Bigl(x^{s+t+2} + (-1)^{s}\Bigr)
$\\
where $s = s(Q')$ and $t = t(Q')$.

\item [(II/III)]
If $Q$ is of Type II or Type III, then
$
\chi_Q(x)= (x+1)^{t+1} (x-1) \Bigl(x^{s+t+2} + (-1)^{s+1}\Bigr)
$\\
where $s = s(Q') + s(Q'')$ and $t = t(Q') + t(Q'')$.

\item [(IV)]
If $Q$ is of Type IV, then we have the following.
\begin{enumerate}
\renewcommand{\theenumii}{\alph{enumii}}
\item
If $Q$ is an oriented cycle of length $n$ without spikes then
\[
\chi_Q(x) = \begin{cases} x^n - 1 & \text{if $n$ is odd,} \\
\left(x^{\frac{n}{2}} - 1\right)^2 & \text{if $n$ is even.}
\end{cases}
\]
\item
If $Q$ has parameters $\bigl((1,s,t), (1,0,0), \dots, (1,0,0) \bigr)$
with $b \geq 3$ spikes, then
\[
\chi_Q(x)= (x+1)^t (x^b-1) \Bigl(x^{s+t+b} + (-1)^{s+1}\Bigr).
\]
\item
If $Q$ has parameters $\bigl((3, s, t)\bigr)$
then
\[
\chi_Q(x) = (x+1)^{t-1} (x-1) \Bigl(x^{s+t+4} + 2 \cdot x^{s+t+3} +
(-1)^{s+1} \cdot 2x + (-1)^{s+1}\Bigr).
\]
\item \label{it:polyIV33}
If $Q$ has parameters $\bigl((3, s_1, t_1), (3, s_2, t_2)\bigr)$ then
\begin{align*}
\chi_Q(x) = (x+1)^{t_1+t_2-2} \cdot (x-1) \cdot
\Bigl(
& x^{s_1+t_1+s_2+t_2} \cdot (x^9 + 3x^8 + 4x^7 + 4x^6) \\
& + \bigl((-1)^{s_1} x^{s_2 + t_2} + (-1)^{s_2} x^{s_1+t_1} \bigr)
(x^5-x^4) \\
& + (-1)^{s_1+s_2+1}(1 + 3x + 4x^2 + 4x^3)
\Bigr) .
\end{align*}
\end{enumerate}
\end{enumerate}
\end{remark}

\begin{remark}
The third author has computed the Hochschild cohomology groups of
all the cluster-tilted algebras of Dynkin type~\cite{Ladkani12}. By
using these derived invariants it is possible to refine
Corollary~\ref{cor:notdereq} and show that a cluster-tilted algebra
in Type II with more than $4$ vertices is not derived equivalent to
any cluster-tilted algebra in Type IV, and that moreover a standard
form in the class ($\mathrm{d_3}$) is not derived equivalent to any
standard from in one of the classes (a), (b) or~(c). However, the
use of the Hochschild cohomology groups does not settle any of the
subtle questions raised in Section~\ref{sec:opp} and
Section~\ref{sec:D15}.
\end{remark}

\subsection{Complete derived equivalence classification up to $D_{14}$}
\label{sec:D14}

Fixing the number $n$ of vertices, it is possible to enumerate over all
the standard forms of derived equivalence with  $n$ vertices as given
in Theorem~\ref{t:stdder}, and compute the Cartan matrices of the
corresponding cluster-tilted algebras and their associated polynomials.
As long as the resulting polynomials (or any other derived invariants)
do not coincide for two distinct standard forms, the derived
equivalence classification is complete since then we know that any
cluster-tilted algebra in type $D$ is derived equivalent to one of the
standard forms, and moreover any two distinct such forms are not
derived equivalent.

By carrying out this procedure on a computer using the \textsc{Magma}
system~\cite{Magma}, we have been able to obtain a complete derived
equivalence classification of the cluster-tilted algebras of type $D_n$
for $n \leq 14$. Table~\ref{t:D14} lists, for $4 \leq n \leq 14$, the
number of such algebras (using the formula given
in~\cite{BuanTorkildsen09}) and the number of their derived equivalence
classes.

\begin{table}
\begin{center}
\begin{tabular}{r|r|r}
$n$ & Algebras & Classes \\
\hline 4 & 6 & 3 \\
5 & 26 & 5 \\
6 & 80 & 9 \\
7 & 246 & 10 \\
8 & 810 & 17 \\
9 & 2704 & 18 \\
10 & 9252 & 29 \\
11 & 32066 & 31 \\
12 & 112720 & 49 \\
13 & 400024 & 53 \\
14 & 1432860 & 81
\end{tabular}
\end{center}
\caption{The numbers of cluster-tilted algebras of type $D_n$ and their
derived equivalence classes, $n \leq 14$.} \label{t:D14}
\end{table}

As a consequence of our methods, we deduce the following
characterization of derived equivalence for cluster-tilted algebras of
type $D_n$ with $n \leq 14$.

\begin{theorem} \label{thm:D14}
Let $\gL$ and $\gL'$ be two cluster-tilted algebras of type $D_n$ with
$n \leq 14$. Then the following conditions are equivalent:
\begin{enumerate}
\renewcommand{\theenumi}{\alph{enumi}}
\item \label{it:D14:poly}
$\gL$ and $\gL'$ have the same associated polynomials;

\item \label{it:D14:Cartan}
The Cartan matrices of $\gL$ and $\gL'$ represent equivalent bilinear
forms over $\bZ$;

\item
$\gL$ and $\gL'$ are derived equivalent;

\item
Either $\gL$ and $\gL'$ are both self-injective, or none of them is
self-injective and they are connected by a sequence of good
mutations and good double mutations.
\end{enumerate}
\end{theorem}

\begin{remark}
The implications (d) $\Rightarrow$ (c) $\Rightarrow$ (b)
$\Rightarrow$ (a) always hold, regardless of the number of vertices.
The implication (a) $\Rightarrow$ (b) does not hold already in type
$D_{15}$, see Example~\ref{ex:D15} below. It is unknown whether the
implication (b) $\Rightarrow$ (c) holds in general for Dynkin type
$D$, see Example~\ref{ex:D15_op} in type $D_{15}$,
Example~\ref{ex:D19} in type $D_{19}$ as well as
Question~\ref{qu:derived}. Moreover, as the latter examples present
pairs of cluster-tilted algebras of type $D$ satisfying
condition~(b) but not~(d), it is impossible that both implications
(b) $\Rightarrow$ (c) and (c) $\Rightarrow$ (d) would hold in
general. Finally, note that Conjecture~\ref{conj:dereq} is a
slightly weaker version of the implication (c) $\Rightarrow$ (d).
\end{remark}

\begin{remark}
It follows from their derived equivalence classification
(see~\cite{Buan-Vatne} for type $A$ and~\cite{BHL09} for type $E$)
that a statement analogous to Theorem~\ref{thm:D14} is true also for
cluster-tilted algebras of Dynkin types $A$ and $E$, replacing the
condition (d) by:
\begin{enumerate}
\item [(d')]
\emph{$\gL$ and $\gL'$ are connected by a sequence of good mutations.}
\end{enumerate}
However, for Dynkin type $D$ the following two examples in types $D_6$
and $D_8$ demonstrate that one must replace condition (d') by the
weaker one (d). Thus, in some sense the derived equivalence
classification in type $D_8$ is more complicated than that in type
$E_8$.
\end{remark}

\begin{example}
\label{ex:selfinj} For any $b \geq 3$, the cluster-tilted algebra in
Type IV with a central cycle of length $2b$ without any spike is
derived equivalent to that in Type IV with parameter sequence
$\bigl((1,0,0), (1,0,0), \dots, (1,0,0)\bigr)$ of length $b$ (see
Lemma~\ref{l:selfinj}). There is no sequence of good mutations or
good double mutations connecting these two
self-injective~\cite{Ringel} algebras. Indeed, none of the algebra
mutations at any of the vertices is defined. The smallest such pair
occurs in type $D_6$ and the corresponding quivers are shown below.
\begin{align*}
\xymatrix@=0.4pc{
& {\bullet} \ar[ddl] & & {\bullet} \ar[ll] \\ \\
{\bullet} \ar[ddr] & & & & {\bullet} \ar[uul] \\ \\
& {\bullet} \ar[rr] & & {\bullet} \ar[uur]
}
& &
\xymatrix@=0.4pc{
{\bullet} \ar[rr] & & {\bullet} \ar[ddl] \ar[rr] & &
{\bullet} \ar[ddl] \\ \\
& {\bullet} \ar[uul] \ar[rr] & & {\bullet} \ar[uul] \ar[ddl] \\ \\
& & {\bullet} \ar[uul]
}
\end{align*}
\end{example}

\begin{example}
\label{ex:D8}
The two cluster-tilted algebras of type $D_8$ with quivers
\begin{align*}
\xymatrix@=0.4pc{
& {\bullet} \ar[ddl] & & {\bullet} \ar[ddl] & & {\bullet} \ar[ddl]
\\ \\
{\bullet} \ar[ddr] & & {\bullet} \ar[uul] \ar[rr]
& & {\bullet} \ar[uul] \ar[rr] & & {\bullet} \ar[uul] \\ \\
& {\bullet} \ar[uur]
}
& &
\xymatrix@=0.4pc{
& {\bullet} \ar[ddl] & & {\bullet} \ar[ddl] & & {\bullet} \ar[ddl]
\\ \\
{\bullet} \ar[rr] & & {\bullet} \ar[uul] \ar[ddr] & &
{\bullet} \ar[uul] \ar[rr] & & {\bullet} \ar[uul] \\ \\
& & & {\bullet} \ar[uur]
}
\end{align*}
of Type III are connected by a good double mutation and hence
derived equivalent (see Corollary~\ref{c:IIIdbl}), but they cannot
be connected by a sequence of good mutations. Indeed, the mutations
at all the vertices of the right quiver are bad (see cases II.1,
III.3 in Tables~\ref{t:mutII} and~\ref{t:mutIII} in
Section~\ref{sec:goodmut}).
\end{example}

\subsection{Opposite algebras}
\label{sec:opp}

The smallest example of two distinct standard forms (as in
Theorem~\ref{t:stdder}) for which it is unknown whether they are
derived equivalent or not arises for $n=15$ vertices. It is related to
the question of derived equivalence of an algebra and its opposite
which we briefly discuss below.

It follows from the description of the quivers and relations given in
Section~\ref{sec:ctaAD} that if $\gL$ is a cluster-tilted algebra of
Dynkin type $D$, then so is its opposite algebra $\gL^{op}$. Moreover,
a careful analysis of the classes given in Theorem~\ref{t:stdder} shows
that any cluster-tilted algebra with standard form in the classes
(a),(b),(c),($\mathrm{d_1}$) or ($\mathrm{d_2}$) is derived equivalent
to its opposite, and the opposite of a cluster-tilted algebra with
standard from in class ($\mathrm{d_3}$) and parameter sequence
\[
\bigl( (1, 0, 0), (1, 0, 0), \dots, (1, 0, 0), (3, s_1, t_1), (3, s_2,
t_2), \dots, (3, s_k, t_k) \bigr)
\]
is derived equivalent to a standard form in the same class, but with
parameter sequence
\[
\bigl( (1, 0, 0), (1, 0, 0), \dots, (1, 0, 0), (3, s_k, t_k), \dots,
(3, s_2, t_2), (3, s_1, t_1) \bigr)
\]
which may not be equivalent to the original one when $k \geq 3$. The
smallest such pairs of rotation-inequivalent standard forms occur when
$k=3$, and the number of vertices is then $15$.

The equivalence class of the integral bilinear form defined by the
Cartan matrix can be a very tricky derived invariant when it comes to
assessing the derived equivalence of an algebra $\gL$ and its opposite
$\gL^{op}$. Indeed, the Cartan matrix of $\gL^{op}$ is the transpose of
that of $\gL$. Since the bilinear forms defined by any square matrix
(over any field) and its transpose are always equivalent over that
field~\cite{DSZ03,Gow80,YipBallantine81}, it follows that the bilinear
forms defined by the integral matrices $C_{\gL}$ and $C_{\gL^{op}}$ are
equivalent over $\bQ$ as well as over all prime fields $\bF_p$. Hence
determining their equivalence over $\bZ$ becomes a delicate
arithmetical question. Moreover, it follows that the asymmetry matrices
(when defined) are similar over any field, and hence the associated
polynomials corresponding to $\gL$ and $\gL^{op}$ always coincide.

To illustrate these difficulties, we present here the two smallest
examples occurring in type $D_{15}$. In one example, the equivalence
of the bilinear forms is known, whereas in the other it is unknown.
In both cases, we are not able to tell whether the algebra is derived
equivalent to its opposite.

\begin{example} \label{ex:D15_op}
The Cartan matrices of the opposite cluster-tilted algebras of type
$D_{15}$ with standard forms
\begin{align*}
\bigl( (3, 1, 0), (3, 0, 1), (3, 0, 0) \bigr) &, & \bigl( (3, 0, 0),
(3, 0, 1), (3, 1, 0) \bigr)
\end{align*}
define equivalent bilinear forms over the integers. This has been
verified by computer search (using \textsc{Magma}~\cite{Magma}).
\end{example}

\begin{example}
For the opposite cluster-tilted algebras of type $D_{15}$ with standard
forms
\begin{align*}
\bigl( (3, 1, 0), (3, 2, 0), (3, 0, 0) \bigr) &, & \bigl( (3, 0, 0),
(3, 2, 0), (3, 1, 0) \bigr)
\end{align*}
it is unknown whether their Cartan matrices define equivalent bilinear
forms over $\bZ$.
\end{example}

The above discussion motivates the following question.
\begin{quest} \label{q:opp}
Let $Q$ be an acyclic quiver such that its path algebra $KQ$ is derived
equivalent to its opposite, and let $\gL$ be a cluster-tilted algebra
of type $Q$. Is it true that $\gL$ is derived equivalent to its
opposite $\gL^{op}$?
\end{quest}

\begin{remark}
The answer to the above question is affirmative for cluster-tilted
algebras of Dynkin types $A$, $E$ and affine type $\widetilde{A}$, as
well as for cluster-tilted algebras of type $D$ with at most $14$
simples. This follows from the corresponding derived equivalence
classifications.
\end{remark}

\begin{remark}
If the answer to the question is positive for cluster-tilted algebras
of Dynkin type $D$, then in class ($\mathrm{d_3}$) of
Theorem~\ref{t:stdder}, one has to consider the non-negative integers
$s_1, t_1, \dots, s_k, t_k$ in the $k$-term sequence
\[
\bigr((s_1, t_1), (s_2, t_2), \dots, (s_k, t_k) \bigr)
\]
up to rotation as well as order reversal.
\end{remark}

\subsection{Other open questions for $D_n$, $n \geq 15$}
\label{sec:D15}

An immediate consequence of part~(IV)(\ref{it:polyIV33}) of
Remark~\ref{rem:polyIV} is the following systematic construction of
standard forms with the same associated polynomial.

\begin{remark} \label{rem:collpoly}
Let $s_1, t_1, s_2, t_2 \geq 0$. Then the cluster-tilted algebras with
standard forms
\begin{align} \label{e:collpoly}
\bigl( (3, 2+s_1, t_1), (3, s_2, 2+t_2) \bigr) & & \bigl( (3, 2+s_2,
t_2), (3, s_1, 2+t_1) \bigr)
\end{align}
have the same associated polynomial.
In other words, for a cluster-tilted algebra with quiver
\[
\xymatrix@=0.3pc{
& & & & & & {\bullet} \ar[ddl] & & {\bullet} \ar[ll] \ar[rr]
& & {\bullet} \ar[ddl] \ar[rr] & & {\bullet} \ar[rr]
& & {Q'} \\ \\
& {\bullet} \ar[ddr] & & {\bullet} \ar[ddr] & & {\bullet} \ar[ddr]
& & & & {\bullet} \ar[uul] \\ \\
{Q''} \ar[uur] & & {\bullet} \ar[uur] \ar[ll] & &
{\bullet} \ar[uur] \ar[ll] & & {\bullet} \ar[ll] \ar[rr]
& & {\bullet} \ar[uur]
}
\]
where $Q'$ and $Q''$ are rooted quivers of type $A$, exchanging the
rooted quivers $Q'$ and $Q''$ does not change the associated polynomial.
\end{remark}

The following proposition, which is a specific case of a statement in
the note~\cite{Ladkani_euler}, shows that moreover, in most cases, the
Cartan matrices of the corresponding standard forms define equivalent
bilinear forms (over $\bZ$).

\begin{prop} \label{p:collCartan}
Let $s_1, s_2 \geq 0$ and $t_1, t_2 > 0$. Then the bilinear forms
defined by the Cartan matrices of the cluster-tilted algebras with
standard forms as in~\eqref{e:collpoly} are equivalent over $\bZ$. In
other words, for a cluster-tilted algebra with quiver
\[
\xymatrix@=0.3pc{
& & & & & & & & {\bullet} \ar[ddl] & & {\bullet} \ar[ll] \ar[rr]
& & {\bullet} \ar[ddl] \ar[rr] & & {\bullet} \ar[rr] \ar[ddl]
& & {\bullet} \ar[rr] & & {Q'} \\ \\
& {\bullet} \ar[ddr] & & {\bullet} \ar[ddr] & & {\bullet} \ar[ddr]
& & {\bullet} \ar[ddr] & & & & {\bullet} \ar[uul]
& & {\bullet} \ar[uul] \\ \\
{Q''} \ar[uur] & & {\bullet} \ar[uur] \ar[ll]
& & {\bullet} \ar[uur] \ar[ll] & & {\bullet} \ar[uur] \ar[ll]
& & {\bullet} \ar[ll] \ar[rr]
& & {\bullet} \ar[uur]
}
\]
where $Q'$ and $Q''$ are rooted quivers of type $A$, exchanging the
rooted quivers $Q'$ and $Q''$ does not change the equivalence class of
the bilinear form defined by the Cartan matrix.
\end{prop}

\begin{quest} \label{qu:derived}
Are any two cluster-tilted algebras as in
Proposition~\ref{p:collCartan} derived equivalent?
\end{quest}

Let $m \geq 1$. By considering the $m+1$ standard forms with parameters
\begin{align*}
\bigl((3, 2i, 2m+1-2i), (3, 2m+1-2i, 1+2i)\bigr) & & 0 \leq i \leq m
\end{align*}
and invoking Proposition~\ref{p:collCartan}, we obtain the following.
\begin{cor}
Let $m \geq 1$. Then one can find $m+1$ distinct standard forms of
cluster-tilted algebras of type $D_{6m+13}$ whose Cartan matrices define
equivalent bilinear forms.
\end{cor}

\begin{example} \label{ex:D19}
The smallest case occurs when $m=1$, giving a pair of
cluster-tilted algebras of type $D_{19}$  whose Cartan matrices define
equivalent bilinear forms but their derived equivalence is unknown.
They correspond to the choice of $Q'=A_1$ and $Q''=A_2$ in
Proposition~\ref{p:collCartan}.
\end{example}

In some of the remaining cases in Remark~\ref{rem:collpoly}, despite
the collision of the associated polynomials, it is still possible to
distinguish the standard forms by using the bilinear forms of their
Cartan matrices, whereas in other cases this remains unsettled. We
illustrate this by two examples.

\begin{example} \label{ex:D15}
The two standard forms $\bigl( (3, 2, 0), (3, 1, 2) \bigr)$ and $\bigl(
(3, 3, 0), (3, 0, 2) \bigr)$ in type $D_{15}$ corresponding to the
choice of $Q'=A_1$ and $Q''=A_2$ in Remark~\ref{rem:collpoly} have the
same associated polynomial, namely
\[
20 \cdot \left( x^{15} + 2 x^{14} + x^{13} - 4 x^{11} + x^9 - 3 x^8 +
3 x^7 - x^6 + 4 x^4 - x^2 - 2 x - 1 \right) \\
\]
but their asymmetry matrices are not similar over the finite field $\bF_3$
(and hence not over $\bZ$). Thus the bilinear forms defined by their Cartan
matrices are not equivalent so the two algebras are not derived equivalent.
\end{example}

This is the smallest non-trivial example of
Remark~\ref{rem:collpoly}, and it also shows that the implication
(\ref{it:D14:poly}) $\Rightarrow$ (\ref{it:D14:Cartan}) in
Theorem~\ref{thm:D14} does not hold in general for cluster-tilted
algebras of type $D$.

\begin{example} \label{ex:D17}
The two standard forms $\bigl((3,2,0),(3,3,2)\bigr)$ and
$\bigl((3,5,0),(3,0,2)\bigr)$ in type $D_{17}$ corresponding to the
choice of $Q'=A_1$ and $Q''=A_4$ in Remark~\ref{rem:collpoly} have the
same associated polynomials and the asymmetries are similar over $\bQ$
as well as over all finite fields $\bF_p$, but it is unknown whether
the bilinear forms are equivalent over $\bZ$ or not.
\end{example}

Table~\ref{t:D20} lists for $15 \leq n \leq 20$ the number of
cluster-tilted algebras of type $D_n$ (according to the formula
in~\cite{BuanTorkildsen09}) together with upper and lower bounds on the
number of their derived equivalence classes. There are two upper bounds
which are obtained by counting the number of standard forms given in
Theorem~\ref{t:stdder} with (``$\mathrm{max^{op}}$'') or without
(``max'') the assumption of affirmative answer to Question~\ref{q:opp}
concerning the derived equivalence of opposite cluster-tilted algebras.
The lower bound (``min'') is obtained by considering the number of
standard forms with distinct numerical invariants of derived
equivalence. It follows that in types $D_{15}$, $D_{16}$ and $D_{18}$
there are no unsettled cases apart from those arising as opposites of
algebras.

\begin{table}
\begin{center}
\begin{tabular}{r||r||r|r|r||}
& & \multicolumn{3}{|c||}{Classes} \\
\multicolumn{1}{c||}{$n$} & \multicolumn{1}{|c||}{Algebras}
& min & $\mathrm{max^{op}}$ & max \\
\hline
15 &    5170604 &  91 &  91 &  93 \\
16 &   18784170 & 136 & 136 & 139 \\
17 &   68635478 & 156 & 157 & 167 \\
18 &  252088496 & 231 & 231 & 248 \\
19 &  930138522 & 273 & 275 & 312 \\
20 & 3446167860 & 395 & 401 & 461
\end{tabular}
\end{center}
\caption{The number of cluster-tilted algebras of type $D_n$ together
with lower and upper bounds on the number of their derived equivalence
classes, $15 \leq n \leq 20$.} \label{t:D20}
\end{table}

\subsection{Good mutation and derived equivalence at one vertex}

Let $Q$ be mutation equivalent to a Dynkin quiver. When assessing
whether its quiver mutation at a vertex $k$ is good or not, one needs
to consider which of the algebra mutations at $k$ of the two
cluster-tilted algebras $\gL_Q$ and $\gL_{\mu_k(Q)}$ corresponding to
$Q$ and its mutation $\mu_k(Q)$ are defined.
\emph{A-priori}, there may be $16$ possibilities as there are four
algebra mutations to consider (negative and positive for $\gL_Q$ and
$\gL_{\mu_k(Q)}$) and each of them can be either defined or not.
However, our following result shows that the question which of the
algebra mutations of $\gL_Q$ at $k$ is defined and the analogous
question for $\gL_{\mu_k(Q)}$ are not independent of each other, and
the number of possibilities that can actually occur is only $5$.

\begin{prop}
Let $Q$ be mutation equivalent to a Dynkin quiver and let $k$ be a
vertex of $Q$. Consider the two algebra mutations $\mu^-_k(\gL_Q)$ and
$\mu^+_k(\gL_Q)$ of the cluster-tilted algebra $\gL_Q$.
\begin{enumerate}
\renewcommand{\theenumi}{\alph{enumi}}
\item
If none of these mutations is defined, then both algebra mutations
$\mu^-_k(\gL_{\mu_k(Q)})$ and $\mu^+_k(\gL_{\mu_k(Q)})$ are defined.
Obviously, the quiver mutation at $k$ is bad.

\item
If $\mu^-_k(\gL_Q)$ is defined but $\mu^+_k(\gL_Q)$ is not, then
$\mu^+_k(\gL_{\mu_k(Q)})$ is defined and $\mu^-_k(\gL_{\mu_k(Q)})$ is
not, hence the quiver mutation at $k$ is good.

\item
If $\mu^+_k(\gL_Q)$ is defined but $\mu^-_k(\gL_Q)$ is not, then
$\mu^-_k(\gL_{\mu_k(Q)})$ is defined and $\mu^+_k(\gL_{\mu_k(Q)})$ is
not, hence the quiver mutation at $k$ is good.

\item
If both algebra mutations of $\gL_Q$ at $k$ are defined, then either
both mutations of $\gL_{\mu_k(Q)}$ at $k$ are defined or none of them
is defined. Accordingly, the quiver mutation at $k$ may be good or bad.
\end{enumerate}

In other words, there are only $5$ possible cases which are given in
the table below, where $\surd$ indicates that the corresponding algebra
mutation is defined and $\times$ indicates that it is not.
\[
\begin{array}{|c|c|c|c|c|}
\mu^-_k(\gL_Q) & \mu^+_k(\gL_Q) & \mu^+(\gL_{\mu_k(Q)}) &
\mu^-_k(\gL_{\mu_k(Q)}) & \text{$\mu_k(Q)$ is \dots} \\
\hline
\times & \times & \surd & \surd  & \text{bad} \\
\surd  & \times & \surd & \times & \text{good} \\
\times & \surd  & \times & \surd & \text{good} \\
\surd  & \surd  & \surd  & \surd & \text{good} \\
\surd  & \surd & \times & \times & \text{bad} \\
\hline
\end{array}
\]
\end{prop}

Since a good mutation implies the derived equivalence of the
corresponding cluster-tilted algebras, it also implies that the
determinants of their Cartan matrices are equal. The next proposition
shows that in Dynkin types $A$ and $D$, the converse also holds.

\begin{prop}
Let $\gL_Q$ be a cluster-tilted algebra of Dynkin type $A$ or $D$. Then
the mutation of $Q$ at $k$ is good if and only if 
$\det C_Q = \det C_{\mu_k(Q)}$.
\end{prop}

This can be shown by checking all the cases discussed in
Section~\ref{sec:goodmut}.  Note that in type $E$, the statement of the
proposition does not hold, as can be verified by a computer. However,
the following weaker assertion is true for all Dynkin types including
type $E$, see~\cite[Corollary~1]{BHL09}, characterizing good mutations
as those mutations whose corresponding cluster-tilted algebras
are derived equivalent.

\begin{cor} \label{c:goodmut}
Let $\gL_Q$ be a cluster-tilted algebra of Dynkin type. Then the
mutation of $Q$ at $k$ is good if and only if $\gL_Q$ and
$\gL_{\mu_k(Q)}$ are derived equivalent.
\end{cor}

\subsection{Good mutation equivalence classification}

Whereas the classification according to derived equivalence becomes
subtle when the number of vertices grows, leaving some questions still
undecided, this does not happen for the stronger, but algorithmically
more tractable relation of good mutation equivalence.

\begin{dfn}
Two cluster-tilted algebras (of Dynkin type) with quivers $Q'$ and $Q''$
are called \emph{good mutation equivalent} if one can move from $Q'$ to
$Q''$ by performing a sequence of good mutations. In other words, there
exists a sequence of vertices $k_1, k_2, \dots, k_m$ such that if we set
$Q_0 = Q'$ and $Q_j = \mu_{k_j}(Q_{j-1})$ for $1 \leq j \leq m$, and denote
by $\gL_j$ the cluster-tilted algebra with quiver $Q_j$, then
$Q'' = Q_r$ and for any $1 \leq j \leq m$ we have
$\gL_j = \mu^-_{k_j}(\gL_{j-1})$ or $\gL_j = \mu^+_{k_j}(\gL_{j-1})$.
\end{dfn}

\begin{remark}
Any two cluster-tilted algebras which are good mutation equivalent are
also derived equivalent. The converse is true in Dynkin types $A$ and
$E$ but not in type $D$, see Section~\ref{sec:D14} above.
\end{remark}

We now present our results concerning good mutations in type $D$.

\begin{theorem} \label{t:alggood}
There is an algorithm which decides for two quivers which are mutation
equivalent to $D_n$ given in parametric notation (i.e.\ specified by
their Type I,II,III,IV and the parameters), whether the corresponding
cluster-tilted algebras are good mutation equivalent or not. 
\end{theorem}

\begin{remark}
It can be shown that the running time of this algorithm is at 
most quadratic in the number of parameters.
\end{remark}

We provide also a list of ``standard forms'' for good mutation
equivalence.

\begin{theorem}
\label{t:stdgood} A cluster-tilted algebra of type $D_n$ is good
mutation equivalent to exactly one of the cluster-tilted algebras
with the following quivers:
\begin{enumerate}
\renewcommand{\theenumi}{\alph{enumi}}
\renewcommand{\labelenumi}{(\theenumi)}
\item
$D_n$ (i.e.\ Type I with a linearly oriented $A_{n-2}$ quiver attached);
\[
\xymatrix@=1pc{
{\bullet} \ar[dr] \\
& {\bullet} \ar[r] & {\dots} \ar[r] & {\bullet} \\
{\bullet} \ar[ur]
}
\]

\item
Type II as in the following figure, where $S,T \geq 0$ and $S+2T=n-4$;
\[
\xymatrix@=0.25pc{
&& {\bullet} \ar[ddll] &&
&& && & {\bullet} \ar[ddl] && && {\bullet} \ar[ddl]
\\ \\
{\bullet} \ar[rrrr] && &&
{\bullet} \ar[rr]^{1} \ar[uull] \ar[ddll]
& & {\ldots} \ar[rr]^{S} & &
{\bullet} \ar[rr]^{1}  & & {\bullet} \ar[uul] & {\ldots} & {\bullet}
\ar[rr]^{T} & & {\bullet} \ar[uul]
\\ \\
&& {\bullet} \ar[uull]
}
\]

\item
Type III as in the following figure, with $S \geq 0$, the non-negative
integers $T_1, T_2$ are considered up to rotation of the sequence
$(T_1, T_2)$, and $S + 2(T_1 + T_2) = n-4$;
\[
\xymatrix@=0.25pc{
& {\bullet} \ar[ddr] && && {\bullet} \ar[ddr] &
&& {\bullet} \ar[ddll] &&
&& && & {\bullet} \ar[ddl] && && {\bullet} \ar[ddl]
\\ \\
{\bullet} \ar[uur] && {\bullet} \ar[ll]_{T_2} & {\ldots} & {\bullet}
\ar[uur] &&
{\bullet} \ar[ll]_{1} \ar[ddrr] && && {\bullet} \ar[rr]^{1} \ar[uull]
&& {\ldots} \ar[rr]^{S} &&
{\bullet} \ar[rr]^{1}  && {\bullet} \ar[uul] & {\ldots} & {\bullet}
\ar[rr]^{T_1} && {\bullet} \ar[uul]
\\ \\
&& && && && {\bullet} \ar[uurr]
}
\]

\item [($\mathrm{d_1}$)]
Type IV with a central cycle of length $n$ without any spikes;
\begin{center}
\includegraphics[scale=0.75]{type_IV_cycle}
\end{center}

\item [($\mathrm{d_{2,1}}$)]
Type IV with parameter sequence
$
\bigl( (1,S,0), (1,0,0), \dots, (1,0,0) \bigr)
$
of length $b \geq 3$, with $S \geq 0$ such that $n = 2b + S$ and the
attached rooted quiver of type $A$ is linearly oriented $A_{S+1}$;
\begin{center}
\includegraphics[scale=0.75]{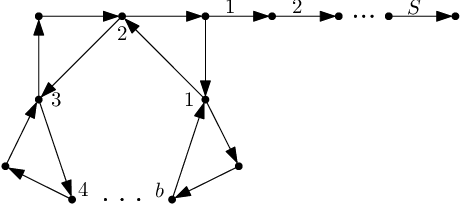}
\end{center}

\item [($\mathrm{d_{2,2}}$)]
Type IV with parameter sequence
\[
\bigl( (1,S,T_1), (1,0,0), \dots, (1,0,0),
(1,0,T_2), (1,0,0), \dots, (1,0,0), \dots,
(1,0,T_l), (1,0,0), \dots, (1,0,0) \bigr)
\]
which is a concatenation of $l \geq 1$ blocks of positive lengths
$b_1, b_2, \dots, b_l$ whose sum is not smaller than $3$,
with $S \geq 0$ and $T_1, \dots, T_l > 0$ considered up to rotation
of the sequence
$
\bigl( (b_1, T_1), (b_2, T_2), \dots, (b_l, T_l) \bigr),
$
$n = 2(b_1 + \dots + b_l + T_1 + \dots + T_l) + S$ and the attached
quivers of type $A$ are in standard form;
\begin{center}
\includegraphics[scale=0.75]{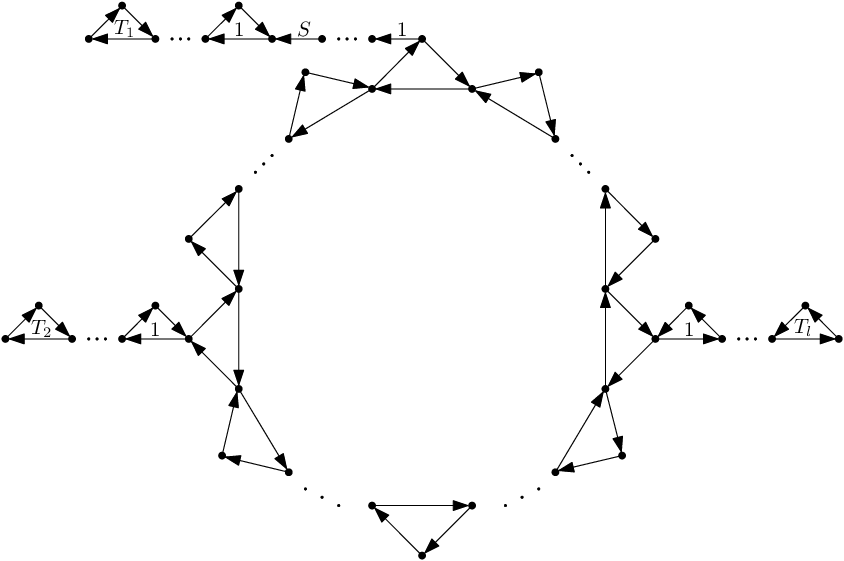}
\end{center}

\item [($\mathrm{d_{3,1}}$)]
Type IV with parameter sequence
$
\bigl( (1,0,0), \dots, (1,0,0), (3, S_1, 0), \dots, (3, S_a, 0) \bigr)
$
for some $a>0$, where the number of the triples $(1,0,0)$ is $b \geq
0$, the sequence of non-negative integers $(S_1, \dots, S_a)$ is
considered up to a cyclic permutation, $n = 4a + 2b + S_1 + \dots +
S_a$ and the attached rooted quivers of type $A$ are in standard form
(i.e.\ linearly oriented $A_{S_1+1}, \dots, A_{S_a+1}$);
\begin{center}
\includegraphics[scale=0.75]{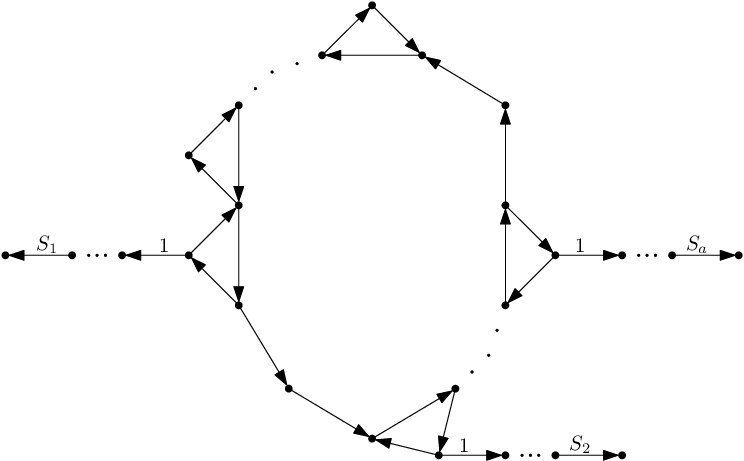}
\end{center}

\item [($\mathrm{d_{3,2}}$)]
Type IV with parameter sequence which is a concatenation of $l \geq 1$
sequences of the form
\[
\gamma_j = \begin{cases}
\bigl( (1,0,T_j), (1,0,0), \dots, (1,0,0), (3,S_{j,1},0), (3,S_{j,2},0),
\dots, (3,S_{j,a_j},0) \bigr) & \text{if $b_j > 0$,} \\
\bigl( (3, S_{j,1}, T_j), (3, S_{j,2}, 0), \dots, (3, S_{j,a_j},0)
& \text{otherwise,}
\end{cases}
\]
where each sequence $\gamma_j$ for $1 \leq j \leq l$ is defined by
non-negative integers $a_j$, $b_j$ not both zero, a sequence of $a_j$
non-negative integers $S_{j,1}, \dots, S_{j,a_j}$ and a positive
integer $T_j$, and not all the $a_j$ are zero. All these numbers are
considered up to rotation of the $l$-term sequence
\[
\Bigl( \bigl(b_1, (S_{1,1},\dots,S_{1,a_1}), T_1 \bigr),
\bigl(b_2, (S_{2,1},\dots,S_{2,a_2}), T_2 \bigr), \dots,
\bigl(b_l, (S_{l,1},\dots,S_{l,a_l}), T_l \bigr)
\Bigr),
\]
they satisfy $n = \sum_{j=1}^l (4a_j + 2b_j + S_{j,1} + \dots +
S_{j,a_j} + 2T_j)$, and the attached rooted quivers of type $A$ are in
standard form.

In other words, the quiver is a concatenation of $l \geq 1$ quivers $\gamma_j$ of the form
$$\gamma_j=
\begin{cases}
\includegraphics[scale=0.75]{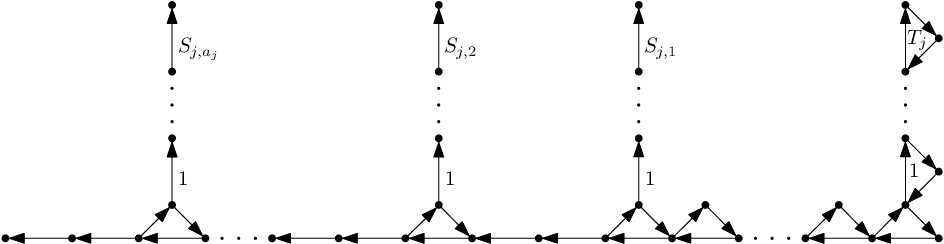} & \text{if $b_j > 0$,}\\
\includegraphics[scale=0.75]{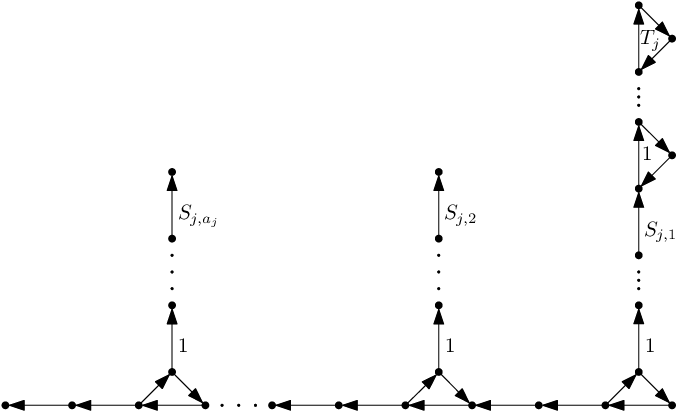} & \text{if $b_j = 0$,}
\end{cases}$$
where the last vertex of $\gamma_l$ is glued to the first vertex of $\gamma_1$.
\end{enumerate}
\end{theorem}

\begin{remark}
An algorithm to compute the standard form of a cluster-tilted
algebra of type $D$ given in parametric notation is presented in
Section \ref{sec:alg:goodmut}.
\end{remark}

The next remark explains how the standard forms for good mutation
equivalence can be seen as refinements of the standard forms for
derived equivalence appearing in Theorem~\ref{t:stdder}.

\begin{remark}
The derived equivalence classes of the forms~(a) and~(b) in
Theorem~\ref{t:stdder} are also good mutation equivalence classes
(of the corresponding form). However, derived equivalence classes of
the form~(c) in Theorem~\ref{t:stdder} decompose into (generally,
many) good mutation equivalence classes of the form (c) in
Theorem~\ref{t:stdgood}.

For a standard form listed in part ($\mathrm{d_{2}}$) of Theorem
\ref{t:stdder}, if $t=0$ then the set of cluster-tilted algebras
which can be brought to that standard form by a sequence of good
mutations and good double mutations comprises a good mutation
equivalence class of the form ($\mathrm{d_{2,1}}$), otherwise this
set decomposes into (usually many) good mutation equivalence classes
of the form ($\mathrm{d_{2,2}}$). Similarly, for a standard form
listed in part ($\mathrm{d_{3}}$) of Theorem \ref{t:stdder}, if all
the $t_i$ are zero then the corresponding set of cluster-tilted
algebras comprises a good mutation equivalence class of the form
($\mathrm{d_{3,1}}$), otherwise it decomposes into (usually many)
good mutation equivalence classes of the form ($\mathrm{d_{3,2}}$).
\end{remark}

\section{Good mutation equivalences}
\label{sec:goodmut}
In this section we determine all the good mutations for cluster-tilted algebras
of Dynkin types $A$ and $D$.

\subsection{Rooted quivers of type $A$}

For a rooted quiver $(Q,v)$ of type $A$, we call a mutation at a vertex
other than the root $v$ a mutation \emph{outside the root}.

\begin{prop} \label{p:goodmutA}
Any two rooted quivers of type $A$ with the same numbers of lines and
triangles can be connected by a sequence of good mutations outside the
root.
\end{prop}

\begin{remark} \label{rem:stdA}
It is enough to show that a rooted quiver of type $A$ can be transformed
to its standard form via good mutations outside the root.
\end{remark}

We begin by characterizing the good mutations in Dynkin type $A$.

\begin{lemma}
Let $Q$ be a quiver mutation equivalent to $A_n$. Then a mutation of
$Q$ is good if and only if it does not change the number of triangles.
\end{lemma}
\begin{proof}
Each row of Table~\ref{t:mutA} displays a pair of neighborhoods of a
vertex $\bullet$ in such a quiver related by a mutation (at $\bullet$).
Using the description of the relations of the corresponding
cluster-tilted algebras as in Remark~\ref{rem:relctaA}, we can use
Proposition~\ref{p:critmut} and easily determine, for each entry in the
table, which of the negative $\mu^-_{\bullet}$ or the positive
$\mu^+_\bullet$ mutations is defined. Then Proposition~\ref{p:goodmut}
tells us if the quiver mutation is good or not.

By examining the entries in the table, we see that the only bad
mutation occurs in row $2b$, where a triangle is created (or
destroyed).
\end{proof}

\begin{table}
\[
\begin{array}{|c|cc|cc|c|}
\hline
1 &
\begin{array}{c}
\xymatrix@R=1pc@C=1.2pc{
{\circ} \ar[dr] \\
& {\bullet} \\
& &
}
\end{array}
& \mu^-_{\bullet} &
\begin{array}{c}
\xymatrix@R=1pc@C=1.2pc{
{\circ} \\
& {\bullet} \ar[ul] \\
& &
}
\end{array}
& \mu^+_{\bullet}
& \text{good}
\\ \hline
2a &
\begin{array}{c}
\xymatrix@R=1pc@C=1.2pc{
{\circ} \ar[dr] \\
& {\bullet} \\
& & {\circ} \ar[ul]
}
\end{array}
& \mu^-_{\bullet} &
\begin{array}{c}
\xymatrix@R=1pc@C=1.2pc{
{\circ} \\
& {\bullet} \ar[ul] \ar[dr] \\
& & {\circ}
}
\end{array}
& \mu^+_{\bullet}
& \text{good}
\\ \hline
2b &
\begin{array}{c}
\xymatrix@R=1pc@C=1.2pc{
{\circ} \ar[dr] \\
& {\bullet} \ar[dl] \\
{\circ} & &
}
\end{array}
& \mu^-_{\bullet}, \mu^+_{\bullet} &
\begin{array}{c}
\xymatrix@R=1pc@C=1.2pc{
{\circ} \ar[dd] \\
& {\bullet} \ar[ul] \\
{\circ} \ar[ur] & &
}
\end{array}
& \text{none}
& \text{bad}
\\ \hline
3 &
\begin{array}{c}
\xymatrix@R=1pc@C=1.2pc{
{\circ} \ar[dr] \\
& {\bullet} \ar[dl] \\
{\circ} \ar[rr] & & {\circ} \ar[ul]
}
\end{array}
& \mu^-_{\bullet} &
\begin{array}{c}
\xymatrix@R=1pc@C=1.2pc{
{\circ} \ar[dd] \\
& {\bullet} \ar[ul] \ar[dr] \\
{\circ} \ar[ur] & & {\circ}
}
\end{array}
& \mu^+_{\bullet}
& \text{good}
\\ \hline
4 &
\begin{array}{c}
\xymatrix@R=1pc@C=1.2pc{
{\circ} \ar[dr] & & {\circ} \ar[ll] \\
& {\bullet} \ar[dl] \ar[ur] \\
{\circ} \ar[rr] & & {\circ} \ar[ul]
}
\end{array}
& \mu^-_{\bullet}, \mu^+_{\bullet} &
\begin{array}{c}
\xymatrix@R=1pc@C=1.2pc{
{\circ} \ar[dd] & & {\circ} \ar[dl] \\
& {\bullet} \ar[ul] \ar[dr] \\
{\circ} \ar[ur] & & {\circ} \ar[uu]
}
\end{array}
& \mu^-_{\bullet}, \mu^+_{\bullet}
& \text{good}
\\ \hline
\end{array}
\]
\caption{The neighborhoods in Dynkin type $A$ and their mutations.}
\label{t:mutA}
\end{table}

\begin{proof}[Proof of Proposition~\ref{p:goodmutA}]
In view of Remark~\ref{rem:stdA},
we give an algorithm for the mutation to the standard form above
(similar to the procedures in \cite{Bastian} and \cite{Buan-Vatne}):
Let $Q$ be a rooted quiver of type $A$ which has at least one triangle
(otherwise we get the desired orientation of a standard form by
sink/source mutations as in 1 and 2a in Table~\ref{t:mutA}).
For any triangle $C$ in $Q$ denote by $v_C$ the unique vertex of the
triangle having minimal distance to the root $c$. Choose a triangle
$C_1$ in $Q$ such that to the vertices of the triangle $\not=
v_1:=v_{C_1}$ only linear parts are attached; denote them by $p_1$ and
$p_2$, respectively.

\begin{center}
\includegraphics[scale=0.75]{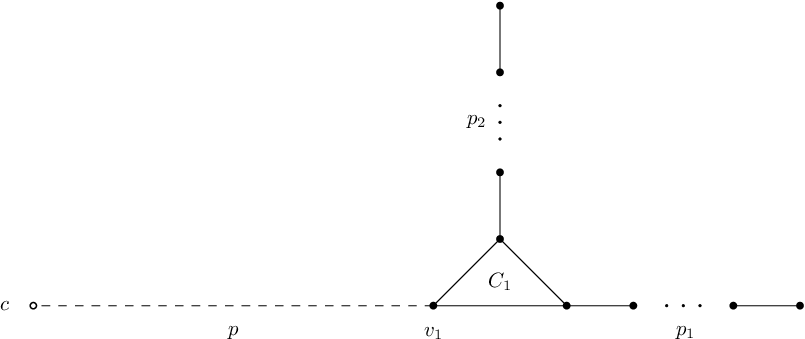}
\end{center}
Denote by $C_2, \dots, C_k$ the (possibly) other triangles along the path $p$
from $v_1$ to $c$.

\begin{center}
\includegraphics[scale=0.75]{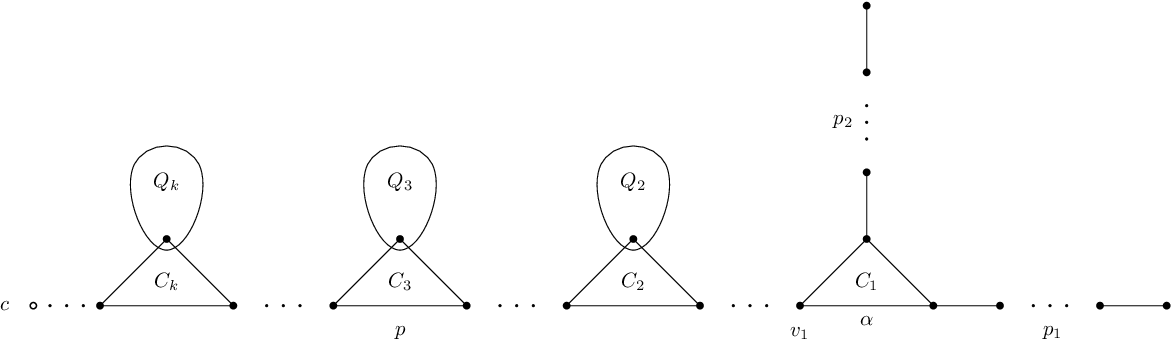}
\end{center}
Now we move all subquivers $p_2$, $Q_2, \dots, Q_k$ onto the path $p_1 \alpha p$.
For this we use the same mutations as in the steps 1 and 2 in in \cite[Lemma 3.10]{Bastian}.
Note that this can be done with the good mutations presented in Table~\ref{t:mutA}.
Thus, we get a new complete set of triangles $\{C_1, C'_2, \dots, C'_l\}$ on the
path from $v_1$ to $c$:

\begin{center}
\includegraphics[scale=0.75]{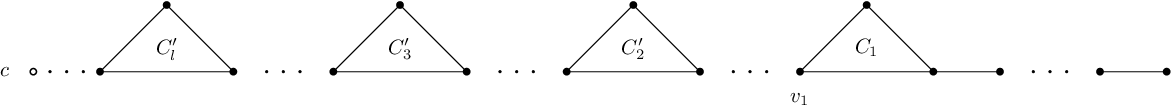}
\end{center}
We move all the triangles along the path to the right side.
For this we use the same mutations as in step 4 in \cite[Lemma 3.10]{Bastian}.
We then obtain a quiver of the form

\begin{center}
\includegraphics[scale=0.75]{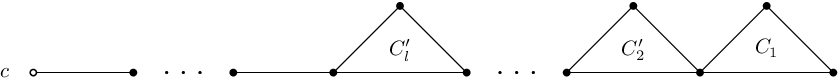}
\end{center}
\end{proof}

\subsection{Good mutations in Types I and II}

The good mutations involving quivers in Types I and II are given in
Tables~\ref{t:mutI} and~\ref{t:mutII} below.
In each row of these tables, we list:
\begin{enumerate}
\renewcommand{\theenumi}{\alph{enumi}}
\item
The quiver, where $Q$, $Q'$, $Q''$ and $Q'''$ are rooted quivers of type $A$;

\item
Which of the algebra mutations (negative $\mu^-_{\bullet}$, or positive
$\mu^+_{\bullet}$) at the distinguished vertex $\bullet$ are defined;

\item
The (Fomin-Zelevinsky) mutation of the quiver at the vertex $\bullet$;
and for the corresponding cluster-tilted algebra:

\item
Which of the algebra mutations at the vertex $\bullet$ are defined.

\item
Based on these, we determine whether the mutation is good or not,
see Proposition~\ref{p:goodmut}.
\end{enumerate}

To check whether a mutation is defined or not, we use the criterion of
Proposition~\ref{p:critmut}. Observe that since the gluing process
introduces no new relations, it is enough to assume that each rooted
quiver of type $A$ consists of just a single vertex. The finite list of
quivers we obtain can thus be examined by using a computer. We
illustrate the details of these checks on a few examples. Since there
is at most one arrow between any two vertices, we indicate a path by
the sequence of vertices it traverses.

\begin{table}
\[
\begin{array}{|c|cc|cc|c|}
\hline
\text{I.1} &
\begin{array}{c}
\xymatrix@=1pc{
{\bullet} \ar[dr] \\
& {Q} \\
{\circ} \ar[ur]
}
\end{array}
& \mu^+_{\bullet} &
\begin{array}{c}
\xymatrix@=1pc{
{\bullet} \\
& {Q} \ar[ul] \\
{\circ} \ar[ur]
}
\end{array}
& \mu^-_{\bullet}
& \text{good}
\\ \hline
%
\text{I.2} &
\begin{array}{c}
\xymatrix@=1pc{
{\circ} \ar[dr] \\
& {\bullet} & {Q} \ar[l]\\
{\circ} \ar[ur]
}
\end{array}
& \mu^-_{\bullet} &
\begin{array}{c}
\xymatrix@=1pc{
{\circ} \\
& {\bullet} \ar[ul] \ar[dl] \ar[r] & {Q}\\
{\circ}
}
\end{array}
& \mu^+_{\bullet}
& \text{good}
\\ \hline
%
\text{I.3a} &
\begin{array}{c}
\xymatrix@=1pc{
{\circ} \ar[dr] \\
& {\bullet} \ar[r] & {Q} \\
{\circ} \ar[ur]
}
\end{array}
& \mu^-_{\bullet}, \mu^+_{\bullet} &
\begin{array}{c}
\xymatrix@=1pc{
& {\circ} \ar[dl] \\
{Q} \ar[rr] & & {\bullet} \ar[ul] \ar[dl]\\
& {\circ} \ar[ul]
}
\end{array}
& \text{none}
& \text{bad}
\\ \hline
%
\text{I.3b} &
\begin{array}{c}
\xymatrix@=1pc{
{\circ} \\
& {\bullet} \ar[ul] \ar[dl] & {Q} \ar[l] \\
{\circ}
}
\end{array}
& \mu^-_{\bullet}, \mu^+_{\bullet} &
\begin{array}{c}
\xymatrix@=1pc{
& {\circ} \ar[dl] \\
{\bullet} \ar[rr] & & {Q} \ar[ul] \ar[dl]\\
& {\circ} \ar[ul]
}
\end{array}
& \text{none}
& \text{bad}
\\ \hline
%
\text{I.4a} &
\begin{array}{c}
\xymatrix@=1pc{
{\circ} \ar[dr] \\
& {\bullet} \ar[dl] \ar[r] & {Q} \\
{\circ}
}
\end{array}
& \mu^-_{\bullet}, \mu^+_{\bullet} &
\begin{array}{c}
\xymatrix@=1pc{
& {Q} \ar[dr] \\
{\circ} \ar[ur] \ar[dr] & & {\bullet} \ar[ll] \\
& {\circ} \ar[ur]
}
\end{array}
& \text{none}
& \text{bad}
\\ \hline
%
\text{I.4b} &
\begin{array}{c}
\xymatrix@=1pc{
{\circ} \ar[dr] \\
& {\bullet} \ar[dl] & {Q} \ar[l] \\
{\circ}
}
\end{array}
& \mu^-_{\bullet}, \mu^+_{\bullet} &
\begin{array}{c}
\xymatrix@=1pc{
& {Q} \ar[dr] \\
{\bullet} \ar[ur] \ar[dr] & & {\circ} \ar[ll] \\
& {\circ} \ar[ur]
}
\end{array}
& \text{none}
& \text{bad}
\\ \hline
%
\text{I.4c} &
\begin{array}{c}
\xymatrix@=1pc{
{\circ} \ar[dr] & & {Q''} \ar[dd] \\
& {\bullet} \ar[dl] \ar[ur] \\
{\circ} & & {Q'} \ar[ul]
}
\end{array}
& \mu^-_{\bullet}, \mu^+_{\bullet} &
\begin{array}{c}
\xymatrix@=1pc{
{\circ} \ar[dd] \ar[rr] & & {Q''} \ar[dl] \\
& {\bullet} \ar[ul] \ar[dr] \\
{\circ} \ar[ur] & & {Q'} \ar[ll]
}
\end{array}
& \text{none}
& \text{bad}
\\ \hline
%
\text{I.5a} &
\begin{array}{c}
\xymatrix@=1pc{
{\circ} \ar[dr] & & {Q''} \ar[dd] \\
& {\bullet} \ar[ur] \\
{\circ} \ar[ur] & & {Q'} \ar[ul]
}
\end{array}
& \mu^-_{\bullet} &
\begin{array}{c}
\xymatrix@=1pc{
& {\circ} \ar[dl] \\
{Q''} \ar[rr] & & {\bullet} \ar[ul] \ar[dl] \ar[r] & {Q'} \\
& {\circ} \ar[ul]
}
\end{array}
& \mu^+_{\bullet}
& \text{good}
\\ \hline
%
\text{I.5b} &
\begin{array}{c}
\xymatrix@=1pc{
{\circ} & & {Q''} \ar[dd] \\
& {\bullet} \ar[ul] \ar[dl] \ar[ur] \\
{\circ} & & {Q'} \ar[ul]
}
\end{array}
& \mu^+_{\bullet} &
\begin{array}{c}
\xymatrix@=1pc{
& & {\circ} \ar[dl] \\
{Q''} \ar[r] & {\bullet} \ar[rr] & & {Q'} \ar[ul] \ar[dl] \\
& & {\circ} \ar[ul]
}
\end{array}
& \mu^-_{\bullet}
& \text{good}
\\ \hline
\end{array}
\]
\caption{Mutations involving Type I quivers.} \label{t:mutI}
\end{table}

\begin{table}
\[
\begin{array}{|c|cc|cc|c|}
\hline
\text{II.1} &
\begin{array}{c}
\xymatrix@=1pc{
& {\bullet} \ar[dl] \\
{Q''} \ar[rr] & & {Q'} \ar[ul] \ar[dl] \\
& {\circ} \ar[ul]
}
\end{array}
& \mu^-_{\bullet}, \mu^+_{\bullet} &
\begin{array}{c}
\xymatrix@=1pc{
& {\bullet} \ar[dr] \\
{Q''} \ar[ur] & & {Q'} \ar[dl] \\
& {\circ} \ar[ul]
}
\end{array}
& \text{none}
& \text{bad}
\\ \hline
%
\text{II.2} &
\begin{array}{c}
\xymatrix@=1pc{
& {\circ} \ar[dl] \\
{Q''} \ar[rr] & & {\bullet} \ar[ul] \ar[dl] & {Q'} \ar[l] \\
& {\circ} \ar[ul]
}
\end{array}
& \mu^-_{\bullet} &
\begin{array}{c}
\xymatrix@=1pc{
& & {\circ} \ar[dl] \\
{Q''} & {\bullet} \ar[l] \ar[rr] & & {Q'} \ar[ul] \ar[dl] \\
& & {\circ} \ar[ul]
}
\end{array}
& \mu^+_{\bullet}
& \text{good}
\\ \hline
%
\text{II.3} &
\begin{array}{c}
\xymatrix@=1pc{
& {\circ} \ar[dl] & & {Q'} \ar[dl] \\
{Q''} \ar[rr] & & {\bullet} \ar[ul] \ar[dl] \ar[dr] \\
& {\circ} \ar[ul] & & {Q'''} \ar[uu]
}
\end{array}
& \mu^-_{\bullet}, \mu^+_{\bullet} &
\begin{array}{c}
\xymatrix@=1pc{
{Q'''} \ar[dr] & & {\circ} \ar[dl] \\
& {\bullet} \ar[dl] \ar[rr] & & {Q'} \ar[ul] \ar[dl] \\
{Q''} \ar[uu] & & {\circ} \ar[ul]
}
\end{array}
& \mu^-_{\bullet}, \mu^+_{\bullet}
& \text{good}
\\ \hline
\end{array}
\]
\caption{Mutations involving Type II quivers.} \label{t:mutII}
\end{table}

\begin{example}
Consider the case I.4c in Table~\ref{t:mutI}. We look at the two
cluster-tilted algebras $\gL$ and $\gL'$ with the following quivers
\begin{align*}
\xymatrix@=1pc{
{\bullet_1} \ar[dr] & & {\bullet_4} \ar[dd] \\
& {\bullet_0} \ar[dl] \ar[ur] \\
{\bullet_2} & & {\bullet_3} \ar[ul]
}
& &
\xymatrix@=1pc{
{\bullet_1} \ar[dd] \ar[rr] & & {\bullet_4} \ar[dl] \\
& {\bullet_0} \ar[ul] \ar[dr] \\
{\bullet_2} \ar[ur] & & {\bullet_3} \ar[ll]
}
\end{align*}
and examine their mutations at the vertex $0$.

Since the arrow $1 \to 0$ does not appear in any relation of $\gL$, its
composition with any non-zero path starting at $0$ is non-zero. Thus
the negative mutation $\mu^-_0(\gL)$ is defined. Similarly, since $0
\to 2$ does not appear in any relation of $\gL$, its composition with
any non-zero path that ends at $0$ is non-zero, and the positive
mutation $\mu^+_0(\gL)$ is also defined.

Consider now $\gL'$. The path $0,1,2$ is non-zero, as it equals
$0,3,2$, but both compositions $2,0,1,2$ and $4,0,1,2$ vanish because
of the zero relations $2,0,1$ and $4,0,1$, hence $\mu^-_0(\gL')$ is not
defined. Similarly, the path $1,2,0$ is non-zero, as it equals $1,4,0$,
but both compositions $1,2,0,1$ and $1,2,0,3$ vanish because of the
zero relations $2,0,1$ and $2,0,3$, showing that $\mu^+_0(\gL')$ is not
defined.
\end{example}

\begin{example}
Consider the case I.5a in Table~\ref{t:mutI}. We look at the two
cluster-tilted algebras $\gL$ and $\gL'$ with the following quivers
\begin{align*}
\xymatrix@=1pc{
{\bullet_1} \ar[dr] & & {\bullet_4} \ar[dd] \\
& {\bullet_0} \ar[ur] \\
{\bullet_2} \ar[ur] & & {\bullet_3} \ar[ul]
}
& &
\xymatrix@=1pc{
& {\bullet_1} \ar[dl] \\
{\bullet_4} \ar[rr] & & {\bullet_0} \ar[ul] \ar[dl] \ar[r]
& {\bullet_3} \\
& {\bullet_2} \ar[ul]
}
\end{align*}
and examine their mutations at the vertex $0$.

As in the previous example, since the arrow $1 \to 0$ (or $2 \to 0$)
does not appear in any relation of $\gL$, the negative mutation
$\mu^-_0(\gL)$ is defined. But $\mu^+_0(\gL)$ is not defined since the
composition of the arrow $3 \to 0$ with $0 \to 4$ vanishes. Similarly
for $\gL'$, the positive mutation $\mu^+_0(\gL')$ is defined since the
arrow $0 \to 3$ does not appear in any relation, and $\mu^-_0(\gL')$ is
not defined since the composition of the arrow $4 \to 0$ with the arrow
$0 \to 1$ vanishes.
\end{example}

\begin{example}
Consider the case II.3 in Table~\ref{t:mutII}. We will show that
if $\gL$ is one of the cluster-tilted algebras with the quivers
given below
\begin{align*}
\xymatrix@=1pc{
& {\bullet_1} \ar[dl] & & {\bullet_3} \ar[dl]_{\alpha} \\
{\bullet_4} \ar[rr]^{\alpha'}
& & {\bullet_0} \ar[ul] \ar[dl]^{\beta'} \ar[dr]_{\beta} \\
& {\bullet_2} \ar[ul] & & {\bullet_5} \ar[uu]
}
& &
\xymatrix@=1pc{
{\bullet_5} \ar[dr]^{\alpha} & & {\bullet_1} \ar[dl]_{\alpha'} \\
& {\bullet_0} \ar[dl]^{\beta} \ar[rr]_{\beta'}
& & {\bullet_3} \ar[ul] \ar[dl] \\
{\bullet_4} \ar[uu] & & {\bullet_2} \ar[ul]
}
\end{align*}
then both algebra mutations $\mu^-_0(\gL)$ and $\mu^+_0(\gL)$ are defined.

Indeed, let $p=\gamma_1\gamma_2 \dots \gamma_r$ be a non-zero path starting
at $0$ written as a sequence of arrows. If $\gamma_1 \neq \beta$, then
the composition $\alpha \cdot p$ is not zero, whereas otherwise the composition
$\alpha' \cdot p$ is not zero, hence $\mu^-_0(\gL)$ is defined. Similarly,
if $p=\gamma_1 \dots \gamma_r$ is a non-zero path ending at $0$, then
composition $p \cdot \beta$ is not zero if $\gamma_r \neq \alpha$, and
otherwise $p \cdot \beta'$ is not zero, hence $\mu^+_0(\gL)$ is defined as well.
\end{example}

\subsection{Good mutations in Types III and IV}

These are given in Tables~\ref{t:mutIII} and~\ref{t:mutIV}.
Table~\ref{t:mutIII} is computed in a similar way to
Tables~\ref{t:mutI} and~\ref{t:mutII}. In Table~\ref{t:mutIV}, the
dotted lines indicate the central cycle, and the two vertices at the
sides may be identified (for the right quivers in IV.2a and IV.2b,
these identifications lead to the left quivers of III.1 and III.2). The
proof that all the mutations listed in that table are good relies on
the lemmas below.

Mutations at vertices on the central cycle are discussed in
Lemmas~\ref{l:IV0c}, \ref{l:IV1c} and~\ref{l:IV2c}, whereas mutations at
the spikes are discussed in Lemmas~\ref{l:IV0s} and~\ref{l:IV1s}. The
moves IV.1a and IV.1b in Table~\ref{t:mutIV} follow from
Corollary~\ref{c:IV0}. The moves IV.2a and IV.2b follow
from Corollary~\ref{c:IV1}. Lemma~\ref{l:IV2c} implies
that there are no additional good mutations involving Type IV quivers.

\begin{table}
\[
\begin{array}{|c|cc|cc|c|}
\hline
\text{III.1} &
\begin{array}{c}
\xymatrix@=1pc{
& {\circ} \ar[dl] & & {Q'} \ar[dl] \\
{Q''} \ar[dr] & & {\bullet} \ar[ul] \\
& {\circ} \ar[ur]
}
\end{array}
& \mu^-_{\bullet} &
\begin{array}{c}
\xymatrix@=1pc{
& {\circ} \ar[dl] \ar[dr] & & {Q'} \ar[ll] \\
{Q''} \ar[dr] & & {\bullet} \ar[dl] \ar[ur] \\
& {\circ} \ar[uu]
}
\end{array}
& \mu^+_{\bullet}
& \text{good}
\\ \hline
%
\text{III.2} &
\begin{array}{c}
\xymatrix@=1pc{
& {\circ} \ar[dl] \\
{Q''} \ar[dr] & & {\bullet} \ar[ul] \ar[dr] \\
& {\circ} \ar[ur] & & {Q'}
}
\end{array}
& \mu^+_{\bullet} &
\begin{array}{c}
\xymatrix@=1pc{
& {\circ} \ar[dl] \ar[dr] \\
{Q''} \ar[dr] & & {\bullet} \ar[dl] \\
& {\circ} \ar[uu] \ar[rr] & & {Q'} \ar[ul]
}
\end{array}
& \mu^-_{\bullet}
& \text{good}
\\ \hline
%
\text{III.3} &
\begin{array}{c}
\xymatrix@=1pc{
& {\circ} \ar[dl] & & {Q'} \ar[dl] \\
{Q''} \ar[dr] & & {\bullet} \ar[ul] \ar[dr] \\
& {\circ} \ar[ur] & & {Q'''} \ar[uu]
}
\end{array}
& \mu^-_{\bullet}, \mu^+_{\bullet} &
\begin{array}{c}
\xymatrix@=1pc{
& {\circ} \ar[dl] \ar[dr] & & {Q'} \ar[ll] \\
{Q''} \ar[dr] & & {\bullet} \ar[dl] \ar[ur] \\
& {\circ} \ar[uu] \ar[rr] & & {Q'''} \ar[ul]
}
\end{array}
& \text{none}
& \text{bad}
\\ \hline
\end{array}
\]
\caption{Mutations involving Type III quivers.} \label{t:mutIII}
\end{table}

\begin{table}
\[
\begin{array}{|c|cc|cc|}
\hline
\text{IV.1a} &
\begin{array}{c}
\xymatrix@=0.3pc{
& & {\bullet} \ar[dll] \\
{\circ} \ar[dd] & & & & {\circ} \ar[ull] \ar[drr] \\
& & & & & & {Q} \ar[dll] \\
{\circ} \ar@/_0.9pc/@{.}[rrrr] & & & & {\circ} \ar[uu] \\
&
}
\end{array}
& \mu^-_{\bullet} &
\begin{array}{c}
\xymatrix@=0.3pc{
& & {\bullet} \ar[ddr] \\ \\
& {\circ} \ar[uur] \ar[ddl] & & {\circ} \ar[ll] \ar[rr]
& & {Q} \ar[ddl] \\ \\
{\circ} \ar@/_0.9pc/@{.}[rrrr] & & & & {\circ} \ar[uul] \\
&
}
\end{array}
& \mu^+_{\bullet}
\\ \hline
%
\text{IV.1b} &
\begin{array}{c}
\xymatrix@=0.3pc{
& & & & {\bullet} \ar[dll] \\
& & {\circ} \ar[dd] & & & & {\circ} \ar[ull] \\
{Q} \ar[urr] \\
& & {\circ} \ar[ull] \ar@/_0.9pc/@{.}[rrrr] & & & & {\circ} \ar[uu] \\
&
}
\end{array}
& \mu^+_{\bullet} &
\begin{array}{c}
\xymatrix@=0.3pc{
& & & {\bullet} \ar[ddr] \\ \\
{Q} \ar[rr] & & {\circ} \ar[uur] \ar[ddl] & & {\circ} \ar[ll] \\ \\
& {\circ} \ar[uul] \ar@/_0.9pc/@{.}[rrrr] & & & & {\circ} \ar[uul] \\
&
}
\end{array}
& \mu^-_{\bullet}
\\ \hline
%
\text{IV.2a} &
\begin{array}{c}
\xymatrix@=0.3pc{
{Q'} \ar[drr] \\
& & {\bullet} \ar[dll] \\
{\circ} \ar[uu] \ar@/_0.9pc/@{.}[ddrrrr]
& & & & {\circ} \ar[ull] \ar[drr] \\
& & & & & & {Q''} \ar[dll] \\
& & & & {\circ} \ar[uu] \\
&
}
\end{array}
& \mu^-_{\bullet} &
\begin{array}{c}
\xymatrix@=0.3pc{
{Q'} & & {\bullet} \ar[ll] \ar[ddr] \\ \\
& {\circ} \ar[uur] \ar@/_0.9pc/@{.}[ddrrr] & & {\circ} \ar[ll] \ar[rr]
& & {Q''} \ar[ddl] \\ \\
& & & & {\circ} \ar[uul]
}
\end{array}
& \mu^+_{\bullet}
\\ \hline
%
\text{IV.2b} &
\begin{array}{c}
\xymatrix@=0.3pc{
& & & & & & {Q'} \ar[dd] \\
& & & & {\bullet} \ar[dll] \ar[urr] \\
& & {\circ} \ar[dd] & & & & {\circ} \ar[ull] \\
{Q''} \ar[urr] \\
& & {\circ} \ar[ull] \ar@/_0.9pc/@{.}[uurrrr] \\
&
}
\end{array}
& \mu^+_{\bullet} &
\begin{array}{c}
\xymatrix@=0.3pc{
& & & {\bullet} \ar[ddr] & & {Q'} \ar[ll] \\ \\
{Q''} \ar[rr] & & {\circ} \ar[uur] \ar[ddl] & & {\circ} \ar[ll] \\ \\
& {\circ} \ar[uul] \ar@/_0.9pc/@{.}[uurrr]
}
\end{array}
& \mu^-_{\bullet}
\\ \hline
\end{array}
\]
\caption{Good mutations involving Type IV quivers.} \label{t:mutIV}
\end{table}

\begin{lemma} \label{l:IV0c}
Let $m \geq 2$ and consider a cluster-tilted algebra $\gL$ of Type IV
with the quiver
\begin{equation} \label{e:IV0c}
\begin{array}{c}
\xymatrix@=0.5pc{
& & & & {\bullet_0} \ar[dll] \\
& & {\bullet_1} \ar[dd] & & & & {\bullet_m} \ar[ull] \ar@{-->}[drr] \\
{Q_+} \ar@{-->}[urr] & & & & & & & & {Q_-} \ar@{-->}[dll] \\
& & {\bullet_2} \ar@{-->}[ull] \ar@/_1pc/@{.}[rrrr]
& & & & {\bullet} \ar[uu] \\
&
}
\end{array}
\end{equation}
having a central cycle $0, 1, \dots, m$ and optional spikes $Q_-$ and
$Q_+$ (which coincide when $m=2$). Then:
\begin{enumerate}
\renewcommand{\theenumi}{\alph{enumi}}
\renewcommand{\labelenumi}{(\theenumi)}
\item
$\mu^-_0(\gL)$ is defined if and only if the spike $Q_-$ is present.

\item
$\mu^+_0(\gL)$ is defined if and only if the spike $Q_+$ is present.
\end{enumerate}
\end{lemma}
\begin{proof}
We prove only the first assertion, as the proof of the second is
similar.

We use the criterion of Proposition~\ref{p:critmut}. The negative mutation
$\mu^-_0(\gL)$ is defined if and only if the composition of the arrow $m
\to 0$ with any non-zero path starting at $0$ is not zero. This holds for
all such paths of length smaller than $m-1$, so we only need to
consider the path $0, 1, \dots, m-1$. Now, the composition $m, 0, 1,
\dots, m-1$ vanishes if $Q_-$ is not present, and otherwise equals the
(non-zero) path $m, v_-, m-1$ where $v_-$ denotes the root of $Q_-$.
\end{proof}

\begin{lemma} \label{l:IV0s}
Let $m \geq 3$ and consider a cluster-tilted algebra $\gL'$ of Type IV
with the quiver
\begin{equation} \label{e:IV0s}
\begin{array}{c}
\xymatrix@=0.5pc{
& & & {\bullet_0} \ar[ddr] \\ \\
{Q_+} \ar@{-->}[rr] & & {\bullet_1} \ar[uur] \ar[ddl]
& & {\bullet_m} \ar[ll] \ar@{-->}[rr] & & {Q_-} \ar@{-->}[ddl] \\ \\
& {\bullet_2} \ar@{-->}[uul] \ar@/_1pc/@{.}[rrrr] & & & &
{\bullet} \ar[uul] \\
&
}
\end{array}
\end{equation}
having a central cycle $1, 2, \dots, m$ and optional spikes $Q_-$ and $Q_+$.
Then:
\begin{enumerate}
\renewcommand{\theenumi}{\alph{enumi}}
\renewcommand{\labelenumi}{(\theenumi)}
\item
$\mu^-_0(\gL')$ is defined if and only if the spike $Q_-$ is not
present.

\item
$\mu^+_0(\gL')$ is defined if and only if the spike $Q_+$ is not
present.
\end{enumerate}
\end{lemma}
\begin{proof}
We prove only the first assertion, as the proof of the second is
similar.

We use the criterion of Proposition~\ref{p:critmut}. The negative
mutation $\mu^-_0(\gL')$ is defined if and only if the composition of
the arrow $1 \to 0$ with any non-zero path starting at $0$ is not zero.
For the path $0, m$, the composition $1, 0, m$ equals the path $1, 2,
\dots, m$ and hence it is non-zero. This shows that $\mu^-_0(\gL')$ is
defined when $Q_-$ is not present. When $Q_-$ is present, the path $0,
m, v_-$ to the root $v_-$ of $Q_-$ is non-zero, but the composition $1,
0, m, v_-$ equals the path $1, 2, \dots, m, v_-$ which is zero since
the path $m-1, m, v_-$ vanishes
\end{proof}

\begin{cor} \label{c:IV0}
Let $\gL$ be a cluster-tilted algebra corresponding to a quiver as
in~\eqref{e:IV0c} and let $\gL'$ be the one corresponding to its mutation
at $0$, as in~\eqref{e:IV0s}. The following table lists which of the algebra
mutations at $0$ are defined for $\gL$ and $\gL'$ depending on whether
the optional spikes $Q_-$ or $Q_+$ are present (``yes'') or not (``no'').
\[
\begin{array}{c|c|c|c|c}
Q_- & Q_+ & \gL & \gL' \\
\hline
\text{yes} & \text{yes} & \mu^-, \mu^+ & \text{none} & \text{bad}\\
\text{yes} & \text{no}  & \mu^- & \mu^+ & \text{good} \\
\text{no}  & \text{yes} & \mu^+ & \mu^- & \text{good} \\
\text{no}  & \text{no}  & \text{none} & \mu^-, \mu^+ & \text{bad}
\end{array}
\]
\end{cor}

\begin{lemma} \label{l:IV1c}
Let $m \geq 2$ and consider cluster-tilted algebras $\gL_-$ and $\gL_+$
of Type IV with the following quivers
\begin{align} \label{e:IV1c}
\begin{array}{c}
\xymatrix@=0.5pc{
{Q_0} \ar[drr] \\
& & {\bullet_0} \ar[dll] \\
{\bullet_1} \ar[uu] \ar@/_1pc/@{.}[ddrrrr]
& & & & {\bullet_m} \ar[ull] \ar@{-->}[drr] \\
& & & & & & {Q_-} \ar@{-->}[dll] \\
& & & & {\bullet} \ar[uu]
}
\end{array}
& &
\begin{array}{c}
\xymatrix@=0.5pc{
& & & & & & {Q_0} \ar[dd] \\
& & & & {\bullet_0} \ar[dll] \ar[urr] \\
& & {\bullet_1} \ar[dd] & & & & {\bullet_m} \ar[ull] \\
{Q_+} \ar@{-->}[urr] \\
& & {\bullet_2} \ar@{-->}[ull] \ar@/_1pc/@{.}[uurrrr]
}
\end{array}
\end{align}
having a central cycle $0, 1, \dots, m$ and optional spikes $Q_-$ or $Q_+$.
Then:
\begin{enumerate}
\renewcommand{\theenumi}{\alph{enumi}}
\renewcommand{\labelenumi}{(\theenumi)}
\item
$\mu^+_0(\gL_-)$ is never defined;

\item
$\mu^-_0(\gL_-)$ is defined if and only if the spike $Q_-$ is present;

\item
$\mu^-_0(\gL_+)$ is never defined;

\item
$\mu^+_0(\gL_+)$ is defined if and only if the spike $Q_+$ is present.
\end{enumerate}
\end{lemma}
\begin{proof}
We prove only the first two assertions, the proofs of the 
others are similar.
\begin{enumerate}
\renewcommand{\theenumi}{\alph{enumi}}
\item
Let $v_0$ denote the root of $Q_0$. Then the path $v_0, 0$ is non-zero
whereas $v_0, 0, 1$ is zero.

\item
Since the path $v_0, 0, 1$ is zero, the composition of the arrow $v_0 \to 0$
with any non-trivial path starting at $0$ is zero.
Therefore the negative mutation at $0$ is defined if and only if the
composition of the arrow $m \to 0$ with any non-zero path starting at $0$
is not zero, and the proof goes in the same manner as in Lemma~\ref{l:IV0c}.
\end{enumerate}
\end{proof}

\begin{lemma} \label{l:IV1s}
Let $m \geq 3$ and consider cluster-tilted algebras $\gL'_-$ and
$\gL'_+$ of Type IV with the following quivers
\begin{align} \label{e:IV1s}
\begin{array}{c}
\xymatrix@=0.5pc{
{Q_0} & & {\bullet_0} \ar[ll] \ar[ddr] \\ \\
& {\bullet_1} \ar[uur] \ar@/_1pc/@{.}[ddrrr] & & {\bullet_m} \ar[ll]
\ar@{-->}[rr] & & {Q_-} \ar@{-->}[ddl] \\ \\
& & & & {\bullet} \ar[uul]
}
\end{array}
& &
\begin{array}{c}
\xymatrix@=0.5pc{
& & & {\bullet_0} \ar[ddr] & & {Q_0} \ar[ll] \\ \\
{Q_+} \ar@{-->}[rr] & & {\bullet_1} \ar[uur] \ar[ddl]
& & {\bullet_m} \ar[ll] \\ \\
& {\bullet_2} \ar@{-->}[uul] \ar@/_1pc/@{.}[uurrr]
}
\end{array}
\end{align}
having a central cycle $1, \dots, m$ and optional spikes $Q_-$ or $Q_+$.
Then:
\begin{enumerate}
\renewcommand{\theenumi}{\alph{enumi}}
\renewcommand{\labelenumi}{(\theenumi)}
\item
$\mu^+_0(\gL'_-)$ is always defined;

\item
$\mu^-_0(\gL'_-)$ is defined if and only if the spike $Q_-$ is not
present;

\item
$\mu^-_0(\gL'_+)$ is always defined;

\item
$\mu^+_0(\gL'_+)$ is defined if and only if the spike $Q_+$ is not
present.
\end{enumerate}
\end{lemma}
\begin{proof}
We prove only the first two assertions, the proofs of the others are similar.
\begin{enumerate}
\renewcommand{\theenumi}{\alph{enumi}}
\item
Let $v_0$ denote the root of $Q_0$. Then the composition of any non-zero path
ending at $0$ with the arrow $0 \to v_0$ is not zero.

\item
Since the composition of the arrow $1 \to 0$ with any non-zero path
whose first arrow is $0 \to v_0$ is not zero, we only need to consider paths
whose first arrow $0 \to m$. The proof is then the same as in
Lemma~\ref{l:IV0s}.
\end{enumerate}
\end{proof}

\begin{cor} \label{c:IV1}
Let $\gL_-$ and $\gL_+$ be cluster-tilted algebras corresponding to quivers as
in~\eqref{e:IV1c} and let $\gL'_-$ and $\gL'_+$ be the ones corresponding to
their mutations at $0$, as in~\eqref{e:IV1s}.
The following tables list which of the algebra
mutations at $0$ are defined for $\gL_-$, $\gL'_-$, $\gL_+$ and $\gL'_+$
depending on whether the optional spikes $Q_-$ or $Q_+$ are present (``yes'')
or not (``no'').
\begin{align*}
\begin{array}{c|c|c|c}
Q_- & \gL_- & \gL'_- \\
\hline
\text{yes} & \mu^- & \mu^+ & \text{good} \\
\text{no}  & \text{none} & \mu^-, \mu^+ & \text{bad}
\end{array}
& &
\begin{array}{c|c|c|c}
Q_+ & \gL_+ & \gL'_+ \\
\hline
\text{yes} & \mu^+ & \mu^- & \text{good} \\
\text{no}  & \text{none} & \mu^-, \mu^+ & \text{bad}
\end{array}
\end{align*}
\end{cor}

\begin{lemma} \label{l:IV2c}
Let $m \geq 2$ and consider a cluster-tilted algebra $\gL$ of Type IV
with the following quiver
\[
\xymatrix@=0.5pc{
{Q''} \ar[drr] & & & & {Q'} \ar[dd] \\
& & {\bullet_0} \ar[dll] \ar[urr] \\
{\bullet_1} \ar[uu] \ar@/_1pc/@{.}[rrrr] & & & & {\bullet_m} \ar[ull] \\
&
}
\]
having a central cycle $0, 1, \dots, m$. Then the algebra mutations
$\mu^-_0(\gL)$ and $\mu^+_0(\gL)$ are never defined.
\end{lemma}
\begin{proof}
Denote by $v'$, $v''$ the roots of $Q'$ and $Q''$, respectively, and
consider the path $0, 1, \dots, m$. It is non-zero, since it equals the path
$0, v', m$. However, its composition with the arrow $v'' \to 0$ is zero
since the path $v'', 0, 1$ vanishes, and its composition with the arrow
$m \to 0$ is zero as well, since it equals $m, 0, v', m$ and the
path $m, 0, v'$ vanishes. By Proposition~\ref{p:critmut}, the mutation
$\mu^-_0(\gL)$ is not defined.
The proof for $\mu^+_0(\gL)$ is similar.
\end{proof}

\section{Further derived equivalences in Types III and IV}
\label{sec:further-derived}

\subsection{Good double mutations in Types III and IV}

The good double mutations we consider in this section consist of two
algebra mutations. The first takes a cluster-tilted algebra $\gL$ to a
derived equivalent algebra which is not cluster-tilted, whereas the
second takes that algebra to another cluster-tilted algebra $\gL'$,
thus obtaining a derived equivalence of $\gL$ and $\gL'$. As already
demonstrated in Example~\ref{ex:D8}, these derived equivalences cannot
in general be obtained by performing sequences consisting of only good
mutations.

\begin{lemma} \label{l:IV2s}
Let $m \geq 3$ and consider a cluster-tilted algebra $\gL = \gL_{\wt{Q}}$
of Type IV with the quiver $\wt{Q}$ as in the left picture
\begin{align*}
\begin{array}{c}
\xymatrix@=0.5pc{
& & {Q''} \ar[rr] & & {Q'} \ar[ddl] \\ \\
& & & {\bullet_0} \ar[uul] \ar[ddr] \\ \\
{Q_-} \ar@{-->}[rr] & & {\bullet_1} \ar[uur] \ar[ddl]
& & {\bullet_m} \ar[ll] \ar@{-->}[rr] & & {Q_+} \ar@{-->}[ddl] \\ \\
& {\bullet_2} \ar@{-->}[uul] \ar@/_1pc/@{.}[rrrr]
& & & & {\bullet} \ar[uul]
}
\end{array}
& &
\begin{array}{c}
\xymatrix@=0.5pc{
& & {Q''} \ar[drr] & & & & {Q'} \ar[dd] \\
& & & & {\bullet_0} \ar[dll] \ar[urr] \\
& & {\bullet_1} \ar[uu] \ar[dd]
& & & & {\bullet_m} \ar[ull] \ar@{-->}[drr] \\
{Q_-} \ar@{-->}[urr] & & & & & & & & {Q_+} \ar@{-->}[dll] \\
& & {\bullet_2} \ar@{-->}[ull] \ar@/_1pc/@{.}[rrrr]
& & & & {\bullet} \ar[uu]
}
\end{array}
\end{align*}
having a central cycle $1, \dots, m$ and optional spikes $Q_-$ and
$Q_+$. Let $\mu_0(\wt{Q})$ denote the mutation of $\wt{Q}$ at the
vertex $0$, as in the right picture. Then:
\begin{enumerate}
\renewcommand{\theenumi}{\alph{enumi}}
\renewcommand{\labelenumi}{(\theenumi)}
\item
$\mu^-_0(\gL)$ is always defined and is isomorphic to the quotient of
the cluster-tilted algebra $\gL_{\mu_0(\wt{Q})}$ by the ideal generated
by the path $p$ given by
\[
p =
\begin{cases}
1, 2, \dots, m, 0 & \text{if the spike $Q_-$ is present,} \\
2, \dots, m, 0    & \text{otherwise.}
\end{cases}
\]
\item
$\mu^+_0(\gL)$ is always defined and is isomorphic to the quotient of
the cluster-tilted algebra $\gL_{\mu_0(\wt{Q})}$ by the ideal generated
by the path $p$ given by
\[
p =
\begin{cases}
0, 1, \dots, m   & \text{if the spike $Q_+$ is present,} \\
0, 1, \dots, m-1 & \text{otherwise}.
\end{cases}
\]
\end{enumerate}
\end{lemma}
\begin{proof}
We prove only the first assertion and leave the second to the reader.
Let $\gL = \gL_{\wt{Q}}$ be the cluster-tilted algebra corresponding to
the quiver $\wt{Q}$ depicted as
\begin{center}
\includegraphics[scale=0.75]{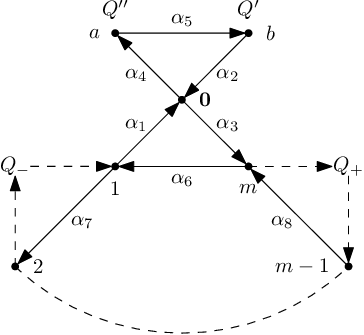}
\end{center}

It is easily seen using Proposition~\ref{p:critmut} that the negative
mutation $\mu^-_0(\gL)$ is defined. In order to describe it explicitly,
we recall that $\mu^-_0(\gL) = \End_{\cD^b(\gL)}(T^-_0(\gL))$, where
\[
T^-_0(\gL) = \bigl( P_0 \xrightarrow{(\alpha_1, \alpha_2)} (P_1 \oplus
P_b) \bigr) \oplus \bigl( \bigoplus_{i \neq 0} P_i \bigr) = L_0 \oplus
\bigl( \bigoplus_{i \neq 0} P_i \bigr) .
\]

Using an alternating sum formula of Happel \cite{Happel} we can compute
the Cartan matrix of $\mu^-_0(\gL)$ to be
$$C_{\mu^-_0(\gL)} = \arraycolsep0.4em\begin{array}[c]{c|*{5}{c}:ccc}
  & 0 & 1 & m & a & b & 2 & (m-1) & \cdots\\
\hline
0 & 1 & 1 & 1 & 0 & 1 & 1 & 1 & \cdots\\
1 & \framebox{0} & 1 & 1 & 1 & 0 & 1 & 1 & \cdots\\
m & 1 & 1 & 1 & 0 & 0 & 1 & ? & \cdots\\
a & 1 & 0 & 0 & 1 & 1 & 0 & 0 & \cdots\\
b & 0 & 0 & 1 & 0 & 1 & 0 & 0 & \cdots\\
\hdashline
2 & 1/\framebox{0} & 1/0 & 1 & 0 & 0 & 1 & 1 & \cdots\\
(m-1) & 1 & 1 & 1 & 0 & 0 & ? & 1 & \cdots\\
\vdots & \vdots & \vdots & \vdots & \vdots & \vdots & \vdots & \vdots
\end{array} $$
where $1/0$ means $1$ if $Q_-$ is present and $0$ if $Q_-$ is not present.

Now to each arrow of the following quiver we define a homomorphism of
complexes between the summands of $T^-_0(\gL)$.
\begin{center}
\includegraphics[scale=0.75]{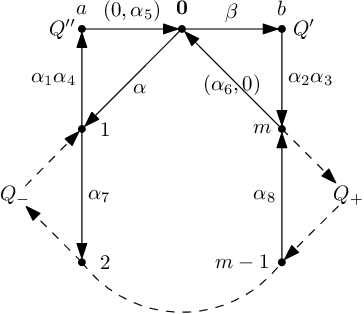}
\end{center}

First we have the embeddings $\alpha:=(\mathrm{id},0): P_1 \to L_0$ and
$\beta :=(0, \mathrm{id}): P_b \to L_0$ (in degree zero). Moreover, we
have the homomorphisms $\alpha_1\alpha_4 : P_a \to P_1$,
$\alpha_2\alpha_3 : P_m \to P_b$, $(\alpha_6,0) : L_0 \to P_m$ and
$(0,\alpha_5) : L_0 \to P_a$. All the other homomorphisms are as
before.

Now we have to show that these homomorphisms satisfy the defining
relations of the algebra $\gL_{\mu_0(\wt{Q})}/I(p)$, up to homotopy,
where $I(p)$ is the ideal generated by the path $p$ stated in the
lemma. Clearly, the concatenation of $(0, \alpha_5)$ and $\alpha$ and
the concatenation of $(\alpha_6,0)$ and $\beta$ are zero-relations. The
concatenation of $\alpha_2\alpha_3$ and $(\alpha_6,0)$ is zero as
before. It is easy to see that the two paths from vertex $0$ to vertex
$m$ are the same since $(0,\alpha_2\alpha_3)$ is homotopic to
$(\alpha_7\dots\alpha_8,0)$ (and $\alpha_7\dots\alpha_8 =
\alpha_1\alpha_3$ in $\gL$). The path from vertex $0$ to vertex $a$ is
zero since $(\alpha_1\alpha_4,0)$ is homotopic to zero. There is no
non-zero path from vertex $1$ to vertex $0$ since
$(0,\alpha_1\alpha_4\alpha_5) = 0 = (\alpha_7 \dots
\alpha_8\alpha_6,0)$. This corresponds to the path $p$ in the case if
$Q_-$ is present and is marked in the Cartan matrix by a box. If $Q_-$
is not present then the path from vertex $2$ to vertex $0$ is already
zero since $(\dots \alpha_8\alpha_6,0) = 0$ which is also marked in the
Cartan matrix above. Thus, $\mu^-_0(\gL)$ is isomorphic to the quotient
of the cluster-tilted algebra $\gL_{\mu_0(\wt{Q})}$ by the ideal
generated by the path $p$.
\end{proof}

\begin{cor} \label{c:IVdbl}
The two cluster-tilted algebras with quivers
\begin{align*}
\begin{array}{c}
\xymatrix@=0.5pc{
& & {Q''} \ar[rr] & & {Q'} \ar[ddl] \\ \\
& & & {\bullet_0} \ar[uul] \ar[ddr] \\ \\
{Q'''} \ar[rr] & & {\bullet_1} \ar[uur] \ar[ddl]
& & {\bullet_m} \ar[ll] \\ \\
& {\bullet_2} \ar[uul] \ar@/_1pc/@{.}[uurrr]
}
\end{array}
& &
\begin{array}{c}
\xymatrix@=0.5pc{
{Q'''} \ar[rr] & & {Q''} \ar[ddl] \\ \\
& {\bullet_1} \ar[uul] \ar[ddr] \\ \\
{\bullet_2} \ar[uur] \ar@/_1pc/@{.}[ddrrr]
& & {\bullet_0} \ar[ll] \ar[rr] & & {Q'} \ar[ddl] \\ \\
& & & {\bullet_m} \ar[uul]
}
\end{array}
\end{align*}
(where $Q'$, $Q''$ and $Q'''$ are rooted quivers of type $A$)
are related by a good double mutation (at the vertex $0$ and then at
$1$).
\end{cor}
\begin{proof}
Denoting the left algebra by $\gL_L$ and the right one by $\gL_R$, we
see that $\mu^-_0(\gL_L) \simeq \mu^+_1(\gL_R)$, as by
Lemma~\ref{l:IV2s} these algebra mutations are isomorphic to quotient
of the cluster-tilted algebra of the quiver
\[
\xymatrix@=0.5pc{
& & & {Q''} \ar[ddr] \\ \\
{Q'''} \ar[rr] & & {\bullet_1} \ar[uur] \ar[ddl]
& & {\bullet_0} \ar[ll] \ar[rr] & & {Q'} \ar[ddl] \\ \\
& {\bullet_2} \ar[uul] \ar@/_1pc/@{.}[rrrr] & & & &
{\bullet_m} \ar[uul] \\
&
}
\]
by the ideal generated by the path $1, 2, \dots, m, 0$.
\end{proof}

There is an analogous version of Lemma~\ref{l:IV2s} for cluster-tilted
algebras in Type III, corresponding to the case where $m=2$, and the
spikes $Q_-$ and $Q_+$ coincide (and are present).
That is, there is a central cycle of length $m=2$ (hence it is
``invisible'') with all spikes present.

\begin{lemma}
Consider the cluster-tilted algebra $\gL_{\wt{Q}}$ of Type III
whose quiver $\wt{Q}$ is shown in the left picture, where $Q'$, $Q''$
and $Q'''$ are rooted quivers of type $A$.
\begin{align*}
\begin{array}{c}
\xymatrix@=0.5pc{
& & {\bullet_1} \ar[drr] & & & & {Q''} \ar[dd] \\
{Q'''} \ar[urr] & & & & {\bullet_0} \ar[dll] \ar[urr] \\
& & {\bullet_2} \ar[ull] & & & & {Q'} \ar[ull]
}
\end{array}
& &
\begin{array}{c}
\xymatrix@=0.5pc{
& & & & {Q''} \ar[dd] \\
& & {\bullet_1} \ar[urr] \ar[dd]_{\beta} \\
{Q'''} \ar[urr] & & & & {\bullet_0} \ar[ull]_{\alpha} \ar[dd] \\
& & {\bullet_2} \ar[urr]_{\gamma} \ar[ull] \\
& & & & {Q'} \ar[ull]
}
\end{array}
\end{align*}
Let $\mu_0(\wt{Q})$ denote the mutation of $\wt{Q}$ at the
vertex $0$, as in the right picture. Then:
\begin{enumerate}
\renewcommand{\theenumi}{\alph{enumi}}
\item
$\mu^-_0(\gL)$ is always defined and is isomorphic to the quotient of
the cluster-tilted algebra $\gL_{\mu_0(\wt{Q})}$ by the ideal generated
by the path $\beta \gamma$.

\item
$\mu^+_0(\gL)$ is always defined and is isomorphic to the quotient of
the cluster-tilted algebra $\gL_{\mu_0(\wt{Q})}$ by the ideal generated
by the path $\alpha \beta$.
\end{enumerate}
\end{lemma}

\begin{cor} \label{c:IIIdbl}
The cluster-tilted algebras of Type III with quivers
\begin{align*}
\xymatrix@=0.5pc{
& & {\bullet_1} \ar[drr] & & & & {Q''} \ar[dd] \\
{Q'''} \ar[urr] & & & & {\bullet_0} \ar[dll] \ar[urr] \\
& & {\bullet_2} \ar[ull] & & & & {Q'} \ar[ull]
}
& &
\xymatrix@=0.5pc{
{Q''} \ar[drr] & & & & {\bullet_0} \ar[drr] \\
& & {\bullet_1} \ar[dll] \ar[urr] & & & & {Q'} \ar[dll] \\
{Q'''} \ar[uu] & & & & {\bullet_2} \ar[ull]
}
\end{align*}
(where $Q'$, $Q''$ and $Q'''$ are rooted quivers of type $A$)
are related by a good double mutation (at $0$ and then at $1$).
\end{cor}

\subsection{Self-injective cluster-tilted algebras}
The self-injective cluster-tilted algebras have been determined by
Ringel in~\cite{Ringel}. They are all of Dynkin type $D_n$, $n \geq 3$.
Fixing the number $n$ of vertices, there are one or two such algebras
according to whether $n$ is odd or even. Namely, there is the algebra
corresponding to the cycle of length $n$ without spikes, and when
$n=2m$ is even, there is also the one of Type IV with parameter
sequence $\bigl((1,0,0), (1,0,0), \dots, (1,0,0)\bigr)$ of length $m$.

The following lemma shows that these two algebras are in fact derived
equivalent. Note that this could also be deduced from the derived
equivalence classification of self-injective algebras of finite
representation type~\cite{Asashiba}.

\begin{lemma} \label{l:selfinj}
Let $m \geq 3$. Then the cluster-tilted algebra of Type IV with a
central cycle of length $2m$ without any spike is derived equivalent to
that in Type IV with parameter sequence $\bigl((1,0,0), (1,0,0), \dots,
(1,0,0)\bigr)$ of length $m$.
\end{lemma}
\begin{proof}
Let $\gL$ be the cluster-tilted algebra corresponding to a cycle of
length $2m$ as in the left picture.
\begin{align*}
\includegraphics[scale=0.75]{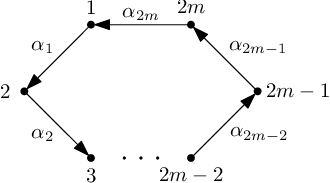}
& &
\includegraphics[scale=0.75]{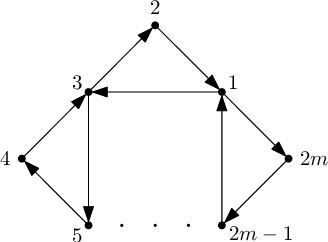}
\end{align*}
We leave it to the reader to verify that the following complex of
projective $\gL$-modules
\[
T = \Bigl( \bigoplus_{i=1}^m \bigl( P_{2i} \xrightarrow{\alpha_{2i-1}}
P_{2i-1} \bigr) \Bigr) \oplus \bigl( \bigoplus_{i=1}^m P_{2i-1} \bigr)
\]
(where the terms $P_{2i-1}$ are always at degree $0$) is a tilting
complex whose endomorphism algebra $\End_{\cD^b(\gL)}(T)$ is isomorphic
to the cluster-tilted algebra whose quiver is given in the right
picture.
\end{proof}

\section{Algorithms and standard forms}
\label{sec:algstdform}

In this section we provide standard forms for derived equivalence
(Theorem~\ref{t:stdder}) as well as ones for good mutation equivalence
(Theorem~\ref{t:stdgood}) of cluster-tilted algebras of type $D$. We
also describe an explicit algorithm which decides on good mutation
equivalence (Theorem~\ref{t:alggood}).

\subsection{Good mutations and good 
double mutations in parametric notation}
We start by describing all the good mutations determined in
Section~\ref{sec:goodmut} using the parametric notation of
Section~\ref{sec:ctaAD} which will be useful in the sequel. Note that
by Proposition~\ref{p:goodmutA} two quivers with the same type and
parameters are indeed equivalent by good mutations so this notation
makes sense.

Each row of Table~\ref{t:mutgood} describes a good mutation between two
quivers of cluster-tilted algebras of type $D$ given in parametric
notation. The numbers $s',s'',s''',t',t'',t'''$ are arbitrary 
non-negative integers
and correspond to the parameters of the rooted quivers $Q'$, $Q''$,
$Q'''$ of type $A$ appearing in the corresponding pictures (referenced
by the column ``Move'').

\begin{table}
\begin{center}
\begin{tabular}{|c|cc|cc|c|}
\hline
Move & Type & Parameters & Type & Parameters & Remarks\\
\hline
I.5a &
I & $(s'+s'',t'+t''+1)$ &
II & $(s'+1,t',s'',t'')$ & \\
I.5b &
I & $(s'+s'',t'+t''+1)$ &
II & $(s',t',s''+1,t'')$ & \\
II.2 &
II & $(s'+1, t', s'', t'')$ &
II & $(s',t',s''+1,t'')$ & \\
II.3 &
II & $(s'+s''', t'+t'''+1, s'', t'')$ &
II & $(s',t',s''+s''', t''+t'''+1)$ & \\
III.1 &
III & $(s'+1, t', s'', t'')$ &
IV & $\bigl((2,s',t'),(1,s'',t'')\bigr)$ & \\
III.2 &
III & $(s'+1, t', s'', t'')$ &
IV & $\bigl((1,s',t'),(2,s'',t'')\bigr)$ & \\
IV.1a &
IV & $\bigl((d_1,s_1,t_1), (d_2,s_2,t_2),\dots \bigr)$ &
IV & $\bigl((1,s_1,t_1),(d_1-2,0,0),(d_2,s_2,t_2),\dots \bigr)$ &
$d_1 \geq 4$ \\
IV.1b &
IV & $\bigl((d_1,s_1,t_1), (d_2,s_2,t_2),\dots \bigr)$ &
IV & $\bigl((d_1-2,s_1,t_1),(1,0,0),(d_2,s_2,t_2),\dots \bigr)$ &
$d_1 \geq 4$ \\
IV.2a &
IV & $\bigl((2,s_1,t_1), (d_2,s_2,t_2), \dots \bigr)$ &
IV & $\bigl((1,s_1,t_1), (d_2,s_2+1,t_2), \dots \bigr)$ & \\
IV.2b &
IV & $\bigl((2,s_1,t_1), (d_2,s_2,t_2), \dots \bigr)$ &
IV & $\bigl((1,s_1+1,t_1), (d_2,s_2,t_2), \dots \bigr)$ & \\
\hline
\end{tabular}
\end{center}

\caption{All good mutations in parametric notation.} \label{t:mutgood}
\end{table}

By looking at the first four rows of the table we immediately draw the
following conclusions.

\begin{lemma} \label{l:IImutgood}
Consider quivers of Type I or II.
\begin{enumerate}
\renewcommand{\theenumi}{\alph{enumi}}
\item
The subset consisting of the quivers of Type I or II is closed under
good mutations.
\item
The subset consisting of all the orientations of a $D_n$ diagram is
closed under good mutations
\item
A quiver in Type I with parameters $(s,t+1)$ for some $s,t \geq 0$ is
equivalent by good mutations to one in Type II with parameters
$(s+1,t,0,0)$.
\item
A quiver in Type II with parameters $(s',t',s'',t'')$ is equivalent by
good mutations to one in Type II with parameters $(s'+s'',t'+t'',0,0)$.
\item
Two quivers in Type II with parameters $(s_1,t_1,0,0)$ and
$(s_2,t_2,0,0)$ are equivalent by good mutations if and only if
$s_1=s_2$ and $t_1=t_2$.
\end{enumerate}
\end{lemma}

\begin{lemma} \label{l:IIImutgood}
Consider quivers of Type III.
\begin{enumerate}
\renewcommand{\theenumi}{\alph{enumi}}
\item
A quiver of Type III with parameters $(s'+1,t',s'',t'')$ is good
mutation equivalent to one of Type III with parameters
$(s',t',s''+1,t'')$.

\item
A quiver of Type III with parameters $(s',t',s'',t'')$ is good mutation
equivalent to one of Type III with parameters $(s'+s'',t',0,t'')$.
\end{enumerate}
\end{lemma}
\begin{proof}
\begin{enumerate}
\renewcommand{\theenumi}{\alph{enumi}}
\item
We have
\begin{align*}
\text{III}(s'+1,t',s'',t'') &\xrightarrow{\text{III.1}} \text{IV}
\bigl( (2,s',t'),(1,s'',t'') \bigr) \simeq \text{IV} \bigl(
(1,s'',t''), (2,s',t') \bigr) \xrightarrow{\text{III.2}}
\text{III}(s''+1,t'',s',t')
\\
&\simeq \text{III}(s',t',s''+1,t'')
\end{align*}
where the isomorphisms follow from rotational symmetries.

\item
Follows from the first part.
\end{enumerate}
\end{proof}

\begin{remark}
The previous lemma shows that in Type III, it is possible to move
linear parts in the rooted quivers of type $A$ from side to side by
using good mutations. It is not possible, however, to move triangles by
good mutations (see III.3 in Table~\ref{t:mutIII} and
Example~\ref{ex:D8}).
\end{remark}

For the next two lemmas we need the following terminology on spikes in
quivers of Type IV. Spikes are \emph{consecutive} if the distance
$d_i$ between them is $1$. A spike is \emph{free} if the attached
rooted quiver of type $A$ consists of just a single vertex.

\begin{lemma} \label{l:IVfreespike}
Consider quivers of Type IV. A free spike at the end of a group of at
least two consecutive spikes can be moved by good mutations to the next
group of consecutive spikes. In other words, the two quivers with
parameters $\bigl((1,s,t), (d,0,0), \dots \bigr)$ and $\bigl((d,s,t),
(1,0,0), \dots \bigr)$ are connected by good mutations.
\end{lemma}
\begin{proof}
We assume $d \geq 2$, otherwise there is nothing to prove. Then
\[
\bigl((1,s,t), (d,0,0), \dots \bigr) \xrightarrow{\text{IV.1a}} \bigl(
(d+2,s,t), \dots \bigr) \xrightarrow{\text{IV.1b}} \bigl((d,s,t),
(1,0,0), \dots \bigr).
\]
\end{proof}

\begin{lemma} \label{l:IVgroupline}
Lines in a rooted quiver of type $A$ attached to a spike in a group of
consecutive spikes in a quiver of Type IV can be moved by good
mutations to a rooted quiver attached to any spike in that group.
\end{lemma}
\begin{proof}
It suffices to show that the two quivers with parameters
$$
\bigl((1,s_1,t_1), (d_2, s_2+1, t_2), \dots \bigr) \mbox{~~~and~~~}
\bigl((1,s_1+1,t_1), (d_2, s_2, t_2), \dots \bigr)
$$
are good mutation equivalent. Indeed,
\[
\bigl((1,s_1,t_1), (d_2, s_2+1, t_2), \dots \bigr)
\xrightarrow{\text{IV.2a}} \bigl( (2, s_1, t_1), (d_2, s_2, t_2), \dots
\bigr) \xrightarrow{\text{IV.2b}} \bigl( (1,s_1+1,t_1), (d_2, s_2,
t_2), \dots \bigr).
\]
In pictures,
\[
\begin{array}{ccccc}
\begin{array}{c}
\xymatrix@=0.3pc{
{Q'} & & {\bullet} \ar[ll] \ar[ddr] \\ \\
& {\cdot} \ar[uur] \ar@/_0.9pc/@{.}[ddrrr] & & {\circ} \ar[ll] \ar[rr]
& & {Q''} \ar[ddl] \\ \\
& & & & {\cdot} \ar[uul]
}
\end{array}
&
\xrightarrow[\mu^+_{\bullet}]{\text{IV.2a}}
&
\begin{array}{c}
\xymatrix@=0.3pc{
{Q'} \ar[drr] \\
& & {\bullet} \ar[dll] \\
{\cdot} \ar[uu] \ar@/_0.9pc/@{.}[ddrrrr]
& & & & {\circ} \ar[ull] \ar[drr] \\
& & & & & & {Q''} \ar[dll] \\
& & & & {\cdot} \ar[uu] \\
&
}
\end{array}
&
\xrightarrow[\mu^+_{\circ}]{\text{IV.2b}}
&
\begin{array}{c}
\xymatrix@=0.3pc{
& & & {\circ} \ar[ddr] & & {Q''} \ar[ll] \\ \\
{Q'} \ar[rr] & & {\bullet} \ar[uur] \ar[ddl] & & {\cdot} \ar[ll] \\ \\
& {\cdot} \ar[uul] \ar@/_0.9pc/@{.}[uurrr]
}
\end{array}
\end{array}
\]
\end{proof}

We now turn to good double mutations as determined in
Section~\ref{sec:further-derived}. They are presented in parametric
notation in Table~\ref{t:mutdblgood}, based on Corollaries~\ref{c:IVdbl}
and~\ref{c:IIIdbl}. Using them, we can achieve the following further
transformations of quivers in Types III and IV described in the next
lemmas.

\begin{table}
\begin{center}
\begin{tabular}{|cc|cc|}
\hline
Type & Parameters & Type & Parameters \\
\hline
III & $(s'+s'', t'+t''+1, s''',t''')$ &
III & $(s', t', s''+s''', t''+t'''+1)$ \\
IV & $\bigl( (1,s'+s'',t'+t''+1), (d_2, s''',t'''), \dots \bigr)$ &
IV & $\bigl( (1,s',t'), (d_2, s''+s''', t''+t'''+1), \dots \bigr)$ \\
\hline
\end{tabular}
\end{center}
\caption{Good double mutations in parametric notation.}
\label{t:mutdblgood}
\end{table}

\begin{lemma} \label{l:IIImutdblgood}
Consider quivers of Type III.
\begin{enumerate}
\renewcommand{\theenumi}{\alph{enumi}}
\item
A quiver of Type III with parameters $(s',t'+1,s'',t'')$ is equivalent
by good double mutations to one of Type III with parameters
$(s',t',s'',t''+1)$.

\item
A quiver of Type III with parameters $(s',t',s'',t'')$ can be
transformed using good mutations and good double mutations to one of
Type III with parameters $(s'+s'',t'+t'',0,0)$.
\end{enumerate}
\end{lemma}
\begin{proof}
\begin{enumerate}
\renewcommand{\theenumi}{\alph{enumi}}
\item
Follows from the first row in Table~\ref{t:mutdblgood}.

\item
Follows from the first part together with Lemma~\ref{l:IIImutgood}.
\end{enumerate}
\end{proof}

\begin{lemma} \label{l:IVgrouptriangle}
Triangles in a rooted quiver of type $A$ attached to a spike in a group
of consecutive spikes in a quiver of Type IV can be moved by good
double mutations to a rooted quiver attached to any spike in that
group.
\end{lemma}
\begin{proof}
It suffices to show that the two quivers with parameters
$$
\bigl((1,s_1,t_1), (d_2, s_2, t_2+1), \dots \bigr) \mbox{~~~and~~~}
\bigl((1,s_1,t_1+1), (d_2, s_2, t_2), \dots \bigr)
$$
are equivalent by good
double mutation. Indeed, setting $(s'',t'')=(0,0)$ in the second row of
Table~\ref{t:mutdblgood} shows this equivalence.
\end{proof}

\begin{remark}
A careful look at Tables~\ref{t:mutgood} and~\ref{t:mutdblgood} shows
that one can regard Type III quivers with parameters $(s',t',s'',t'')$
as ``formal'' Type IV quivers with parameters $\bigl( (1,s',t'),
(1,s'',t'') \bigr)$. Indeed, the good mutation moves III.1 and III.2 in
Table~\ref{t:mutgood} then become just specific cases of moves IV.2b
and IV.2a, respectively, and the first row in Table~\ref{t:mutdblgood}
becomes a specific instance of the second.
\end{remark}

\subsection{Proof of Theorem \protect{\ref{t:stdder}}}
\label{sec:stdder}

In this section we prove Theorem~\ref{t:stdder}. Namely, given a quiver
$Q$ of a cluster-tilted algebra of Dynkin type $D$, we show how to find
a quiver in one of the standard forms of Theorem~\ref{t:stdder} whose
cluster-tilted algebra is derived equivalent to that of $Q$.

Indeed, if $Q$ is of Type I or II, then by Lemma~\ref{l:IImutgood} we
can transform it by good mutations to a quiver in the classes (a) or
(b) in Theorem~\ref{t:stdder}, thus proving the derived equivalence for
this case. Similarly, if $Q$ is of Type III, then by
Lemma~\ref{l:IIImutgood} and Lemma~\ref{l:IIImutdblgood} the
corresponding cluster-tilted algebra is derived equivalent to one in
class (c) of the theorem.

Let $Q$ be a quiver of Type IV. If it is a cycle without any spikes, we
distinguish two cases. If the number of vertices is even, then by
Lemma~\ref{l:selfinj} the corresponding cluster-tilted algebra is
derived equivalent to another one in Type IV with spikes. Otherwise, it
gives rise to the standard form $(\mathrm{d}_1)$.

If $Q$ is of Type IV with some spikes, let $\bigl( (d_1, s_1, t_1),
(d_2, s_2, t_2), \dots, (d_r, s_r, t_r) \bigr)$ be its parameters. By
iteratively applying the good mutations IV.1a (or IV.1b) we can
repeatedly shorten by $2$ all the distances $d_i \geq 4$ until they
become $2$ or $3$. By applying the good mutations IV.2a (or IV.2b) we
can shorten further any distance $2$ to a distance of $1$. Thus we get
a parameter sequence where all distances $d_i$ are either $1$ or $3$.

If all the distances are $1$, we are in class $(\mathrm{d}_2)$ when
they sum up to at least $3$ or in class (c) otherwise. The latter case
can be dealt with by Lemma~\ref{l:IIImutgood} and
Lemma~\ref{l:IIImutdblgood} yielding the required standard form,
whereas for the former we note that by Lemma~\ref{l:IVgroupline} and
Lemma~\ref{l:IVgrouptriangle} we can successively move all the lines
and triangles of the attached rooted quivers of type $A$ and
concentrate them on a single spike, yielding the standard form of
$(\mathrm{d}_2)$.

Otherwise, when there is at least one distance of $3$, we observe the
following. Consider a group of consecutive spikes. By
Lemma~\ref{l:IVgroupline} and Lemma~\ref{l:IVgrouptriangle}, inside
such a group one can always concentrate the attached type $A$ quivers
at one of the spikes in the group, thus creating a free spike at the
end of the group. By Lemma~\ref{l:IVfreespike} this free spike can then
be moved to the beginning of the next group. In this way, one can move
all spikes of the group except one to the next group, thus creating a
single spike with some rooted quiver of type $A$ attached.

Continuing in this way, we can eventually merge all groups of at least
two consecutive spikes into one large group, with all the other spikes
being single spikes. In other words, the sequence of distances will
look like $(1,1,\dots,1,3,3,\dots,3)$. At this large group of
consecutive spikes, one can concentrate the rooted quivers of type $A$
at the last spike, yielding exactly the standard form occurring in
$(\mathrm{d}_3)$ of Theorem~\ref{t:stdder}.

\smallskip

To complete the proof of Theorem~\ref{t:stdder}, we show how to
distinguish among standard forms which are not of the class
$(\mathrm{d}_3)$. First, observe that when the number $n$ of vertices
is odd, the form in the class $(\mathrm{d}_1)$ corresponds to the
unique self-injective cluster-tilted algebra with $n$
vertices~\cite{Ringel}, hence it is distinguished by the fact that
being self-injective is invariant under derived
equivalence~\cite{Al-Nofayee}.

The standard forms in all other classes (except $(\mathrm{d}_3)$) can
be distinguished by the determinants of their Cartan matrices. Indeed,
according to Theorem \ref{thm-det-typeD},
these are given in the list below as
\begin{align*}
\begin{array}{c}
1 \\ (a)
\end{array}
& &
\begin{array}{c}
2^{t+1} \\ (b)
\end{array}
& &
\begin{array}{c}
3 \cdot 2^t \\ (c)
\end{array}
& &
\begin{array}{c}
(2b-1) \cdot 2^t \\ (d_2)
\end{array}
\end{align*}
from which it is clear how to distinguish among the standard forms.

\subsection{Algorithm for good mutation equivalence}
\label{sec:alg:goodmut}
In this section we prove Theorem~\ref{t:alggood} and
Theorem~\ref{t:stdgood}. We first observe that by
Lemma~\ref{l:IImutgood} the set of quivers of Types I or II is closed
under good mutations, and moreover that lemma completely characterizes
good mutation equivalence among these quivers in terms of their
parameters, leading to the classes (a) and (b) in
Theorem~\ref{t:stdgood}. We observe also that the cyclic quiver of Type
IV without any spikes does not admit any good mutations, thus it falls
into a separate equivalence class $(\mathrm{d}_1)$ in that theorem.
Therefore we are left to deal only with quivers of Types III and IV
(with spikes). Before describing the algorithm, we introduce a few
notations.

\begin{notat}
Given $r \geq 1$ and a non-empty subset $\phi \neq \cI \subseteq \{1,
2, \dots, r\}$, we define the following two partitions of the set $\{1,
2, \dots, r\}$. Write the elements of $\cI$ in an increasing order $1
\leq i_1 < i_2 < \dots < i_l \leq r$ for $l = |\cI|$, and define the
intervals
\begin{align*}
i_1^+ &= \{i_1, i_1 + 1, \dots, i_2-1\} &
i_1^- &= \{i_l+1, \dots, r, 1, \dots, i_1\} \\
i_2^+ &= \{i_2, i_2 + 1, \dots, i_3-1\} &
i_2^- &= \{i_1+1, \dots, i_2-1, i_2\} \\
& \dots & & \dots \\
i_l^+ &= \{i_l, \dots, r, 1, \dots, i_1-1\} & i_l^- &= \{i_{l-1}+1,
\dots, i_l-1, i_l\}
\end{align*}
We call the partition $i_1^+ \cup i_2^+ \cup \dots \cup i_l^+$ the
\emph{positive} partition defined by $\cI$. Similarly, we call $i_1^-
\cup i_2^- \cup \dots \cup i_l^-$ the \emph{negative} partition defined
by $\cI$.
\end{notat}

\begin{notat}
We partition the set of positive integers as $N_1 \cup N_2 \cup N_3$,
where
\begin{align*}
N_1 = \{1\} &,& N_2 = \{n \geq 2 \,:\, \text{$n$ is even}\} &,& N_3 =
\left\{n \geq 3 \,:\, \text{$n$ is odd}\right\}.
\end{align*}
\end{notat}

\begin{notat}
Given a sequence $\bigl( (d_1, s_1, t_1), (d_2, s_2, t_2), \dots, (d_r,
s_r, t_r) \bigr)$ of triples of non-negative integers and a subset $I$
of $\{1, \dots, r\}$, we define the quantities
\begin{align*}
a(I) &= |\{i \in I \,:\, d_i \in N_3 \}|, \\
b(I) &= |\{i \in I \,:\, d_i \in N_1 \}|
     + \sum_{i \in I \,:\, d_i \in N_2} \frac{d_i}{2}
     + \sum_{i \in I \,:\, d_i \in N_3} \frac{d_i-3}{2}, \\
s(I) &= |\{i \in I \,:\, d_i \in N_2 \}| + \sum_{i \in I} s_i .
\end{align*}
\end{notat}

\begin{alg}[Good mutation class] \label{alg:goodmut}
Given a non-empty sequence of triples of non-negative integers
\[
\bigl( (d_1, s_1, t_1), (d_2, s_2, t_2), \dots, (d_r, s_r, t_r) \bigr)
\]
such that
\begin{itemize}
\item
$d_i \geq 1$ and $s_i, t_i \geq 0$ for all $1 \leq i \leq r$,
\item
$d_1 + d_2 + \dots + d_r \geq 2$ and $(d_1, \dots, d_r) \neq (2)$,
\end{itemize}
parameterizing a quiver of Type III or IV (with spikes), we output its
class (c), $(\mathrm{d}_{2,1})$, $(\mathrm{d}_{2,2})$,
$(\mathrm{d}_{3,1})$ or $(\mathrm{d}_{3,2})$ and the parameters in that
class as specified in Theorem~\ref{t:stdgood} by performing the
following operations.

\begin{enumerate}
\renewcommand{\labelenumi}{\theenumi.}
\item
Compute the subsets
\begin{align*}
\cI_D = \{1 \leq i \leq r \,:\, d_i \in N_3\} &, & \cI_T = \{ 1 \leq i
\leq r \,:\, t_i > 0\} .
\end{align*}

\item
If $\cI_D = \phi$ and $\cI_T = \phi$, set $b, S$ as
\begin{align*}
b = b(\{1, 2, \dots, r\}) &,& S = s(\{1, 2, \dots, r\}) .
\end{align*}
If $b \geq 3$, we are in class $(\mathrm{d}_{2,1})$, otherwise we are
in class (c) with parameters $(S,0,0,0)$.

\item
If $\cI_D = \phi$ and $\cI_T \neq \phi$, enumerate the elements of
$\cI_T$ in increasing order as $\cI_T = \{i_1 < i_2 < \dots < i_l\}$
with $l=|\cI_T|$, and set $b_j, T_j$ for $1 \leq j \leq l$ and $S$ as
\begin{align*}
b_j = b(i_j^+) &,& T_j = t_{i_j} &,& S = s(\{1, 2, \dots, r\}) .
\end{align*}
If $b_1 + \dots + b_l \geq 3$, we are in class $(\mathrm{d}_{2,2})$.
Otherwise, we are in class (c) with parameters $(S+T_1,0,0,0)$ if $l=1$
or $(S+T_1,0,T_2,0)$ if $l=2$.

\item
If $\cI_D \neq \phi$ and $\cI_T = \phi$, enumerate the elements of
$\cI_D$ in increasing order as $\cI_D = \{i_{1,1} < i_{1,2} < \dots <
i_{1,a}\}$ and set $a, b$ and $S_1, \dots, S_a$ as
\begin{align*}
a = a(\{1,2,\dots, r\}) =|\cI_D| &,& b = b(\{1,2,\dots, r\}) &,& S_j =
s(i_{1,j}^-) .
\end{align*}
We are in class $(\mathrm{d}_{3,1})$.

\item
If $\cI_D \neq \phi$ and $\cI_T \neq \phi$, enumerate the elements of
$\cI_T$ in increasing order as $\cI_T = \{i_1 < i_2 < \dots < i_l\}$
with $l=|\cI_T|$. For any $1 \leq j \leq l$,
\begin{itemize}
\item
Enumerate the elements of $i_j^+ \cap \cI_D$ (where the positive
partition is taken with respect to the subset $\cI_T$) in the order
they appear within the interval $i_j^+$ as $i_{j,1} < i_{j,2} < \dots <
i_{j, a(i_j^+)}$.

\item
Set $a_j, b_j, T_j$ and $S_{j,1}, \dots, S_{j,a_j}$ as
\[
a_j = a(i_j^+) ,\quad b_j = b(i_j^+) ,\quad T_j = t_{i_j} ,\quad
(S_{j,1}, \dots, S_{j,a_j}) = \bigl( s(i_{j,1}^-), s(i_{j,2}^-), \dots,
s(i_{j,a(i_j^+)}^-)\bigr)
\]
with the positive partition taken with respect to $\cI_T$ and the
negative one with respect to $\cI_D$.
\end{itemize}
We are in class $(\mathrm{d}_{3,2})$.
\end{enumerate}
\end{alg}

\begin{notat}
We call two sequences $(v_0,v_1,\dots,v_{m-1})$ and
$(w_0,w_1,\dots,w_{m-1})$ \emph{cyclic equivalent} if there is some $0
\leq j \leq m$ such that $w_i = v_{(i+j) \bmod{m}}$ for all $0 \leq i <
m$.
\end{notat}

\begin{dfn}
We define the space $\cS$ of \emph{good mutation parameters} as a
disjoint union of the following five sets. We also define an
equivalence relation $\sim$ on $\cS$ inside each set, and agree that
elements from different sets are inequivalent.
\begin{enumerate}
\item [(c)]
Triples $(T_1, T_2, S)$ of non-negative integers. $(T_1,T_2,S) \sim
(T'_1, T'_2, S')$ if and only if $S=S'$ and $(T_1,T_2)$, $(T'_1,T'_2)$
are cyclic equivalent.

\item [$(\mathrm{d}_{2,1})$]
Pairs $(b,S)$ with $b \geq 3$ and $S \geq 0$. $(b,S) \sim (b',S')$ if
and only if $b=b'$ and $S=S'$.

\item [$(\mathrm{d}_{2,2})$]
Pairs
\[
\Bigl( \bigl( (b_1, T_1), (b_2, T_2), \dots, (b_l, T_l) \bigr), S
\Bigr)
\]
for some $l \geq 1$, where the numbers $b_j, T_j$ are positive, $b_1 +
\dots + b_l \geq 3$ and $S \geq 0$. Two such pairs are equivalent if
and only if $S=S'$ and $\bigl( (b_1, T_1), (b_2, T_2), \dots, (b_l,
T_l) \bigr)$, $\bigl( (b'_1, T'_1), (b'_2, T'_2), \dots, (b'_l, T'_l)
\bigr)$ are cyclic equivalent.

\item [$(\mathrm{d}_{3,1})$]
Pairs $\bigl( b, (S_1, \dots, S_a) \bigr)$ where $b \geq 0$ and $(S_1,
\dots, S_a)$ is a sequence of $a > 0$ non-negative integers. Two such
pairs are equivalent if and only if $b=b'$ and $(S_1, \dots, S_a)$,
$(S'_1, \dots, S'_a)$ are cyclic equivalent.

\item [$(\mathrm{d}_{3,2})$]
Sequences
\[
\Bigl( \bigl(b_1, (S_{1,1},\dots,S_{1,a_1}), T_1 \bigr), \bigl(b_2,
(S_{2,1},\dots,S_{2,a_2}), T_2 \bigr), \dots, \bigl(b_l,
(S_{l,1},\dots,S_{l,a_l}), T_l \bigr) \Bigr),
\]
of any length $l \geq 1$, where for any $1 \leq j \leq l$ the numbers
$a_j$, $b_j$ are non-negative integers not both zero,
$(S_{j,1},\dots,S_{j,a_j})$ is a (possibly empty) sequence of $a_j$
non-negative integers and $T_j$ is a positive integer. The relation
$\sim$ is just cyclic equivalence.
\end{enumerate}
\end{dfn}

\begin{remark}
It is easy to decide whether two good mutation parameters are
equivalent or not, because this involves only checking for cyclic
equivalence.
\end{remark}

Algorithm~\ref{alg:goodmut} in fact computes a map $\Sigma : \cQ \to
\cS$ from the set $\cQ$ of all quivers of Types III or IV (with spikes)
to the set $\cS$ of good mutation parameters. On the other hand, the
standard forms stated in Theorem~\ref{t:stdgood} can be seen as a map
$Q : \cS \to \cQ$. We also have two natural equivalence relations on
these sets: the equivalence relation $\sim$ defined on $\cS$ via cyclic
equivalence, and the good mutation equivalence on $\cQ$, which we also
denote by $\sim$. The correctness of the algorithm is guaranteed by the
following proposition whose proof is long and tedious, building
on arguments similar to those presented in Section~\ref{sec:stdder};
the complete proof can be found in the first author's thesis
\cite[Prop. 4.2.47]{Bastian-thesis}.

\begin{prop}
Let $q, q' \in \cQ$ and $\sigma, \sigma' \in \cS$.
\begin{enumerate}
\renewcommand{\theenumi}{\alph{enumi}}
\item
If $q \sim q'$, then $\Sigma(q) \sim \Sigma(q')$.

\item
If $\sigma \sim \sigma'$ then $Q(\sigma) \sim Q(\sigma')$.

\item
If $\sigma \in \cS$ then $\Sigma(Q(\sigma)) = \sigma$. In other words,
applying Algorithm~\ref{alg:goodmut} to a standard form as in
Theorem~\ref{t:stdgood} recovers the parameters of that form.

\item
If $q \in \cQ$ then there exists $\sigma \in \cS$ such that $q \sim
Q(\sigma)$. In other words, a quiver can be transformed by good
mutations to a quiver in standard form as in Theorem~\ref{t:stdgood}.
\end{enumerate}
\end{prop}

This completes the proof of Theorem~\ref{t:alggood} and
Theorem~\ref{t:stdgood}.

\subsection{Algorithm for computing standard forms}
\label{sec:algder}

We now present an algorithm to compute the standard
form for a cluster-tilted algebra of type $D$ given
in parametric notation.
We keep the notations of the previous section.

\begin{notat}
Given a sequence $\bigl( (d_1, s_1, t_1), (d_2, s_2, t_2), \dots, (d_r,
s_r, t_r) \bigr)$ of triples of non-negative integers and a subset $I$
of $\{1, \dots, r\}$, we define
$
t(I) = \sum_{i \in I} t_i .
$
\end{notat}

\begin{alg}[Standard form]
Given a non-empty sequence of triples of non-negative integers
\[
\bigl( (d_1, s_1, t_1), (d_2, s_2, t_2), \dots, (d_r, s_r, t_r) \bigr)
\]
such that
\begin{itemize}
\item
$d_i \geq 1$ and $s_i, t_i \geq 0$ for all $1 \leq i \leq r$,
\item
$d_1 + d_2 + \dots + d_r \geq 2$ and $(d_1, \dots, d_r) \neq (2)$,
\end{itemize}
parameterizing a quiver of Type III or IV (with spikes), we output its
class (c), $(\mathrm{d}_2)$ or $(\mathrm{d}_3)$ and the parameters in
that class as specified in Theorem~\ref{t:stdder} by performing the
following operations.

\begin{enumerate}
\renewcommand{\labelenumi}{\theenumi.}
\item
Compute the subset $\cI_D = \{1 \leq i \leq r \,:\, d_i \in N_3\}$.

\item
If $\cI_D = \phi$, set $b$, $s$ and $t$ as
\begin{align*}
b = b(\{1, 2, \dots, r\}) &,& s = s(\{1, 2, \dots, r\}) &,& t = t(\{1,
2, \dots, r\}) .
\end{align*}
If $b \geq 3$, we are in class $(\mathrm{d}_2)$, otherwise we are in
class (c) with parameters $(s,t)$.

\item
If $\cI_D \neq \phi$, enumerate the elements of $\cI_D$ in increasing
order as $\cI_D = \{i_1 < i_2 < \dots < i_k\}$ and set $k, b$, $s_1,
\dots, s_k$ and $t_1, \dots, t_k$ as
\begin{align*}
k = |\cI_D| &,& b = b(\{1,2,\dots, r\}) &,& s_j = s(i_j^-) &,& t_j =
t(i_j^-).
\end{align*}
We are in class $(\mathrm{d}_3)$.
\end{enumerate}
\end{alg}

The proof of correctness of this algorithm is essentially
contained in Section \ref{sec:stdder}.

\appendix

\section{Proofs of Cartan determinants} \label{appendix:det}

As a consequence of Proposition \ref{prop-asymmetry} the determinant of
the Cartan matrix is invariant under derived equivalences. The aim of
this section is to provide a proof of the formulae for the determinants of the Cartan
matrices of all cluster-tilted algebras of Dynkin type $D$ as given
in Theorem \ref{thm-det-typeD}. Recall that
the quivers of the cluster-tilted algebras of type $D$ are given by the
quivers of Types I, II, III and IV described in
Section~\ref{sec:ctaAD}.

As these quivers are defined by gluing of rooted quivers of type $A$,
it is useful to also have formulae for cluster-tilted algebras of
Dynkin type $A$. Since cluster-tilted algebras of type $A$ are gentle,
the Cartan determinants can be obtained as a special case of
\cite{Holm} where formulae for the Cartan determinants of arbitrary
gentle algebras are given; for a simplified proof for the special case
of cluster-tilted algebras of type $A$ see also \cite{Buan-Vatne}.

\begin{prop} \label{prop-det-typeA}
Let $Q$ be a quiver mutation equivalent to a Dynkin quiver of type $A$.
Then the Cartan matrix of the cluster-tilted algebra corresponding to $Q$ has
determinant $\det C_Q = 2^{t(Q)}.$
\end{prop}

\smallskip

For proving Theorem \ref{thm-det-typeD} we shall first show a useful
reduction lemma.
We need the following notation: if $Q$ is a quiver and $V$ a set
of vertices in $Q$, then $Q\setminus V$ is the quiver obtained from $Q$
by removing all vertices in $V$ from $Q$ and all arrows attached to them.

\begin{lemma} \label{lemma-det-shrinking}
Let $Q$ be a quiver in the mutation class of a quiver of Dynkin type
$D$, i.e. $Q$ is of one of the Types I,II,III,IV given in Section
\ref{sec:ctaAD}.
\begin{enumerate}
\item[{(i)}] Suppose $Q$ contains a vertex $a$ of valency 1. Then
$\det C_Q = \det C_{Q\setminus\{a\}}$.
\item[{(ii)}] Suppose that $Q$ contains an oriented triangle with vertices
$a,b,c$ (in this order, i.e. there is an arrow from $b$ to $c$ etc)
where $a$ and $b$ have valency 2 in $Q$ and where
the quiver $Q'=Q\setminus \{a,b\}$ is mutation equivalent to a quiver
of Dynkin type $A$ or $D$.
Then $\det C_Q = 2\cdot \det C_{Q\setminus\{a,b\}}= 2\cdot \det C_{Q'}$.
\end{enumerate}
\end{lemma}

\begin{proof}
\begin{enumerate}
\item[{(i)}] Since taking transposes does not change the determinant
we can assume that $a$ is a sink. Then the Cartan matrix of the
cluster-tilted algebra corresponding to $Q$ has the form
$$C_Q = \left( \begin{array}{c|ccc}
1 & 0 & \ldots & 0 \\
\hline
\ast & & & \\
\vdots & & C_{Q\setminus\{a\}} & \\
\ast & & & \\
\end{array} \right)
$$
from which the desired formula directly follows by
Laplace expansion.
\item[{(ii)}] In the cluster-tilted algebra corresponding to $Q$, every
product of two consecutive arrows in the triangle $a,b,c$ is zero.
Moreover, in the quiver $Q\setminus\{a\}$
there is a one-one correspondence between non-zero paths
starting in $c$ and non-zero paths starting in $b$ by extending
any of the former paths by the arrow from $b$ to $c$ (in fact, by the
unique relations in the cluster-tilted algebra all these extensions
remain non-zero). Therefore, the Cartan matrix of the cluster-tilted
algebra corresponding to $Q$ has the form
$$C_Q = \left(
\begin{array}{ccccccc}
1 & 1 & 0 & 0 & \ldots & \ldots & 0 \\
0 & 1 & 1 & \ast_1 & \ldots & \ldots & \ast_n \\
1 & 0 & 1 & \ast_1 & \ldots & \ldots & \ast_n \\
\ast & 0 & \ast & & & & \\
\vdots & \vdots & \vdots & & C_{Q\setminus\{a,b,c\}}& & \\
\vdots & \vdots & \vdots & &  & & \\
\ast & 0 & \ast & & & & \\
\end{array}
\right)
$$
where the first three rows are labelled by $a$, $b$ and $c$, respectively.
The entries marked by $\ast_i$ are really the same in the rows for $b$ and
$c$ because of the one-one correspondence just mentioned.
Moreover, note that in the (first) row for $a$ and in the (second) column for $b$
we have 0's except the two 1's indicated
because of the zero-relations in the triangle with vertices $a,b,c$.

Denote by $r_v$ the row in the above matrix corresponding to the vertex $v$.
We now perform an elementary row operation, namely replace the first row $r_a$
by $r_a - r_b + r_c$. Then we get
$$\det C_Q =
\det \left(
\begin{array}{ccccccc}
2 & 0 & 0 & 0 & \ldots & \ldots & 0 \\
0 & 1 & 1 & \ast_1 & \ldots & \ldots & \ast_n \\
1 & 0 & 1 & \ast_1 & \ldots & \ldots & \ast_n \\
\ast & 0 & \ast & & & & \\
\vdots & \vdots & \vdots & & C_{Q\setminus\{a,b,c\}}& & \\
\vdots & \vdots & \vdots & &  & & \\
\ast & 0 & \ast & & & & \\
\end{array}
\right)
= 2\cdot \det C_{Q\setminus\{a,b\}}
$$
where the last equality follows directly by Laplace expansion
(with respect to the row of $a$ and then the column of $b$).
\end{enumerate}
\end{proof}

We will also need the following three lemmas dealing with skeleta of
Type IV, i.e.\ the rooted quivers of type $A$ consist of just one
vertex.

\begin{lemma} \label{lemma-det-no_spikes}
Let $Q$ be a quiver of Type IV which contains no spikes at all, i.e.\
it is just an oriented cycle of length $k \geq 3$. Then $\det C_Q =
k-1$.
\end{lemma}

\begin{lemma} \label{lemma-det-all_spikes}
Let $Q$ be a quiver of Type IV with parameter sequence $\bigl((1,0,0),
(1,0,0), \dots, (1,0,0)\bigr)$ of length $k \geq 3$, in other words, it
is an oriented cycle of length $k$ with all spikes present. Then $\det
C_Q = 2k-1$.
\end{lemma}

\begin{lemma} \label{lemma-det-not_all_spikes}
Let $Q$ be a quiver of Type IV with oriented cycle of length $k\ge 3$
and not all spikes are present. Let $c(Q)$ be the number of vertices on
the oriented cycle which are part of two (consecutive) spikes. Then
$\det C_Q = k+c(Q)-1$.
\end{lemma}

\begin{proof}[Proof of Lemma~\protect{\ref{lemma-det-no_spikes}}]
We have
\[
\det C_Q  =  \det \left( \begin{array}{ccccc}
1 & \ldots & \ldots & 1 & 0 \\
0 & 1 & \ldots & \ldots & 1 \\
1 & \ddots & \ddots &  & \vdots \\
\vdots &  & \ddots & \ddots & \vdots \\
1 & \ldots & 1 & 0 & 1
\end{array}
\right) =
(-1)^{k-1} \det \left( \begin{array}{cccc}
0 & 1 & \ldots & 1 \\ 1 & 0 & & \vdots \\ \vdots & & \ddots & 1 \\ 1 & \ldots & 1 & 0
\end{array}
\right) = k-1,
\]
where for the last equality
we have used the following formula whose verification
is a standard exercise in linear algebra: for all $a,b\in\mathbb{R}$ we have
\begin{eqnarray} \label{eq_det_ab}
\det \left( \begin{array}{cccc}
b & a & \ldots & a \\ a & b & & \vdots \\ \vdots & & \ddots & a \\ a & \ldots & a & b
\end{array}
\right) & = & (b-a)^{k-1}(b+(k-1)a).
\end{eqnarray}
\end{proof}

\begin{proof}[Proof of Lemma~\protect{\ref{lemma-det-all_spikes}}]
By Lemma~\ref{l:selfinj}, the cluster-tilted algebra $\gL_Q$ is derived
equivalent to the one corresponding to the cycle of length $2k$. Since
the determinant of the Cartan matrix is invariant under derived
equivalence, the result now follows from
Lemma~\ref{lemma-det-no_spikes}.
\end{proof}

\begin{proof}[Proof of Lemma~\protect{\ref{lemma-det-not_all_spikes}}]

We shall closely look at one group of $t\ge 1$ consecutive spikes in $Q$
and label the vertices as in the following figure

\begin{center}
\includegraphics[scale=0.75]{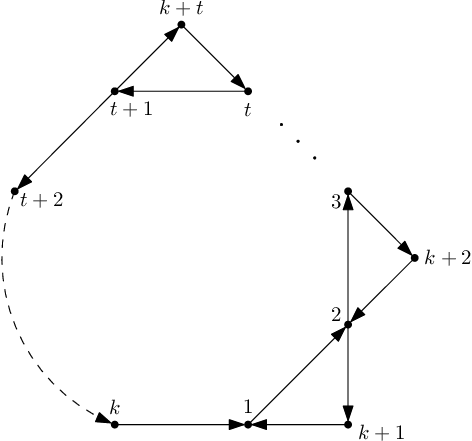}
\end{center}
Then the Cartan matrix has the following shape
$$
C_Q= {\small
\left(
\begin{array}{cccc|cccc||cccc|cccc}
1 & 1 & \ldots & 1 & 1 & \ldots & 1 & 0 & 0 & 0 & \ldots & 0 & 0 & \ldots & \ldots & 0 \\
1 & 1 & \ldots & 1 & 1 & \ldots & \ldots & 1 & 1 & 0 & \ldots & 0 & 0 & \ldots & \ldots & 0 \\
\vdots & \vdots &  & \vdots & \vdots & & & \vdots  & \vdots & \ddots & \ddots &  & \vdots &
 & & \vdots \\
1 & 1 & \ldots & 1 &1 & \ldots & \ldots & 1 & 0 & \ldots & 1 & 0 &0 & \ldots & \ldots & 0\\
\hline
1 & 1 & \ldots & 1 &1 & 1 & \ldots & 1 & 0 & 0 & \ldots & 1 &0 & \ldots & \ldots & 0  \\
\hline
1 & 1 & \ldots & 1 & & & & & 0 & 0 & \ldots & 0 & & \\
1 & 1 & \ldots & 1 & & ? & ? & & 0 & 0 & \ldots & 0 & & ? & ? &  \\
\vdots & \vdots & & \vdots & & ? & ? & & \vdots & \vdots & & \vdots  & & ?& ?  \\
1 & 1 & \ldots & 1 & & & & & 0 & 0 & \ldots & 0  & & &  \\
\hline\hline
1 & 0 & \ldots & 0 & 0 & 0 & \ldots & 0 & 1 & 0 & \ldots & 0 & 0 & 0 & \ldots & 0 \\
0 & 1 &  & \vdots & \vdots & & & \vdots & 1 & \ddots & & & \vdots & & & \vdots \\
\vdots & & \ddots & 0 & \vdots & & & \vdots & 0 &\ddots & \ddots & 0& \vdots & & & \vdots \\
\hline
0 & \ldots & 0 & 1 & 0 &\ldots & \ldots & 0 & 0 & \ldots & 1 & 1 & 0 & 0 &\ldots & 0\\
\hline
0 & \ldots & \ldots & 0 & & & & & 0 & \ldots & \ldots & 0 & & & &  \\
\vdots & & & \vdots & & ? & ? & & \vdots &  &  & \vdots & & ? & ? &  \\
0 & \ldots & \ldots & 0 & & & & & 0 & \ldots & \ldots & 0 & & & &  \\
\end{array}
\right)
}
$$

The two highlighted rows correspond to the vertices $t+1$ and
$k+t$, respectively.
For each vertex $j$ let $r_j$ be the row of $C_Q$ corresponding to $j$.

Note that the column $k+t$ of $C_Q$ has only two non-zero entries, namely in
rows $t+1$ and $k+t$. We first replace row $r_{t+1}$ by $r_{t+1}-r_{k+t}+r_t-r_1$
(in case $t=1$ this indeed means just $r_{t+1}-r_{k+t}$).
Then column $k+t$ has only one non-zero entry, namely on the diagonal; Laplace
expansion along this column yields a new matrix $\tilde{C}$.

We consider the $(t+1)$-st row in this new matrix which has the form
$$\left( \begin{array}{ccccc|ccccc||cccc|cccc}
1 & 1 & \ldots & 1 & 0 &1  & 1 & \ldots & 1 & N & 0 & 0 & \ldots & 0 &0 & \ldots & \ldots & 0  \\
\end{array}
\right)
$$
where the number $N$ at position $(t+1,k)$
is equal to $0$ if $t=1$ and equal to $2$ if $t>1$.

In case $t=1$ we see that $\tilde{C}$ is equal to the Cartan matrix
of the cluster-tilted algebra corresponding to the quiver
$Q\setminus \{k+t\}$. This means that when computing the Cartan determinant
we can remove isolated spikes, i.e. spikes which are not neighboring
any other spike.

In this case $t=1$ the statement of the theorem
follows immediately by induction on the number of spikes
(with the case of no spikes treated earlier as base of the induction).
In fact,
removing the isolated spike does not change the determinant (as we have just
seen), and also
the formula given in the theorem is not affected by removing an isolated
spike.

Let us turn to the more complicated case $t>1$ (which we shall also
show by induction).
If $t>1$, the matrix $\tilde{C}$ is equal to the Cartan matrix of
$Q\setminus \{k+t\}$, except for the $(t+1,k)$-entry which is $2$ in $\tilde{C}$,
but $1$ in $C_{Q\setminus\{k+t\}}$.

To compare the determinants in this case we use the following easy observation. Let
$C=(c_{ij})$ and $\tilde{C}=(\tilde{c}_{ij})$ be two matrices which only
differ at the $(m,n)$-entry. Then
$$\det \tilde{C} - \det C = (-1)^{m+n} (\tilde{c}_{mn}-c_{mn})
C_{mn}$$
where $C_{mn}$ is the matrix obtained from $C$ (or $\tilde{C}$) by removing
row $m$ and column $n$.

Applied to our situation we get
$$\det \tilde{C} -\det C_{Q\setminus\{k+t\}} = (-1)^{t+1+k}(2-1)\det C_{t+1,k}
= (-1)^{t+1+k}\det C_{t+1,k}.
$$
Since $\det \tilde{C}=\det C_Q$ we can rephrase this to get
\begin{eqnarray} \label{det_CQ}
\det C_Q = \det C_{Q\setminus\{k+t\}} + (-1)^{t+1+k}\det C_{t+1,k}.
\end{eqnarray}
By induction on the number of spikes of $Q$ (with the case of no spikes
treated earlier as base of the induction) we can deduce that
$\det C_{Q\setminus\{k+t\}} = k+c(Q)-2$ and hence
$$\det C_Q = k+c(Q)-2 + (-1)^{t+1+k}\det C_{t+1,k}.
$$
For proving the assertion of the theorem we therefore have to show that
$(-1)^{t+1+k}\det C_{t+1,k} = 1.$

We keep the labelling of the rows and columns also for $C_{t+1,k}$
(i.e. there is no row with label $t+1$ or $k+t$, and no column
with label $k$ or $k+t$).

For convenience, the vertices $1,\ldots,k$ on the cycle will be called
cycle vertices and
the remaining vertices will be called outer vertices in the sequel.

Note that in $C_{t+1,k}$ we have $0$'s on the diagonal in all rows
indexed by cycle vertices which have no spike attached. More precisely,
the rows corresponding to cycle vertices are of the form
$$\left( \begin{array}{ccc|ccc||ccccccc|ccc}
1 & \ldots & 1   &1  & \ldots & 1    & 0 & \ldots & 0 & 1 & 0 & \ldots & 0 &
0 & \ldots & 0  \\
\end{array}
\right)
$$
if the vertex has a spike attached, and
$$\left( \begin{array}{cccc|ccc||ccc|ccc}
1 & \ldots & 1   &0 & 1 & \ldots & 1    & 0 & \ldots & 0 & 0 & \ldots & 0 \\
\end{array}
\right)
$$
if there is no spike attached to the vertex.

For each cycle vertex $j\neq 1$ with no spike attached we perform the elementary
row operation replacing $r_j$ by $r_j-r_1$; this gives a unit vector with $-1$
on the diagonal. Laplace expansion along all these rows removes from $C_{t+1,k}$
all rows and columns corresponding to cycle vertices (not equal to vertex $1$)
with no spikes attached;
for the determinant we thus get a sign $(-1)^{k-s(Q)-1}$ where $s(Q)$ is the total
number of spikes of $Q$.

The matrix obtained after this removal process has rows indexed by vertex $1$,
the cycle vertices with spikes attached except vertex $t+1$, and
the outer vertices except vertex $k+t$. Moreover it has the form
$${\small
\left(
\begin{array}{ccccccc|cccc}
1 & 1 & \ldots & \ldots & \ldots & 1 & 1 & 0 & \ldots & \ldots & 0 \\
1 & 1 & \ldots & \ldots & \ldots & 1 & 1 & 1 & 0 & \ldots & 0 \\
\vdots & \vdots &  &  &  & \vdots & \vdots & 0 & 1 & \ddots & \vdots \\
1 & 1 & \ldots & \ldots & \ldots & 1 & 1 & \vdots & \ddots & \ddots  & 0 \\
1 & 1 & \ldots & \ldots & \ldots & 1 & 1 & 0 & \ldots & 0  & 1 \\
\hline
1 & & & & & & & & & & \\
 & \ddots & & & & & & &  & & \\
 &  & 1 & & & & & & ?& ?& \\
 &  & & 0 & 1& & & &   ?& ?& \\
&  &  & & \ddots & \ddots & & & & & \\
&  &  &  &  & 0 & 1& &  & & \\
\end{array}
\right)
}
$$
where in the lower left block the crucial $0$ on the main diagonal occurs
in the column labelled by vertex $t$ (because the row indexed by $k+t$
has been removed).

Now for each cycle vertex $j$ with a spike attached we replace the
row $r_j$ by $r_j-r_1$, giving a unit vector. Consecutive Laplace expansion
along these rows removes all columns corresponding to outer vertices
(and all rows corresponding to cycle vertices with spikes attached).
For each of these Laplace expansions we get a sign $(-1)^{s(Q)+1}$ and
there are $s(Q)-1$ such expansions in total, giving an overall sign
of $(-1)^{s(Q)^2-1} = (-1)^{s(Q)-1}$.

We are left with a $s(Q)\times s(Q)$-matrix of the form
$${\small
\left( \begin{array}{ccccccc}
1 & 1 & \ldots & 1 & 1 & \ldots & 1 \\
1 & 0 & & & & & \\
& \ddots & \ddots & & & & \\
& & 1 & 0 & & & \\
& & & 0 & 1 & & \\
& & & & \ddots & \ddots & \\
 & & & & & 0 & 1
\end{array}
\right)
}
$$
whose determinant is $(-1)^{t+1}$
(use Laplace expansion along column $t$).
Summarizing our arguments we get
$$\det C_{t+1,k} = (-1)^{k-s(Q)-1}\cdot (-1)^{s(Q)-1}\cdot (-1)^{t+1}
= (-1)^{k+t-1}.$$
Substituting this into equation (\ref{det_CQ}) we get for the Cartan
determinant of $Q$ the following
\begin{eqnarray*}
\det C_Q & =& k+c(Q)-2 + (-1)^{t+1+k}\det C_{t+1,k} \\
& = & k+c(Q)-2 + (-1)^{t+1+k}(-1)^{k+t-1} = k+c(Q)-1
\end{eqnarray*}
which is exactly the formula claimed in Theorem \ref{thm-det-typeD}.
\end{proof}

\begin{proof}[Proof of Theorem~\protect{\ref{thm-det-typeD}}]
\begin{enumerate}
\item[{(I)}] Applying part (i) of Lemma \ref{lemma-det-shrinking}
twice gives
$\det C_Q = \det C_{Q\setminus\{a,b\}}$. By definition $Q'=Q\setminus\{a,b\}$
is a quiver of Dynkin type $A$, thus $\det C_{Q'} = 2^{t(Q')}$
by Proposition \ref{prop-det-typeA}.
Clearly, $t(Q)=t(Q')$ for quivers of Type I
and hence
$$ \det C_Q = \det C_{Q\setminus\{a,b\}}=  \det C_{Q'} = 2^{t(Q')}
=2^{t(Q)}.$$

\item[{(II)}] Let $Q$ be a quiver of Type II. By applying
Lemma \ref{lemma-det-shrinking} inductively we can shrink
each of the quivers $Q'$ and $Q''$ (which are of Dynkin type $A$)
to one vertex where for the Cartan determinant of the corresponding
cluster-tilted algebra we get a factor 2 for each triangle we
remove (see part (ii) of Lemma \ref{lemma-det-shrinking}).
Thus we get
\begin{eqnarray} \det C_Q & = & 2^{t(Q')}\cdot 2^{t(Q'')} \cdot \det C_{\tilde{Q}}
\end{eqnarray}
where $\tilde{Q}$ is the quiver with vertices $a,b,c,d$ obtained
after shrinking each of $Q'$ and $Q''$ to one vertex.
Labelling the rows and columns in the order $a,b,c,d$ the cluster-tilted
algebra corresponding to $\tilde{Q}$ has Cartan matrix
$C_{\tilde{Q}} =
{\small
\left(
\begin{array}{cccc}
1 & 0 & 0 & 1 \\ 0 & 1 & 0 & 1 \\
1 & 1 & 1 & 1 \\ 0 & 0 & 1 & 1
\end{array} \right)
}
$
whose determinant is easily computed to be 2.
This gives the desired formula as
$$\det C_Q = 2^{t(Q')}\cdot 2^{t(Q'')} \cdot \det C_{\tilde{Q}}
= 2^{t(Q')+t(Q'')+1} = 2\cdot \det C_{Q'}\det C_{Q''}.
$$

\item[{(III)}] Completely analogous to the previous argument in Type II
we can shrink the subquivers $Q'$ and $Q''$ of any quiver of Type III to
one vertex, ending up with an oriented 4-cycle $\tilde{Q}$.
Labelling the rows and
columns in the order $a,c,b,d$ the cluster-tilted algebra corresponding
to this 4-cycle has Cartan matrix
$C_{\tilde{Q}} =
{\small
\left(
\begin{array}{cccc}
1 & 1 & 1 & 0 \\ 1 & 1 & 0 & 1 \\
0 & 1 & 1 & 1 \\ 1 & 0 & 1 & 1
\end{array} \right)
}
$
which has determinant $3$. As above we then get
$$\det C_Q = 2^{t(Q')}\cdot 2^{t(Q'')} \cdot \det C_{\tilde{Q}}
= 3\cdot 2^{t(Q')+t(Q'')} = 3\cdot \det C_{Q'}\det C_{Q''}.
$$

\item[{(IV)}]
If there are no spikes at all, the result follows from
Lemma~\ref{lemma-det-no_spikes}. Otherwise, by
Lemma~\ref{lemma-det-shrinking} we can again assume that all the rooted
quivers $Q^{(1)}, \ldots, Q^{(r)}$ of type $A$ attached to the spikes
have been shrinked to one vertex, yielding a factor
\[
\prod_{j=1}^r 2^{t(Q^{(j)})} = \prod_{j=1}^r \det C_{Q^{(j)}}
\]
for the Cartan determinant $\det C_Q$. The result then follows from
Lemma~\ref{lemma-det-all_spikes} and
Lemma~\ref{lemma-det-not_all_spikes}.
\end{enumerate}
\end{proof}

\end{document}